\documentclass[12pt]{amsart}

\usepackage{geometry}
\usepackage{amsfonts, adjustbox, amssymb, amscd}
\numberwithin{equation}{section}

\usepackage{bm}
\usepackage{verbatim}
\usepackage{mathrsfs}
\usepackage{graphicx}
\usepackage{tikz-cd}
\usepackage{subcaption}
\usepackage{listings}
\usepackage{subfiles}
\usepackage[toc,page]{appendix}
\usepackage{mathtools}
\usepackage{comment}
\usepackage{enumerate}
\usepackage{enumitem}
\usepackage[all]{xy}
\usepackage{calligra}

\usepackage{graphicx}
\usepackage{appendix}
\usepackage{hyperref}
\hypersetup{
    colorlinks=true,
    citecolor=red,
    linkcolor=blue,
    filecolor=magenta,      
    urlcolor=red,
}
\newcommand{\bb}{\bm{b}}
\newcommand{\BB}{{\bf{B}}}
\newcommand{\Mm}{{\bf{M}}}
\newcommand{\Nn}{{\bf{N}}}

\newcommand{\Pp}{{\bf{P}}}
\newcommand{\NN}{{\bf{N}}}

\newcommand{\Qq}{\mathbb{Q}}

\newcommand{\Rr}{\mathbb{R}}

\newcommand{\Span}{\operatorname{Span}}

\newcommand{\vol}{\operatorname{vol}}
\newcommand{\Center}{\operatorname{center}}

\newcommand{\Exc}{\operatorname{Exc}}

\newcommand{\rk}{\operatorname{rank}}

\newcommand{\ninv}{\operatorname{ninv}}
\newcommand{\inv}{\operatorname{inv}}

\newcommand{\Proj}{{\operatorname{Proj}}}
\newcommand{\can}{{\operatorname{can}}}

\newcommand{\tmld}{{\operatorname{tmld}}}

\DeclareMathOperator{\HHom}{\mathscr{H}\text{\kern -3pt {\calligra\large om}}\,}

\newcommand{\Weil}{\operatorname{Weil}}

\newcommand{\Supp}{\operatorname{Supp}}

\newcommand{\mult}{\operatorname{mult}}

\newcommand{\cont}{\operatorname{cont}}

\newcommand{\Aa}{{\mathfrak{A}}}

\newcommand{\Bb}{{\mathfrak{B}}}
\newcommand{\Ff}{\mathcal{F}}

\newcommand{\MM}{{\mathfrak{M}}}

\newcommand{\Ii}{\Gamma}

\newcommand{\Ll}{{\bf{L}}}


\newcounter{parentnumber}

\newtheorem{thm}{Theorem}[section]
\newtheorem{thmsub}{Theorem}[subsection]
\newtheorem{defnsub}[thmsub]{Definition}
\newtheorem{conj}[thm]{Conjecture}

\newtheorem{lem}[thm]{Lemma}
\newtheorem{prop}[thm]{Proposition}

\newtheorem{alphthm}{Theorem}

\theoremstyle{definition}
\newtheorem{defn}[thm]{Definition}

\theoremstyle{definition}
\newtheorem{rem}[thm]{Remark}

\newtheorem{deflem}[thm]{Definition-Lemma}

\newtheorem{setup}[thm]{Set-up}

\newtheorem{nota}[thm]{Notation}

\theoremstyle{definition}

\begin{document}

\title{On finite generation and boundedness of adjoint foliated structures}
\author{Paolo Cascini}
\author{Jingjun Han}
\author{Jihao Liu}
\author{Fanjun Meng}
\author{Calum Spicer}
\author{Roberto Svaldi}
\author{Lingyao Xie}

\date{\today}

\begin{abstract}
	We prove the existence of good minimal models for any klt algebraically integrable adjoint foliated structure of general type, and that Fano algebraically integrable adjoint foliated structures with total minimal log discrepancies and parameters bounded away from zero form a bounded family. These results serve as the algebraically integrable foliation analogues of the finite generation of the canonical rings proved by Birkar–Cascini–Hacon–M\textsuperscript{c}Kernan, and the Borisov–Alexeev–Borisov conjecture on the boundedness of Fano varieties proved by Birkar, respectively.

As an application, we prove that the ambient variety of any lc Fano algebraically integrable foliation is of Fano type, provided the ambient variety is potentially klt.
\end{abstract}

\address{Department of Mathematics, Imperial College London, 180 Queen’s
Gate, London SW7 2AZ, UK}
\email{p.cascini@imperial.ac.uk}

\address{Shanghai Center for Mathematical Sciences, Fudan University, 2005 Songhu Road, Shanghai, 200438, China}
\email{hanjingjun@fudan.edu.cn}

\address{Department of Mathematics, Peking University, No. 5 Yiheyuan Road, Haidian District, Peking 100871, China}
\email{liujihao@math.pku.edu.cn}

\address{Department of Mathematics, Johns Hopkins University, 3400 N. Charles Street, Baltimore, MD 21218, USA}
\email{fmeng3@jhu.edu}

\address{Department of Mathematics, King’s College London, Strand, London WC2R 2LS, UK}
\email{calum.spicer@kcl.ac.uk}

\address{Dipartimento di Matematica ``F. Enriques'', Universit\`a degli Studi di Milano, Via Saldini 50, 20133 Milano (MI), Italy}
\email{roberto.svaldi@unimi.it}
 
\address{Department of Mathematics, University of California, San Diego, 9500 Gilman Drive \# 0112, La Jolla, CA 92093-0112, USA}
\email{l6xie@ucsd.edu}

\maketitle

\pagestyle{myheadings}\markboth{\hfill P. Cascini, J. Han, J. Liu, F. Meng, C. Spicer, R. Svaldi, and L. Xie \hfill}{\hfill On finite generation and boundedness of adjoint foliated structures\hfill}

\tableofcontents

\section{Introduction}\label{sec:Introduction}
We work over the field of complex numbers $\mathbb{C}$.

In recent years, the study of the birational geometry of foliations has attracted considerable attention due to its deep connections with the geometry of tangent bundles, K\"ahler manifolds, and fibration structures. In particular, the minimal model program for foliations has been extensively developed \cite{McQ08,Bru15,CS20,Spi20,ACSS21,CS21,SS22,CHLX23,CS24,SS23,CHLMSSX24,LMX24,McQ24}.

The birational geometry of foliations differs from the birational geometry of varieties in many striking aspects, but there are two particularly notable divergences.
\begin{enumerate}
    \item {\it Failure of finite generation}: in contrast to the case of varieties, it is known that the canonical ring of a foliation $\mathcal F$ on a smooth variety $X$, $\bigoplus_{m = 0}^{+\infty} H^0(X, \mathcal O_X(mK_{\mathcal F}))$, may fail to be finitely generated, even if $K_{\mathcal F}$ is nef and big, cf. \cite[Theorem IV.2.2]{McQ08}.

    \item {\it Failure of boundedness}: smooth projective varieties $X$ of dimension $d$, with ample $K_X$ and fixed volume $\vol(X) = K_X^d$ are known to be bounded --
    cf. \cite[Theorem 1.1]{HMX18} -- 
    but this is false for foliations, see \cite[Example 4.1]{Pas24}.
    There are many other subtle cases where boundedness statements fail for foliations, especially for Fano foliations.
\end{enumerate}

To address these issues, and other challenges arising in the study of the birational geometry of foliations,  the notion of \emph{adjoint foliated structures} was introduced. An adjoint foliated structure consists of a triple $(X,\Ff,t)$ where $X$ is a normal variety, $\mathcal F$ is a foliation on $X$, and $t \in [0, 1]$. To an adjoint foliated structure we associate the canonical divisor $tK_{\Ff}+(1-t)K_X$. For previous results on adjoint foliated structures, we refer the reader to \cite{PS19,SS23,CHLMSSX24,LW24,CC25,LWX25}.

We expect that the birational geometry of adjoint foliated structures with $t<1$ behaves similarly to that of varieties. 
The main results of our paper confirm this expectation in two very important cases.

Our first main theorem confirms that finite generation holds for algebraically integrable adjoint foliated structures of general type and can be considered as an analogue of \cite[Theorem 1.1]{BCHM10}.

\begin{alphthm}\label{thm: bchm analogue}
    Let $X$ be a smooth projective variety and $\Ff$ an algebraically integrable foliation on $X$ with at worst log canonical singularities. Let $t\in [0,1)$ be a real number such that $K:=tK_{\Ff}+(1-t)K_X$ is big. 

    Then $K$ has a good minimal model. In particular, if $t$ is a rational number, then the canonical ring
    $$R(X,K):=\bigoplus_{m=0}^{+\infty}H^0(X,\mathcal O_X(\lfloor mK\rfloor))$$
    is finitely generated.
\end{alphthm}

As we will explain in greater detail in Section \ref{sec: further discussion gmm}, this theorem has many important implications for the study of the birational geometry of foliations.

Our second main theorem proves that a natural class of algebraically integrable Fano adjoint foliated structures is bounded.  It should be viewed as a version of the Borisov-Alexeev-Borisov conjecture, proved by Birkar \cite[Theorem 1.1]{Bir21}, for adjoint foliated structures.

\begin{alphthm}\label{thm: BAB analogue}
Let $d$ be a positive integer and $\epsilon$ a positive real number. Then the projective varieties $X$ and algebraically integrable foliations $\Ff$ on $X$ such that 
\begin{itemize}
    \item $X$ is $\epsilon$-lc of dimension $d$, 
    \item $\Ff$ is log canonical, and
    \item $-(tK_{\Ff}+(1-t)K_X)$ is ample for some $\epsilon\leq t\leq 1-\epsilon$
\end{itemize}
form a bounded family.
\end{alphthm}

We emphasize that there are no smooth Fano foliations when $\mathcal F \neq T_X$. 
In fact, for $\epsilon >0$, a Fano foliation $\mathcal F \neq T_X$ is never $\epsilon$-lc, cf. \cite[Proposition 5.3]{AD13}. 
It is therefore unclear what the largest natural class of Fano type foliations where boundedness can be expected is, and one of our contributions here is proposing such a class.  
We remark that there are other proposals about what such a class of Fano objects might be, particularly when the ambient variety is smooth. 
For instance, Araujo \cite[5:30]{Ara20} asked whether all Fano foliations on smooth projective varieties of fixed dimension form a bounded family. 
This is true, for example, when the foliations are of rank $1$ by Wahl's result \cite[Theorem 1]{Wah83}: 
indeed, Wahl proved that the set of Fano foliations of rank $1$ on smooth projective varieties (up to isomorphism) is discrete.

A key ingredient in the proof of Theorem \ref{thm: BAB analogue} is the following result of independent interest which shows that any variety admitting a Fano algebraically integrable foliation is of Fano type. 
We remark that the statement of Theorem \ref{thm: ft ai is ft} is new already when $t=1$.

\begin{alphthm}
\label{thm: ft ai is ft}
   Let $X$ be a projective variety and let $\mathcal F$ be an algebraically integrable foliation on $X$ such that
   \begin{itemize}
       \item $X$ is potentially klt, i.e. $(X, D)$ is klt for some $\mathbb R$-divisor $D \ge 0$,
       \item $\mathcal F$ is log canonical, and 
       \item $-(tK_{\Ff}+(1-t)K_X)$ is ample for some $t \in [0, 1]$.
   \end{itemize}
   
   Then $X$ is of Fano type, i.e. $-(K_X+\Delta)$ is ample for some klt pair $(X,\Delta)$. In particular, if $X$ is $\mathbb Q$-factorial, then $X$ is a Mori dream space.
\end{alphthm}

It is illuminating to compare Theorem \ref{thm: ft ai is ft} with an important result proven in  \cite{CP19}, namely, if $K_{\mathcal F}$ is not pseudo-effective then $K_X$ is not pseudo-effective, and in particular $X$ is uniruled. 
Theorem \ref{thm: ft ai is ft} implies that under the (stronger) hypothesis that $\mathcal F$ is of Fano type we get a much stronger conclusion, namely $X$ is of Fano type.  In particular, in this case $X$ is not only uniruled, but it is rationally connected.
When $t = 1$, we also expect that Theorem \ref{thm: ft ai is ft} will be useful in the classification of Fano foliations, begun in \cite{AD13}, which has generated considerable interest and activity in recent years (see Section \ref{s_fbab} for a further discussion on this point).

We remark that Theorem \ref{thm: ft ai is ft} requires that $\Ff$ is algebraically integrable, although we expect that Theorem \ref{thm: ft ai is ft} continues to hold without this hypothesis. We also remark that 
the proof of Theorem \ref{thm: ft ai is ft} does not rely on \cite{CP19} (and hence on the existence of MRC quotients), but rather follows from minimal model program methods, in particular, Miyaoka's foliated bend-and-break technique, cf. \cite[Corollary 2.28]{Spi20}.

\medskip 

\noindent\textbf{Acknowledgements.} The authors would like to thank Caucher Birkar for useful discussions, particularly a useful conversation in July 2024 when he suggested that the third author consider ``the threshold when MMP works" which turns out to be crucial for the proof of the main theorems. The authors would like to thank  Yen-An Chen, Dongchen Jiao, Jie Liu, James M\textsuperscript{c}Kernan, and Pascale Voegtli for several useful discussions.

The work is supported by the National Key R\&D Program of China (\#2024YFA1014400, \#2023YFA1010600, \#2020YFA0713200). The first author is partially funded by a Simons Collaboration Grant.  The second author is supported by NSFC for Excellent Young Scientists (\#12322102). The second author is a member of LMNS, Fudan University. The fourth author is partially supported by an AMS-Simons Travel Grant.
The fifth author is partially funded by EPSRC.
The sixth author is supported by the “Programma per giovani ricercatori Rita Levi Montalcini” of the Italian Ministry of University and Research and by PSR 2022 – Linea 4 of the University of Milan. He is a member of the GNSAGA group of INDAM. The last author is supported by a grant from the Simons Foundation.

\section{Further discussion of the main results and sketch of the proofs}

Our main results should be viewed in the context of a wider web of conjectures and results, which we now explain. In Section \ref{sec: proof sketch}, we will also give a sketch of the proof of Theorem \ref{thm: eogmm ai afs general case} which is the technical heart of our paper.

\subsection{MMP and good minimal models for adjoint foliated structures}
\label{sec: further discussion gmm}
In fact, we conjecture that something stronger than Theorem \ref{thm: bchm analogue} holds.

\begin{conj}[Existence of good minimal models]\label{conj: eogmm afs}
    Let $(X,\Ff,t)$ be a projective klt adjoint foliated structure with $K:=tK_{\Ff}+(1-t)K_X$ pseudo-effective and $t<1$.
    Then:
    \begin{enumerate}
        \item $(X,\Ff,t)$ has a good minimal model, i.e. there exists a $K$-negative birational contraction $\phi\colon X\dashrightarrow X_{\min}$ such that $\phi_*K$ is semi-ample.
        \item If $K$ is big, then $(X,\Ff,t)$ has a canonical model, i.e. there exists a $K$-non-positive birational contraction $\psi\colon X\dashrightarrow X_{\can}$ such that $\psi_*K$ is ample. Moreover, $X_{\can}$ is unique, and if $t$ is a rational number then $$X_{\can}=\Proj\bigoplus_{m=0}^{+\infty}H^0(X,\mathcal O_X(\lfloor m(tK_{\Ff}+(1-t)K_X)\rfloor)).$$
    \end{enumerate}
\end{conj}
Note that Conjecture \ref{conj: eogmm afs}(2) is a special case of Conjecture \ref{conj: eogmm afs}(1). However, as Conjecture \ref{conj: eogmm afs} implies the abundance conjecture by taking $\Ff=T_X$, it is natural to focus on Conjecture \ref{conj: eogmm afs}(2) first. When $\dim X=2$, $X$ is smooth, and $\Ff$ is canonical, Conjecture \ref{conj: eogmm afs}(2) was proven in \cite[Theorem 1.2]{SS23} when $t>\frac{5}{6}$, while \cite[Problem 1.9]{LW24} asked whether Conjecture \ref{conj: eogmm afs}(1) is true when $t=\frac{1}{2}$.  We also remark that as examples show, the condition that $t<1$ is necessary.

We prove Theorem \ref{thm: bchm analogue} as a consequence of an analogous result in a more general setting:
throughout this paper, we will need the additional flexibility of working with an adjoint foliated structure together with a boundary $\Rr$-divisor $B$ and a nef $\bb$-divisor $\Mm$ in the relative setting. For the formal definition of adjoint foliated structures, we refer the reader to Definition \ref{defn: afs}. Due to the complexity of the notation, in the rest of this paper, we shall usually adopt the convention ``$\Aa:=(X,\Ff,B,\Mm,t)$" to simplify the notation of adjoint foliated structures. See Notation \ref{nota: simplified notation of afs} for details.

With this setting in mind, we see that Theorem \ref{thm: bchm analogue} is an easy consequence of the following.

\begin{thmsub}\label{thm: eogmm ai afs general case}
    Let $\Aa/U:=(X,\Ff,B,\Mm,t)/U$ be a klt algebraically integrable adjoint foliated structure such that $K_{\Aa}$ is pseudo-effective$/U$. Assume that either $K_{\Aa}$ is big$/U$ or $B+\Mm_X$ is big$/U$.
    
    Then we may run a $K_{\Aa}$-MMP$/U$ with scaling of an ample$/U$ $\Rr$-divisor on $X$ and any such MMP terminates with a good minimal model of $\Aa/U$. In particular, $\Aa/U$ has a $\mathbb Q$-factorial good minimal model.
\end{thmsub}

Theorem \ref{thm: eogmm ai afs general case} can be viewed as the analogue of \cite[Theorem 1.2]{BCHM10} for klt algebraically integrable adjoint foliated structures and plays a crucial role in the study of these structures. 
For a more general version of it, see Theorem \ref{thm: eogmm lc plus ample case}. Before turning to a sketch of the proof of Theorem \ref{thm: eogmm ai afs general case}, we explain some of its applications.

Theorem \ref{thm: eogmm ai afs general case} immediately implies the base-point-freeness theorem for klt algebraically integrable adjoint foliated structures.

\begin{thmsub}[Base-point-freeness theorem]\label{thm: bpf ai afs}
 Let $\Aa/U$ be a klt algebraically integrable adjoint foliated structure with ambient variety $X$ and $H$ a nef$/U$ $\Rr$-divisor on $X$ such that $aH-K_{\Aa}$ is big$/U$ and nef$/U$ for some positive real number $a$. Then: 
 \begin{enumerate}
     \item $H$ is semi-ample$/U$.
     \item If $H$ is Cartier, then $\mathcal{O}_X(mH)$ is globally generated$/U$ for any integer $m\gg 0$.
 \end{enumerate}
\end{thmsub}

A special case of the base-point-freeness theorem is the contraction theorem.

\begin{thmsub}[Contraction theorem]\label{thm: cont intro}
Let $\Aa/U$ be an lc algebraically integrable adjoint foliated structure associated with a morphism $\pi\colon X\rightarrow U$, and let $R$ be a $K_{\Aa}$-negative extremal ray$/U$. Assume that $X$ is potentially klt. Then there exists a contraction$/U$ $\cont_R\colon X\rightarrow T$ satisfying the following.
    \begin{enumerate}
        \item For any integral curve $C$ such that $\pi(C)$ is a point, $\cont_R(C)$ is a point if and only if $[C]\in R$.
        \item Let $L$ be a line bundle on $X$ such that $L\cdot R=0$. Then there exists a line bundle $L_T$ on $T$ such that $L\cong f^\ast L_T$.
    \end{enumerate}
\end{thmsub}
The only difference between  Theorem \ref{thm: cont intro} and \cite[Theorem 1.4]{CHLMSSX24} is the assumption on the singularities of $X$.  We recall that \cite[Theorem 1.4]{CHLMSSX24} requires that $X$ is klt, while Theorem \ref{thm: cont intro} only requires that $X$ is potentially klt. Although this is a seemingly small technical improvement, it will actually give us much more flexibility when we need to apply the theorem. For example, we are able to run MMP for lc algebraically integrable adjoint foliated structures on potentially klt varieties without any $\mathbb Q$-factoriality hypothesis, cf. \cite[Theorem 1.6]{CHLMSSX24} where such a hypothesis is necessary.
\begin{thmsub}[Existence of MMP]\label{thm: mmp potentially klt ambient algint afs}
Let $\Aa/U$ be an lc algebraically integrable adjoint foliated structure on a potentially klt variety. Then we may run a $K_{\Aa}$-MMP$/U$.
\end{thmsub}

We also have the following result on the existence of Mori fiber spaces.

\begin{thmsub}[Existence of Mori fiber spaces]\label{thm: eomfs ai afs}
 Let $\Aa/U$ be an lc algebraically integrable adjoint foliated structure such that $X$ is potentially klt and $K_{\Aa}$ is not pseudo-effective$/U$. Then we may run a $K_{\Aa}$-MMP$/U$ with scaling of an ample$/U$ $\Rr$-divisor on $X$ and any such MMP terminates with a Mori fiber space of $\Aa/U$.
\end{thmsub}

Another interesting corollary of these results is the following.

\begin{thmsub}\label{thm: gmm klt algint foliation gt on cy variety}
Let $(X,\Delta)$ be a klt Calabi--Yau pair. Then any algebraically integrable foliation of general type has a canonical model.
\end{thmsub}

\subsection{Birational geometry for algebraically integrable adjoint foliated structures}

Besides direct applications to the minimal model program,  Theorem \ref{thm: eogmm ai afs general case} also has several applications in birational geometry that are closely related to the minimal model program.

\begin{thmsub}\label{thm: abundance ai afs}
  Let $\Aa/U$ be a klt algebraically integrable adjoint foliated structure with big$/U$ generalized boundary. Then:
  \begin{enumerate}
      \item (Abundance) $\kappa_{\iota}(X/U,K_{\Aa})=\kappa_{\sigma}(X/U,K_{\Aa})$.
      \item (Non-vanishing) If $K_{\Aa}$ is pseudo-effective$/U$, then $K_{\Aa}\sim_{\mathbb R,U}D\geq 0$ for some $D$.
      \item (Finite generation) If $K_{\Aa}$ is a $\mathbb Q$-divisor, then the canonical ring
$$R(X/U,K_{\Aa}):=\Proj\bigoplus_{m=0}^{+\infty}\pi_*\mathcal{O}_X\left(\left\lfloor mK_{\Aa}\right\rfloor\right)$$
is finitely generated, where $\pi: X\rightarrow U$ is the associated morphism.
  \end{enumerate}
\end{thmsub}

We show that $\mathbb Q$-factorial minimal models of klt algebraically integrable adjoint foliated structures are connected via a sequence of flops, cf. \cite{JV23,CLW24, CM24}.

\begin{thmsub}[Decomposition in flops]\label{thm: flop ai afs}
  Let $\Aa/U$ be an NQC klt algebraically integrable adjoint foliated structure and $\Aa_1/U$, $\Aa_2/U$ two $\Qq$-factorial minimal models of $\Aa/U$. Then the induced birational map$/U$ $\phi: \Aa_1\dashrightarrow \Aa_2$ can be decomposed into a finite sequence of $K_{\Aa_1}$-flops$/U$.
\end{thmsub}

We have the following result which allows us to extract non-terminal places of algebraically integrable adjoint foliated structures.

\begin{thmsub}[Extraction of non-terminal places]\label{thm: extract non-terminal place intro}
   Let $\Aa$ be an lc algebraically integrable adjoint foliated structure such that $X$ is potentially klt. Let $\mathcal{S}$ be a finite set of exceptional$/X$ prime divisors, such that for any $D\in\mathcal{S}$,
   \begin{enumerate}
       \item $a(D,\Aa)\leq 0$, and
       \item if $a(D,\Aa)=0$, then $V\not\subset\Center_XD$ for any nklt center $V$ of $\Aa$.
   \end{enumerate}
   Then there exists a projective birational morphism $f: Y\rightarrow X$ such that $Y$ is $\Qq$-factorial and the divisors contracted by $f$ are exactly the divisors in $\mathcal{S}$.
\end{thmsub}

While it is easy to deduce Theorem \ref{thm: extract non-terminal place intro} from Theorem \ref{thm: eogmm ai afs general case} -- 
cf. the proof of \cite[Corollary 1.4.3]{BCHM10} --
we do not prove Theorem \ref{thm: extract non-terminal place intro} as a consequence of Theorem \ref{thm: eogmm ai afs general case}. In fact, Theorem \ref{thm: extract non-terminal place intro} plays a crucial role in the proof of Theorem \ref{thm: eogmm ai afs general case}.

\subsection{MMP for algebraically integrable foliations on potentially klt varieties}

Our progress on the existence of minimal model program for adjoint foliated structures allows us to address the key technical challenges in proving the existence of minimal model program for algebraically integrable foliations. 

We recall that there has been a large amount of progress on the existence of the minimal model program (cone theorem, contraction theorem, and the existence of flips) for log canonical algebraically integrable foliations, see \cite{ACSS21,CHLX23,CS24,LMX24,CHLMSSX24}. Thanks to Theorem \ref{thm: eogmm ai afs general case}, we establish the existence of minimal model program for lc algebraically integrable foliations on potentially klt varieties unconditionally.

\begin{thmsub}\label{thm: mmp ai potentially klt ambient variety}
Let $(X,\Ff,B)/U$ be an lc algebraically integrable foliated triple such that $X$ is potentially klt. Let $A$ be an ample$/U$ $\Rr$-divisor on $X$. Then: 
\begin{enumerate}
    \item (Existence of MMP) The cone theorem, the contraction theorem, and the existence of flips hold for $K_{\Ff}+B$ over $U$. In particular, we can run a $(K_{\Ff}+B)$-MMP$/U$.
    \item (Existence of good minimal models after polarization) If $K_{\Ff}+B+A$ is pseudo-effective$/U$, then $(X,\Ff,B+A)/U$ admits a good minimal model. In particular:
    \begin{enumerate}
        \item (Base-point-freeness) If $K_{\Ff}+B+A$ is nef$/U$ then it is semi-ample$/U$.
        \item (Abundance) $\kappa_{\iota}(X/U,K_{\Ff}+B+A)=\kappa_{\sigma}(X/U,K_{\Ff}+B+A)$.
        \item (Non-vanishing) $K_{\Ff}+B+A\sim_{\mathbb R,U}D\geq 0$ for some $D$.
    \end{enumerate}
    \item (Existence of Mori fiber spaces) If $K_{\Ff}+B$ is not pseudo-effective$/U$, then $(X,\Ff,B)/U$ admits a Mori fiber space.
    \item (Existence of minimal models vs MMP with scaling) If $(X,\Ff,B)/U$ admits a minimal model in the sense of Birkar-Shokurov, then we may run a $(K_{\Ff}+B)$-MMP$/U$ with scaling of $A$ and any such MMP terminates.
\end{enumerate}
\end{thmsub}

It is interesting to ask that whether the condition that $X$ is potentially klt can be removed. It is known that this condition can be removed when $(X,\Ff,B)$ is induced by a locally stable family, cf. \cite{MZ23}, but this is already difficult to prove.

\subsection{Foliated BAB}
\label{s_fbab}

There has been much recent progress on the classification of Fano foliations via the study of (generalized) indices and Kobayashi-Ochiai type theorems, cf. \cite{ADK08,AD13,AD14,Hör14,AD19,Liu23}. These results often have few restrictions on the singularities of the foliations but usually require the ambient varieties to be smooth. 

Eventually, we would like to view these results in the context of a broader theory of moduli of Fano type foliations. The first step towards such a theory is to prove a boundedness statement for Fano foliations. As explained earlier in the introduction, it is not entirely clear what the correct boundedness statement should be. Theorem \ref{thm: BAB analogue} is one approach to such a statement. It is heavily influenced by Birkar's celebrated proof of the BAB conjecture \cite{Bir19,Bir21}. More precisely, since \cite{Bir19,Bir21} proved that for $\epsilon>0$,  $\epsilon$-lc Fano varieties of fixed dimension are bounded, we would like to show that Fano foliations of fixed dimension that are $\epsilon$-lc are bounded. However, as explained earlier, this is not a reasonable statement since the set of $\epsilon$-lc Fano foliations is empty.

It seems more natural to consider \emph{klt Fano adjoint foliated structures}, i.e. a klt adjoint foliated structure $\Aa$ such that $-K_{\Aa}$ is ample. Theorem \ref{thm: BAB analogue} can therefore be viewed as the natural restatement of BAB in the setting of klt Fano algebraically integrable adjoint foliated structures. We also remark that Theorem \ref{thm: BAB analogue} follows from a more general result:

\begin{thmsub}\label{thm: log foliated BAB}
Let $d$ be a positive integer and $\epsilon$ a positive real number. Then 
the set $\mathcal P$ of foliated projective pairs $(X,\mathcal F)$, where $\mathcal F$ is algebraically integrable and  
    $(X,\Ff,t)$ is of $\epsilon$-Fano type (see Definition-Lemma \ref{deflem: fano type afs}) for some $t\geq\epsilon$
is bounded. 
\end{thmsub}

The first step in our proof of Theorem \ref{thm: log foliated BAB}
is the following generalization of Theorem \ref{thm: ft ai is ft}.

\begin{thmsub}\label{thm: ft ai afs mds}
   Let $\Aa/U$ be an lc algebraically integrable adjoint foliated structure such that $X$ is potentially klt and $-K_{\Aa}$ is ample$/U$.
   
   Then $X$ is of Fano type$/U$, i.e. $-(K_X+\Delta)$ is ample$/U$ for some klt pair $(X,\Delta)$. In particular, if $X$ is $\mathbb Q$-factorial, then $X$ is a Mori dream space.
\end{thmsub}

Theorem  \ref{thm: ft ai afs mds} is also a strengthening of \cite[Theorem 1.9]{LMX24} which shows that, for any lc algebraically integrable foliated triple $(X,\Ff,B)$ such that $X$ is $\mathbb Q$-factorial klt and $-(K_{\Ff}+B)$ is ample, $X$ is a Mori dream space.
 
We also remark that the proof of Theorem \ref{thm: log foliated BAB} relies on \cite{Bir21}. Indeed, by letting $\Ff=T_X$ and $t=1$, Theorem \ref{thm: log foliated BAB} is equivalent to \cite[Theorem 1.1]{Bir21}.

\subsection{Geography of minimal models}

Similar to the case of klt pairs, a key ingredient in the proof of Theorem \ref{thm: eogmm ai afs general case} is the finiteness of models, cf. \cite[Theorem E]{BCHM10}. 
First we provide the following definition:

\begin{defnsub}\label{defn: polytope afs}
Let $\pi\colon X\rightarrow U$ be a projective morphism between normal quasi-projective varieties. A \emph{tuple} of nef$/U$ $\bb$-Cartier $\bb$-divisors on $X$ is of the form $\MM:=(\Mm_1,\dots,\Mm_n)$ for some positive integer $n$ where each $\Mm_i$ is a nef$/U$ $\bb$-Cartier $\bb$-divisor. We define $\dim\MM:=n$. We define
$$\Span_{\mathbb R}(\MM):=\left\{\sum_{i=1}^n a_i\Mm_i\middle| a_i\in\mathbb R\right\}, \Span_{\mathbb R_{\geq 0}}(\MM):=\left\{\sum_{i=1}^n a_i\Mm_i\middle| a_i\in\mathbb R_{\geq 0}\right\}.$$
For any nef$/U$ $\bb$-Cartier $\bb$-divisor $\NN$, we define $(\MM,\NN):=(\Mm_1,\dots,\Mm_n,\NN)$.

Given a tuple $\MM$ of nef$/U$ $\bb$-Cartier $\bb$-divisors on $X$, let $V\subset\Weil_{\mathbb R}(X)\times\Span_{\mathbb R}(\MM)$ be a finite dimensional affine subspace, $\Ff$ a foliation on $X$, and $t\in [0,1]$ a real number. We define:
    $$V_{(X,\Ff,t)}:=\{\Aa\mid \Aa=(X,\Ff,B,\Mm,t),(B,\Mm)\in V\}$$
    $$\mathcal{L}(V_{(X,\Ff,t)}):=\{\Aa \mid \Aa=(X,\Ff,B,\Mm,t)\in V_{(X,\Ff,t)},\Aa\text{ is lc},\Mm\in\Span_{\mathbb R_{\geq 0}}(\MM)\}.$$
    $$\mathcal{E}_{\pi}(V_{(X,\Ff,t)}):=\{\Aa\mid\Aa\in \mathcal{L}(V_{(X,\Ff,t)}), K_{\Aa}\text{ is pseudo-effective}/U\}.$$
\end{defnsub}

We then have the following result on the finiteness of $\mathbb Q$-factorial minimal models (see \cite[Theorem 3.2]{Mas24} for a related result for foliations on threefolds). 

\begin{thmsub}[Finiteness of $\mathbb Q$-factorial minimal models]\label{thm: finiteness of minimal models intro}
Assume that
\begin{itemize}
    \item $\pi\colon X\rightarrow U$ is a projective morphism between normal quasi-projective varieties,
    \item $\MM$ is a tuple of nef$/U$ $\bb$-Cartier $\bb$-divisors,
    \item $V$ is a finite dimensional rational affine subspace of $\Weil_{\mathbb R}(X)\times\Span_{\mathbb R}(\MM)$,
    \item $\Ff$ is an algebraically integrable foliation on $X$,
    \item $t\in [0,1]$ is a rational number. 
    \item $\mathcal{C}\subset\mathcal{L}(V_{(X,\Ff,t)})$ is a rational polytope, such that for any $\Aa:=(X,\Ff,B,\Mm,t)\in\mathcal{C}$, $\Aa$ is klt and $B+\Mm_X$ is big$/U$.
\end{itemize}
Then there are finitely many birational maps$/U$ $\phi_j: X\dashrightarrow Y_j$, $1\leq j\leq k$ satisfying the following. 
\begin{enumerate}
\item For any $\Aa\in\mathcal{C}$ such that $K_{\Aa}$ is pseudo-effective$/U$, there exists an index $1\leq j\leq k$ such that $(\phi_j)_*\Aa/U$ is a $\mathbb Q$-factorial good minimal model of $\Aa/U$.
\item For any $\Aa\in\mathcal{C}$ and any $\mathbb Q$-factorial minimal model $\Aa_Y/U$ of $\Aa/U$ with induced birational map $\phi\colon X\dashrightarrow Y$, there exists an index $1\leq j\leq k$ such that $\phi_j\circ\phi^{-1}\colon Y\dashrightarrow Y_j$ is an isomorphism.
\end{enumerate}
\end{thmsub}

Of course, one can obtain the finiteness of weak lc models, finiteness of ample models, and the Shokurov-type polytope via similar arguments for pairs as in \cite[Corollary 1.1.5]{BCHM10}. We do not need these results in this paper. Since they are technical, we omit them and leave them to the reader.

We also note that the finiteness of models arguments in Theorem \ref{thm: finiteness of minimal models intro} have the nef$/U$ part $\Mm$. Therefore, by letting $\Ff:=T_X$, we obtain the finiteness of minimal models for generalized pairs. This seems to be a folklore result, but we cannot find any reference on it. Hence, Theorem \ref{thm: finiteness of minimal models intro} could be useful for the study of generalized pairs as well.

We need the following important theorem on the Shokurov-type polytope to prove that minimal models are connected by flops, cf. Theorem \ref{thm: flop ai afs}.

\begin{thmsub}[Shokurov-type polytope]\label{thm: shokurov type polytope}
Let 
$$\Aa:=\left(X,\Ff,B:=\sum_{i=1}^mv_{0,i}B_i,\Mm:=\sum_{i=1}^n\mu_{0,i}\Mm_i,t_0\right)\Bigg/U$$
be an lc algebraically integrable adjoint foliated structure, where $B_i\geq 0$ are Weil divisors and $\Mm_i$ are nef$/U$ $\bb$-Cartier $\bb$-divisors. For any $\bm{v}:=(v_1,\dots,v_m,\mu_1,\dots,\mu_n,t)\in\mathbb R^{m+n+1}$, we define $$\Aa(\bm{v}):=\left(X,\Ff,\sum_{i=1}^mv_{i}B_i,\sum_{i=1}^n\mu_{i}\Mm_i,t\right).$$
Let $V$ be the rational envelope of $\bm{v}_0:=(v_{0,1},\dots,v_{0,m},\mu_{0,1},\dots,\mu_{0,n},t_0)\in\mathbb R^{m+n+1}$. Then there exists an open subset $V_0\ni\bm{v}_0$ in $V$, such that for any $\bm{v}\in V_0$,
\begin{enumerate}
    \item  $\Aa(\bm{v})$ is lc, and
    \item if $K_{\Aa}$ is nef$/U$, then $K_{\Aa(\bm{v})}$ is nef$/U$.
\end{enumerate}
\end{thmsub}
The following more practical result is a direct consequence of Theorem \ref{thm: shokurov type polytope}:

\begin{thmsub}\label{thm: nqc of ka}
    Let $\Aa/U$ be an NQC lc algebraically integrable adjoint foliated structure. If $K_{\Aa}$ is nef$/U$ then $K_{\Aa}$ is NQC$/U$. 
\end{thmsub}

\subsection{Sketch of the proof of Theorem \ref{thm: eogmm ai afs general case}}\label{sec: proof sketch}
We now sketch the proof of Theorem \ref{thm: eogmm ai afs general case}.
First we recall the methodology of \cite{CHLMSSX24} where the existence of flips (a special case of the existence of good minimal models) was proven. By passing to a foliated log resolution, we may assume that $\Aa$ is foliated log smooth and $\Mm=\Mm'+\overline{A}$ for some ample$/U$ $\Rr$-divisor $A$ (here $\overline{A}$ stands for the $\bb$-divisor which descends to $A$). Now we can first run a $(K_{\Ff}+B^{\ninv}+\Mm_X)$-MMP$/U$ $\phi: X\dashrightarrow X'$ with scaling of an ample divisor which terminates by \cite[Theorem 1.5]{LMX24}.  We may then run a $K_{\Aa'}$-MMP$/U$ $\psi: X'\dashrightarrow X''$ where $\Aa':=\phi_*\Aa$ by running a $(K_{X'}+B'+\Nn_{X'})$-MMP$/U$ where $\Nn_{X'}= \frac{1}{1-t}(K_{\mathcal F'}+(B^{\ninv})'+\Mm_{X'})$. This is an MMP of a klt generalized pair whose generalized boundary is big$/U$, so the MMP terminates by \cite[Lemma 4.4]{BZ16}. In the case that $X \to U$ is a flipping contraction, one can show that the output of this MMP is the flip.  However, in general, it is unclear if the output of this MMP is our desired good minimal model.  

The basic issue is that there is no clear relation between steps of a
$(K_{\Ff}+B^{\ninv}+\Mm_X)$-MMP$/U$ and steps of a $K_{\Aa}$-MMP$/U$.  This can be seen most clearly on surfaces when $0<t$ is small.  For an integer $r\ge 3$, consider the resolution of a quotient singularity of type $\frac{1}{n}(1, 1)$ by a blow up $b\colon X \to \mathbb C^2/\mu_r$ which extracts a single curve $E$.  As is well known, the quotient of the foliation generated by $\partial_x$ on $\mathbb C^2$ descends to a foliation on $\mathbb C^2/\mu_n$.
Moreover, if we denote by $\mathcal F$ the transform of this foliation on $X$, then $K_{\mathcal F}\cdot E = -1$, but if $t< \frac{n-2}{n-1}$ then $K_{\Aa}\cdot E >0$ where $\Aa=(X,\Ff,t)$.  Another way to phrase this issue is that in general the negative part of the Nakayama-Zariski decomposition of $K_{\mathcal F}$
is not contained in the negative part of the Nakayama-Zariski decomposition of $K_{\Aa}$.

To address this issue we need a way to interpolate between 
$K_{\Ff}+B^{\ninv}+\Mm_X$ and $K_{\Aa}$.
More precisely, assume that $(X,\Ff,B,\Mm)$ is foliated log smooth and let $\Aa_s:=(X,\Ff,B^{\ninv}+\frac{1-s}{1-t}B^{\inv},\Mm,s)$. 
Consider the following subset of $[t, 1]$ 
\[\mathcal P(t) := \{s \in [t, 1]: \Aa_s/U \text{ has a } \mathbb Q\text{-factorial good minimal model}\}.\]

As noted, we can ensure that $1 \in \mathcal P(t)$ by \cite[Theorem 1.5]{LMX24}. What we would like to show is that $t \in \mathcal P(t)$.
We do this by showing the following two statements.

\begin{enumerate}
\item If $s>t$, $s\in \mathcal P(t)$ and $\Aa_{s-\epsilon_0}$ is pseudo-effective for some $\epsilon_0>0$, then $s-\epsilon \in \mathcal P(t)$ for all $0<\epsilon \ll 1$.
\item If $s+\epsilon \in \mathcal P(t)$ for all $0<\epsilon \ll 1$ then $s \in \mathcal P(t)$.
\end{enumerate}
These two statements easily imply that if $1 \in \mathcal P(t)$, then $t \in \mathcal P(t)$.

It is easy to prove (1). We need to show that if $\Aa_s/U$ has a good minimal model, then $\Aa_{s-\epsilon}/U$ has a good minimal model for any $0<\epsilon\ll 1$.  Let $\Aa'_s/U$ be a good minimal model of $\Aa_s/U$ with induced map $\phi\colon X\dashrightarrow X'$ and let $\Aa_r':=\phi_*\Aa_r$ for any $r$. Since $\phi$ is $K_{\Aa_s}$-negative, it is also $K_{\Aa_{s-\epsilon}}$-negative for any $0<\epsilon\ll 1$. Now
$$\frac{s}{\epsilon}K_{\Aa'_{s-\epsilon}}=K_{\Aa'_0}+\frac{s-\epsilon}{\epsilon}K_{\Aa_s'}$$
has the structure of a klt generalized pair (\cite[Proposition 3.3]{CHLMSSX24}) with big$/U$ generalized boundary. Thus, $\Aa'_{s-\epsilon}/U$ has a good minimal model which is also a good minimal model of $\Aa_{s-\epsilon}/U$ since $\phi$ is $K_{\Aa_{s-\epsilon}}$-negative and then we can apply the negativity lemma to conclude.

To prove (2) we need to show that if $\Aa_{s+\epsilon}/U$ has a good minimal model for any $0<\epsilon\ll 1$, then $\Aa_s/U$ has a good minimal model.
As we saw in the example above, the challenge can be stated in terms of the negative part of the Nakayama-Zariski decomposition.  If one could show that 
the $(K_{\Ff}+B^{\ninv}+\Mm_X)$-MMP$/U$ $\phi\colon X \dashrightarrow X'$ only contracts divisors that are contained in $\Supp N_{\sigma}(X/U,K_{\Aa})$, then the output of the MMP $\psi\colon X' \dashrightarrow X''$ will be a minimal model of $\Aa/U$.

Considering that, we observe the following: there exists $0<\delta\ll 1$ such that $E:=\Supp N_{\sigma}(X/U,K_{\Aa_r})$ is fixed for any $s<r\leq s+\delta$, and $\Supp N_{\sigma}(X/U,K_{\Aa_{s}})\subset E$. By assumption there exists a good minimal model of $\Aa_{s+\delta}/U$ with induced map $\phi\colon X\dashrightarrow X'$.  For any $s<r\leq s+\delta$, let $\Aa_r':=\phi_*\Aa_r$. Now pick $0<\tau\ll\delta$. Since
$$\frac{s+\delta}{\tau}K_{\Aa'_{s+\delta-\tau}}=K_{\Aa'_{0}}+\frac{s+\delta-\tau}{\tau}K_{\Aa'_{s+\delta}}$$
has the structure of a klt generalized pair with big$/U$ generalized boundary and $K_{\Aa'_{s+\delta}}$ is semi-ample$/U$, by the length of extremal rays, we may run a $K_{\Aa'_{s+\delta-\tau}}$-MMP$/U$ with scaling which terminates with a good minimal model and the MMP is $K_{\Aa'_{s+\delta}}$-trivial. Let $\psi\colon X'\dashrightarrow X''$ be the induced map and let $\Aa''_r:=\psi_*\Aa'_r$ for any $r$. Finally, we run a $K_{\Aa''_0}$-MMP$/U$ with scaling of $K_{\Aa''_{s+\delta}}$. This makes sense since $\Aa''_0$ is a klt generalized pair with big$/U$ generalized boundary and
$$K_{\Aa''_0}+\frac{s+\delta-\tau}{\tau}K_{\Aa''_{s+\delta}}=\frac{s+\delta}{\tau}K_{\Aa''_{s+\delta-\tau}}$$
is nef$/U$.

Let $\lambda_i$ be the scaling number of the $i$-th step of this MMP $X_i\dashrightarrow X_{i+1}$ with $X_1:=X''$. Then there exists a minimal positive integer $i$ such that
$\lambda_{i}\leq\frac{s}{\delta}$. Let $\Aa^i_r$ be the image of $\Aa_r$ on $X_i$ for any $r$. Then there exists $u>s$ such that $K_{\Aa^i_{r}}$ is nef$/U$ for any $r\in [s,u]$.

By our construction, the divisors contracted by $X\dashrightarrow X_i$ are exactly the components of $E$, hence $\Aa^i_r/U$ and the good minimal model of $\Aa_r/U$ are isomorphic in codimension $1$ for any $r\in (s,u]$. By the negativity lemma, $\Aa^i_r/U$ is a good minimal model of $\Aa_r/U$ for any $r\in (s,u]$. Let $p\colon W\rightarrow X$ and $q\colon  W\rightarrow X_i$ be a resolution of indeterminacy. Then for any $r\in (s,u]$, we have
$$p^*K_{\Aa_r}=q^*K_{\Aa^i_r}+F_r$$
for some $F_r\geq 0$. Thus
$$p^*K_{\Aa_{s}}=\lim_{r\rightarrow s^+}p^*K_{\Aa_r}=\lim_{r\rightarrow s^+}(q^*K_{\Aa^i_r}+F_r)=q^*K_{\Aa^i_{s}}+F_{s},$$
where $F_{s}=\lim_{r\rightarrow s^+}F_r\geq 0$. 

This almost completes the proof of the existence of good minimal models. However, at this point, we can only show that $\Aa^i_s/U$ is a weak lc model of $\Aa_s/U$ rather than a minimal model of $\Aa_s/U$ (because the support of $F_s$ may not be the same as that of $F_r$ for $r>s$). So we still need to show the existence of a minimal model of $\Aa^i_s/U$. The idea seems natural: we simply extract all 
the exceptional$/X_i$ prime divisors $E$ on $X$ such that $a(E,\Aa_s)=a(E,\Aa^i_s)$. To do this, we need the klt case of Theorem \ref{thm: extract non-terminal place intro}, that is, Theorem \ref{thm: 
extract non-terminal places}: given a $\mathbb Q$-factorial klt $\Aa=(X,\Ff,B,\Mm,t)$ and a prime divisor $E$ over $X$, there exists a birational morphism $f\colon  Y\rightarrow X$ where $Y$ is 
$\mathbb Q$-factorial and $E$ is the only divisor contracted by $f$. If $a(E,X,B^{\ninv}+\frac{1}{1-t}B^{\inv},\Mm)\leq 0$, then \cite[Lemma 4.6]{BZ16} applies directly and we are done. But if not, 
then the proof is not straightforward and we first need to apply the methodology of \cite{CHLMSSX24}: 
\begin{itemize}
    \item take a foliated log resolution $W\rightarrow X$, and increase all the coefficients of the exceptional divisors to some positive value strictly less than $1$ except $E$ but keep the coefficient of $E$ unchanged;
    \item run an MMP$/X$ of the foliation part $W\dashrightarrow V$; and
    \item run an MMP$/X$ of the whole structure on $V$ by considering it as a generalized pair$/X$.
\end{itemize}
If $E$ is not exceptional$/V$ then by the very exceptional MMP -- cf. \cite[Lemma 3.4]{Bir12} -- the output of the process above induces the extraction we desire and we are left to deal with the case when $E$ is exceptional$/V$. Now the key idea is to consider the generalized pair structure on $V$, where we put the log canonical $\mathbb R$-divisor of the foliation into the nef part$/X$ of this generalized pair. This gives us a klt generalized pair$/X$ structure on $V$, and by running the MMP, we know that the output of the MMP is a klt generalized pair on $X$. Via a complicated calculation of discrepancies, one can show that the discrepancy of $E$ with respect to this generalized pair is $\leq 0$, so we can extract it by \cite[Lemma 4.6]{BZ16} again.

Finally, to prove Theorem \ref{thm: eogmm ai afs general case}, we are only left to show the termination of MMP with scaling. Since we have the existence of good minimal models, we could apply the same lines of the arguments of \cite{BCHM10} to prove the finiteness of models (Theorem \ref{thm: finiteness of minimal models intro}), and use it to directly deduce the termination of MMP with scaling (Theorem \ref{thm: nqc mmp with scaling posssible not ample}). In practice, we apply the simplified arguments as in \cite[Section 10]{HK10}. We also emphasize that we cannot apply the same lines of the arguments of \cite{BCHM10} to the proof of the existence of good minimal models. This is because \cite[Proof of Theorem D]{BCHM10} essentially uses the Kawamata-Viehweg vanishing theorem and the machinery of the extension theorems which we do not have for adjoint foliated structures.

\medskip

\noindent\textit{Structure of the paper.} In Section \ref{sec: preliminaries} we recall some preliminary results and prove some perturbation results that will be used later. In Section \ref{sec: extract non-terminal place} we prove the extractability of non-terminal places (Theorem \ref{thm: extract non-terminal place intro}). In Section \ref{sec: eogmm} we prove the ``existence of good minimal models" part of the first main theorem (Theorem \ref{thm: eogmm boundary big case}). In Section \ref{sec: fom} we prove the finiteness of minimal models (Theorem \ref{thm: finiteness of minimal models intro}). In Section \ref{sec: mmp with scaling} we consider the minimal model program with scaling. In Section \ref{sec: proof of the main theorems} we prove Theorem \ref{thm: bchm analogue} and provide applications. Actually, we prove all the main theorems except those related to Fano foliations. In Section \ref{sec: bdd fano} we study Fano algebraically integrable adjoint foliated structures and prove Theorems \ref{thm: BAB analogue} and \ref{thm: ft ai is ft}.

\section{Preliminaries}\label{sec: preliminaries}

We will adopt the standard notation and definitions on MMP in \cite{KM98,BCHM10} and use them freely. For adjoint foliated structures, generalized foliated quadruples, foliated triples, foliations, and generalized pairs, we adopt the notation and definitions in \cite{CHLMSSX24}  which generally align with \cite{LLM23,CHLX23} (for generalized foliated quadruples), \cite{CS20, ACSS21, CS21} (for foliations and foliated triples), and \cite{BZ16,HL23} (for generalized pairs and $\bb$-divisors), possibly with minor differences. 

\subsection{Notation}

\begin{defn}
    A \emph{contraction} is a projective morphism of varieties
    $f \colon X \to Y$
    such that 
    $f_\ast \mathcal{O}_X=\mathcal{O}_Y$.
\end{defn}

When $X, Y$ are normal, the condition 
$f_\ast \mathcal{O}_X=\mathcal{O}_Y$
is equivalent to the fibers of $f$ being connected.

\begin{nota}
    Let $h \colon  X\dashrightarrow X'$ be a birational map between normal varieties. We denote by $\Exc(h)$ the reduced divisor supported on the codimension one part of the exceptional locus of $h$. 
\end{nota}

\begin{defn}\label{defn: klt cy type}
A \emph{klt Calabi--Yau pair} $(X,\Delta)$ is a klt pair such that $K_X+\Delta\equiv 0$ for some $\Rr$-divisor $\Delta$. By \cite[Theorem 0.1]{Amb05} and \cite[Lemma 5.3]{HLS24} this is equivalent to say that $(X,D)$ is klt and $K_X+D\sim_{\mathbb Q}0$ for some $\mathbb Q$-divisor $D$.
\end{defn}

\begin{defn}[NQC]
    Let $X\rightarrow U$ be a projective morphism between normal quasi-projective varieties. Let $D$ be a nef $\Rr$-Cartier $\Rr$-divisor on $X$ and $\Mm$ a nef $\bb$-divisor on $X$. 
        We say that $D$ is \emph{NQC}$/U$ if $D=\sum d_iD_i$, where each $d_i\geq 0$ and each $D_i$ is a nef$/U$ Cartier divisor. Here NQC stands for ``Nef $\mathbb Q$-Cartier Combinations", cf. \cite[Definition 2.15]{HL22}. We say that $\Mm$ is \emph{NQC}$/U$ if $\Mm=\sum \mu_i\Mm_i$, where each $\mu_i\geq 0$ and each $\Mm_i$ is a nef$/U$ $\bb$-Cartier $\bb$-divisor.
\end{defn}

\subsection{Adjoint foliated structures}

\begin{defn}[Foliations, {cf. \cite{ACSS21,CS21}}]\label{defn: foliation}
Let $X$ be a normal variety. A \emph{foliation} on $X$ is a coherent sheaf $T_{\Ff}\subset T_X$ such that
\begin{enumerate}
    \item $T_{\Ff}$ is saturated in $T_X$, i.e. $T_X/T_{\Ff}$ is torsion free, and
    \item $T_{\Ff}$ is closed under the Lie bracket.
\end{enumerate}

The \emph{rank} of a foliation $\Ff$ on a variety $X$ is the rank of $\Ff$ as a sheaf and is denoted by $\rk\Ff$. 
The \emph{co-rank} of $\Ff$ is $\dim X-\rk\Ff$. The \emph{canonical divisor} of $\Ff$ is a divisor $K_\Ff$ such that $\mathcal{O}_X(-K_{\mathcal{F}})\cong\mathrm{det}(T_\Ff)$. If $T_\Ff=0$, then we say that $\Ff$ is a \emph{foliation by points}.

Given any dominant map 
$h: Y\dashrightarrow X$ and a foliation $\mathcal F$ on $X$, we denote by $h^{-1}\Ff$ the \emph{pullback} of $\Ff$ on $Y$ as constructed in \cite[3.2]{Dru21} and say that $h^{-1}\Ff$ is \emph{induced by} $\Ff$. Given any birational map $g: X\dashrightarrow X'$, we denote by $g_\ast \Ff:=(g^{-1})^{-1}\Ff$ the \emph{pushforward} of $\Ff$ on $X'$ and also say that $g_\ast \Ff$ is \emph{induced by} $\Ff$. We say that $\Ff$ is an \emph{algebraically integrable foliation} if there exists a dominant map $f: X\dashrightarrow Z$ such that $\Ff=f^{-1}\Ff_Z$, where $\Ff_Z$ is the foliation by points on $Z$, and we say that $\Ff$ is \emph{induced by} $f$.

A subvariety $S\subset X$ is called \emph{$\Ff$-invariant} if for any open subset $U\subset X$ and any section $\partial\in H^0(U,\Ff)$, we have $\partial(\mathcal{I}_{S\cap U})\subset \mathcal{I}_{S\cap U}$,  where $\mathcal{I}_{S\cap U}$ denotes the ideal sheaf of $S\cap U$ in $U$.  
For any prime divisor $P$ on $X$, we define $\epsilon_{\Ff}(P):=1$ if $P$ is not $\Ff$-invariant and $\epsilon_{\Ff}(P):=0$ if $P$ is $\Ff$-invariant. For any prime divisor $E$ over $X$, we define $\epsilon_{\Ff}(E):=\epsilon_{\Ff_Y}(E)$ where $h: Y\dashrightarrow X$ is a birational map such that $E$ is on $Y$ and $\Ff_Y:=h^{-1}\Ff$.

Suppose that the foliation structure $\Ff$ on $X$ is clear in the context. Then, given an $\mathbb R$-divisor $D = \sum_{i = 1}^k a_iD_i$ where each $D_i$ is a prime Weil divisor,
we denote by $D^{{\rm ninv}} \coloneqq \sum \epsilon_{\mathcal F}(D_i)a_iD_i$ the \emph{non-$\Ff$-invariant part} of $D$ and $D^{{\rm inv}} \coloneqq D-D^{{\rm ninv}}$ the \emph{$\Ff$-invariant part} of $D$.
\end{defn}

\begin{defn}\label{defn: afs}
An \emph{adjoint foliated structure} $(X,\Ff,B,\Mm,t)/U$ is the datum of a normal quasi-projective variety $X$ and a projective morphism $X\rightarrow U$, a foliation $\Ff$ on $X$, an $\Rr$-divisor $B\geq 0$ on $X$, a $\bb$-divisor $\Mm$ nef$/U$, and a real number $t\in [0,1]$ such that $K_{(X,\Ff,B,\Mm,t)}:=tK_{\Ff}+(1-t)K_X+B+\Mm_X$ is $\Rr$-Cartier. 

When $t=0$ or $\Ff=T_X$, we call $(X,B,\Mm)/U$ a \emph{generalized pair}, and in addition, if $\Mm=\bm{0}$, then we call $(X,B)/U$ a pair. 
When $t=1$, we call $(X,\Ff,B,\Mm)/U$ a \emph{generalized foliated quadruple}, and in addition, if $\Mm=\bm{0}$, then we call $(X,\Ff,B)/U$ a \emph{foliated triple}. We say that $(X,\Ff,B,\Mm,t)/U$ is \emph{NQC} if $\Mm$ is NQC$/U$, i.e. $\Mm$ is an $\mathbb R_{\geq 0}$-combination of nef$/U$ $\bb$-Cartier $\bb$-divisors.

If $B=0$, or if $\Mm=\bm{0}$, or if $U$ is not important, then we may drop $B,\Mm,U$ respectively. If $U=\{pt\}$ then we also drop $U$ and say that $(X,\Ff,B,\Mm,t)$ is \emph{projective}. 
In particular, we refer to $(X,\mathcal F)$ as a \emph{projective foliated pair}. 
If we allow $B$ to have negative coefficients, then we shall add the prefix ``sub-". If $B$ is a $\Qq$-divisor, $\Mm$ is a $\Qq$-$\bb$-divisor, and $t\in\mathbb Q$, then we shall add the prefix ``$\Qq$-".
\end{defn}

\begin{nota}\label{nota: simplified notation of afs}
In many scenarios, the notation $(X,\Ff,B,\Mm,t)/U$ which represents adjoint foliated structures is too long and messes up the arguments. One important thing is that sometimes we do not need the detailed information of adjoint foliated structures. 
Considering that, in this paper, we shall usually write adjoint foliated structures in the form of ``$\Aa:=(X,\Ff,B,\Mm,t)$" or we shall simply say that ``$\Aa/U$ is an adjoint foliated structure" without mentioning $X,\Ff,B,\Mm,t$ at all.  
For any $\Rr$-divisor $D$ on $X$ and nef$/U$ $\bb$-divisor $\Nn$ on $X$ such that $D+\Nn_X$ is $\Rr$-Cartier, we denote by $(\Aa,D,\Nn):=(X,\Ff,B+D,\Mm+\Nn,t)$. If $D=0$ then we may drop $D$, and if $\Nn=\bm{0}$ then we may drop $\Nn$.

Given an adjoint foliated structure $\Aa/U$ with $\Aa=(X,\Ff,B,\Mm,t)$, we define
$$K_{\Aa}:=K_{(X,\Ff,B,\Mm,t)}:=tK_{\Ff}+(1-t)K_X+B+\Mm_X$$
to be the \emph{canonical $\Rr$-divisor} of $\Aa$. 
Moreover, 
$X,\Ff,B,\Mm,t$ are called the \emph{ambient variety}, \emph{foliation part}, \emph{boundary part}, \emph{nef part} (or \emph{moduli part}), and \emph{parameter} of $\Aa$ respectively. 
The divisor
$B+\Mm_X$ is called the \emph{generalized boundary} of $\Aa$. We say that $\Aa/U$ is \emph{of general type} if $K_{\Aa}$ is big$/U$. 

For any birational map$/U$ $\phi: X\dashrightarrow X'$ which does not extract any divisor and any adjoint foliated structure $\Aa/U$ on $X$, we define $\phi_*\Aa/U$ to be the adjoint foliated structure such that $\phi_*\Aa:=(X',\phi_*\Ff,\phi_*B,\Mm,t)$ and say that $\phi_*\Aa$ is the \emph{image} of $\Aa$ on $X'$. 

For any projective birational morphism $h: X'\rightarrow X$, we define 
$$h^*\Aa:=(X',\Ff',B',\Mm,t)$$
where $\Ff':=h^{-1}\Ff$ and $B'$ is the unique $\Rr$-divisor such that $K_{h^*\Aa}=h^*K_{\Aa}$. For any prime divisor $E$ on $X'$, we denote by
$$a(E,\Aa):=-\mult_EB'$$
the \emph{discrepancy} of $E$ with respect to $\Aa$. The \emph{total minimal log discrepancy} of $\Aa$ is
$$\tmld(\Aa):=\inf\{a(E,\Aa)+t\epsilon_{\Ff}(E)+(1-t)\mid E\text{ is over }X\}.$$
One can easily shows that $\tmld(\Aa)\in [0,+\infty)\cup\{-\infty\}$ by using similar arguments as \cite[Remark 2.3]{CS21}.

For any non-negative real number $\epsilon$, we say that $\Aa$ is \emph{$\epsilon$-lc} (resp. \emph{$\epsilon$-klt}) if $\tmld(\Aa)\geq\epsilon$ (resp. $>\epsilon$). We say that $\Aa$ is \emph{lc} (resp. \emph{klt}) if $\Aa$ is $0$-lc (resp. $0$-klt). If we allow $B$ to have negative coefficients, then we add the prefix ``sub-" for the types of singularities above.

An \emph{nklt place} of $\Aa$ is a prime divisor $E$ over $X$ such that $a(E,\Aa)=-t\epsilon_{\Ff}(E)-(1-t)$. An \emph{nklt center} of $\Aa$ is the image of an nklt place of $\Aa$ on $X$.

For any sub-adjoint foliated structures $\Aa_i/U=(X,\Ff,B_i,\Mm_i,t_i)/U$ and real numbers $a_i\in [0,1]$ such that $\sum a_i=1$, we denote by
$$\sum a_i\Aa_i:=\left(X,\Ff,\sum a_iB_i,\sum a_i\Mm_i,\sum a_it_i\right).$$
\end{nota}

\begin{defn}[Potentially klt]\label{defn: potentially klt}
Let $X$ be a normal quasi-projective variety. We say that $X$ is \emph{potentially klt} if $(X,\Delta)$ is klt for some $\Rr$-divisor $\Delta\geq 0$. 
\end{defn}

We will often use the following two results in this paper:

\begin{thm}[{\cite[Theorem 1.10(1)]{CHLMSSX24}}]\label{thm: klt afs implies potentially klt}
    Let $\Aa/U$ be a klt algebraically integrable adjoint foliated structure with ambient variety $X$. Then $X$ is potentially klt. In particular, there exists a small $\mathbb Q$-factorialization of $X$.
\end{thm}

\begin{prop}[{\cite[Proposition 3.3]{CHLMSSX24}}]\label{prop: lc Aat implies lc Aa0}
    Let $(X,\Ff,B^{\ninv}+(1-t)B^{\inv},\Mm,t)/U$ be a $\mathbb Q$-factorial lc (resp. klt) algebraically integrable adjoint foliated structure such that $t<1$. Then $(X,B,\Mm)$ is lc (resp. klt).
\end{prop}

\begin{defn}[{cf. \cite[3.2 Log canonical foliated pairs]{ACSS21}, \cite[Definition 6.2.1]{CHLX23}}]\label{defn: foliated log smooth}
Let $\Aa/U:=(X,\Ff,B,\Mm,t)/U$ be an algebraically integrable adjoint foliated structure. We say that $\Aa$ is \emph{foliated log smooth} if there exists a contraction $f: X\rightarrow Z$ satisfying the following.
\begin{enumerate}
  \item $X$ has at most quotient singularities.
  \item $\Ff$ is induced by $f$.
  \item $(X,\Sigma_X)$ is toroidal for some reduced divisor $\Sigma_X$ such that $\Supp B\subset\Sigma_X$.  In particular, $(X,\Supp B)$ is toroidal, and $X$ is $\Qq$-factorial klt.
  \item There exists a log smooth pair $(Z,\Sigma_Z)$ such that $$f: (X,\Sigma_X)\rightarrow (Z,\Sigma_Z)$$ is an equidimensional toroidal contraction.
  \item $\Mm$ descends to $X$.
\end{enumerate}
We say that $f: (X,\Sigma_X)\rightarrow (Z,\Sigma_Z)$ is \emph{associated with} $(X,\Ff,B,\Mm,t)$. It is important to remark that $f$ may not be a contraction$/U$. In particular, $\Mm$ may not be nef$/Z$.

Note that the definition of foliated log smooth has nothing to do with $t$. In other words, as long as $(X,\Ff,B,\Mm)$ is foliated log smooth, then $(X,\Ff,B,\Mm,t)$ is foliated log smooth for any 
$t \in [0, 1]$.
\end{defn}

\begin{defn}[Foliated log resolutions]\label{defn: log resolution}
Let $\Aa/U:=(X,\Ff,B,\Mm,t)/U$ be an algebraically integrable sub-adjoint foliated structure. A \emph{foliated log resolution} of $\Aa$ is a birational morphism $h: X'\rightarrow X$ such that 
$$(X',\Ff':=h^{-1}\Ff,B':=h^{-1}_\ast B+\Exc(h),\Mm,t)$$ 
is foliated log smooth. By \cite[Lemma 6.2.4]{CHLX23}, foliated log resolution for $\Aa$ always exists.
\end{defn}

\subsection{Birational maps in the MMP}

\begin{defn}
    Let $X\rightarrow U$ be a projective morphism from a normal quasi-projective variety to a variety.  Let $D$ be an $\Rr$-Cartier $\Rr$-divisor on $X$ and $\phi: X\dashrightarrow X'$ a birational map$/U$. Then we say that $X'$ is a \emph{birational model} of $X$. We say that $\phi$ is $D$-non-positive (resp. $D$-negative, $D$-trivial, $D$-non-negative, $D$-positive) if the following conditions hold:
    \begin{enumerate}
    \item $\phi$ does not extract any divisor.
    \item $D':=\phi_\ast D$ is $\Rr$-Cartier.
    \item There exists a resolution of indeterminacy $p: W\rightarrow X$ and $q: W\rightarrow X'$, such that
    $$p^\ast D=q^\ast D'+F$$
    where $F\geq 0$ (resp. $F\geq 0$ and $\Supp p_\ast F=\Exc(\phi)$, $F=0$, $0\geq F$, $0\geq F$ and $\Supp p_\ast F=\Exc(\phi)$).
    \end{enumerate}
\end{defn}

\begin{defn} Let $X\rightarrow U$ be a projective morphism from a normal quasi-projective variety to a variety. Let $D$ be an $\Rr$-Cartier $\Rr$-divisor on $X$, $\phi: X\dashrightarrow X'$ a $D$-non-positive birational map$/U$, and $D':=\phi_\ast D$. Assume that $D'$ is nef$/U$.
\begin{enumerate}
    \item We say that $X'$ is a \emph{weak lc model}$/U$ of $D$.
    \item If $\phi$ is $D$-negative, then we say that $X'$ is a \emph{minimal model}$/U$ of $D$.
    \item If $D'$ is semi-ample$/U$, then we say that $X'$ is a \emph{semi-ample model}$/U$ of $D$.
    \item If $\phi$ is $D$-negative and $D'$ is semi-ample$/U$, then we say that $X'$ is a \emph{good minimal model}$/U$ of $D$.
\end{enumerate}
\end{defn}

\begin{defn}\label{defn: minimal model}
Let $\Aa/U$ be an adjoint foliated structure with ambient variety $X$ and let $\phi: X\dashrightarrow X'$ be a birational map$/U$. Let $\Aa':=\phi_*\Aa$. We say that $\Aa'/U$ is weak lc model (resp. minimal model, semi-ample model, good minimal model) of $\Aa/U$ if $X'$ is a weak lc model$/U$ (resp. minimal model$/U$, semi-ample model$/U$, good minimal model$/U$) of $K_{\Aa}$.
\end{defn}

We will use the following lemmata many times:

\begin{lem}\label{lem: weak lc model only check codim 1}
 Let $\Aa/U$ be an adjoint foliated structure with ambient variety $X$ and let $\phi: X\dashrightarrow X'$ be a birational map$/U$ which does not extract any divisor. Let $\Aa':=\phi_*\Aa$. Assume that $K_{\Aa'}$ is nef$/U$. Then:
 \begin{enumerate}
     \item If $a(D,\Aa)\leq a(D,\Aa')$ for any prime divisor $D$ on $X$ that is exceptional$/X'$, then $\Aa'/U$ is a weak lc model of $\Aa/U$.
     \item If $a(D,\Aa)<a(D,\Aa')$ for any prime divisor $D$ on $X$ that is exceptional$/X'$, then $\Aa'/U$ is a minimal model of $\Aa/U$.
 \end{enumerate}
\end{lem}
\begin{proof}
Let $p: W\rightarrow X$ and $q: W\rightarrow X'$ be a common resolution and write
$$p^*K_{\Aa}:=q^*K_{\Aa'}+E$$
for some $E$.
(1) We only need to show that $E\geq 0$. For any prime divisor $D$ that is an irreducible component of $E$, $$\mult_DE=a(D,\Aa')-a(D,\Aa).$$ 
By our assumption, $p_*E\geq 0$ and
$\Supp p_*E$ is the union of all prime divisors $D$ on $X$ such that $a(D,\Aa')\not=a(D,\Aa)$. Since $K_{\Aa'}$ is nef$/U$, $q^*K_{\Aa'}$ is nef$/X$, hence $E$ is anti-nef$/X$. By the negativity lemma, $E\geq 0$.

(2) By (1), $E\geq 0$. Since $\Supp p_*E$ is the union of all prime divisors $D$ on $X$ such that $a(D,\Aa')\not=a(D,\Aa)$, by our assumption, $\Supp p_*E\supset\Exc(\phi)$. Since $\phi_*\Aa=\Aa'$, $\Supp p_*E\subset\Exc(\phi)$. Thus $\Supp p_*E=\Exc(\phi)$ and we are done.
\end{proof}

\begin{lem}\label{lem: g-pair version bir12 2.7}
Let $\Aa/U$ be an lc adjoint foliated structure and let $\Aa_1/U$, $\Aa_2/U$ be two weak lc models of $\Aa/U$ with induced birational map $\phi: X_1\dashrightarrow X_2$. Let $h_1: W\rightarrow X_1$ and $h_2: W\rightarrow X_2$ be two birational morphisms such that $\phi\circ h_1=h_2$. Then:
\begin{enumerate}
    \item $$h_1^*K_{\Aa_1}=h_2^*K_{\Aa_2}.$$
    \item If $K_{\Aa_2}$ is semi-ample$/U$, then $K_{\Aa_1}$ is semi-ample$/U$.
    \item If $K_{\Aa_2}$ is ample$/U$, then $\phi$ is a morphism.
\end{enumerate}
\end{lem}
\begin{proof}
Let $\phi_1: X\dashrightarrow X_1$ and $\phi_2: X\dashrightarrow X_2$ be the induced birational maps. Possibly replacing $W$ with a higher model, we may assume that the induced birational map $h: W\rightarrow X$ is a morphism. Let $$E_i:=h^*K_{\Aa}-h_i^*K_{\Aa_i}$$
for $i\in\{1,2\}$. Then $E_i\geq 0$ and is exceptional over $X_i$ for $i\in\{1,2\}$. Thus $h_{1,*}(E_2-E_1)\geq 0$ and $E_1-E_2$ is nef$/X_1$, and $h_{2,*}(E_1-E_2)\geq 0$ and $E_2-E_1$ is nef$/X_2$. By the negativity lemma, $E_2-E_1\geq 0$ and $E_1-E_2\geq 0$. Thus $E_1=E_2$, which implies (1). (2) immediately follows from (1). By (1), if $K_{\Aa_2}$ is ample$/U$, then $h_2: W\rightarrow X_2$ is the ample model$/U$ of $h^*K_{\Aa_1}$, hence $\phi$ is the ample model$/U$ of $K_{\Aa_1}$. Since $K_{\Aa_1}$ is semi-ample$/U$, $\phi$ is a morphism. This implies (3).
\end{proof}

\subsection{Relative Nakayama-Zariski decompositions}

\begin{defn}
    Let $\pi\colon X\rightarrow U$ be a projective morphism from a normal quasi-projective variety to a quasi-projective variety, $D$ a pseudo-effective$/U$ $\Rr$-Cartier $\Rr$-divisor on $X$, and $P$ a prime divisor on $X$. We define $\sigma_{P}(X/U,D)$ as in \cite[Definition 3.1]{LX23} by considering $\sigma_{P}(X/U,D)$ as a number in  $[0,+\infty)\cup\{+\infty\}$. We define $N_{\sigma}(X/U,D)=\sum_Q\sigma_Q(X/U,D)Q$
    where the sum runs through all prime divisors on $X$ and consider it as a formal sum of divisors with coefficients in $[0,+\infty)\cup\{+\infty\}$. We say that $D$ is \emph{movable$/U$} if $N_{\sigma}(X/U,D)=0$.
\end{defn}

\begin{lem}[{\cite[Lemma 3.5]{LX23}}]\label{lem: finiteness of sigmap}
Let $\pi: X\rightarrow U$ be a projective morphism from a normal quasi-projective variety to a quasi-projective variety and let $D$ be a pseudo-effective$/U$ $\Rr$-Cartier $\Rr$-divisor on $X$. Then $N_{\sigma}(X/U,D)$ has finitely many irreducible components. 
\end{lem}

\begin{lem}[{\cite[Lemma 2.25]{LMX24}}]\label{lem: nz for lc divisor}
Let $X\rightarrow U$ be a projective morphism from a normal quasi-projective variety to a quasi-projective variety and $\phi\colon X\dashrightarrow X'$ a birational map$/U$. Let $D$ be an $\Rr$-Cartier $\Rr$-divisor on $X$ such that $\phi$ is $D$-negative and $D':=\phi_\ast D$. Then:
\begin{enumerate}
    \item The divisors contracted by $\phi$ are contained in $\Supp N_{\sigma}(X/U,D)$.
    \item If $D'$ is movable$/U$, then $\Supp N_{\sigma}(X/U,D)$ is the set of all $\phi$-exceptional divisors.
    \end{enumerate}
\end{lem}

\begin{lem}\label{lem: if contract n then movable}
Let $X\rightarrow U$ be a projective morphism from a normal quasi-projective variety to a quasi-projective variety and $\phi: X\dashrightarrow X'$ a birational map$/U$ which does not extract any divisor. Let $D$ be an $\Rr$-Cartier $\Rr$-divisor on $X$ such that $D':=\phi_*D$ is $\Rr$-Cartier and $\phi$ contracts all divisors that are contained in  $\Supp N_{\sigma}(X/U,D)$. Then $D'$ is movable$/U$.
\end{lem}
\begin{proof}
Let $p: W\rightarrow X$ and $q\colon W\rightarrow X'$ be a common resolution. Then we have $p^*D+F=q^*D'+E$ for some $E,F\geq 0$ such that $E\wedge F=0$. Since $\phi$ does not extract any divisor, $E,F$ are exceptional$/X'$. By \cite[Lemma 3.4(3)]{LMX24}, 
$$N_{\sigma}(X'/U,D')=q_*N_{\sigma}(W/U,q^*D'+E)=\phi_*p_*N_{\sigma}(W/U,p^*D+F)=\phi_*N_{\sigma}(X/U,D)=0.$$
\end{proof}

\begin{lem}\label{lem: limit of nakayama-zariski decomposition}
    Let $X\rightarrow U$ be a projective morphism between normal quasi-projective varieties and let $C,D$ be two pseudo-effective$/U$ $\Rr$-Cartier $\Rr$-divisors on $X$. Assume that $N_{\sigma}(X/U,D)$ is an $\Rr$-divisor, i.e. it does not have $+\infty$ as a coefficient. Then there exists a positive real number $s_0$ and a reduced divisor $E$ satisfying the following.
    \begin{enumerate}
        \item $\Supp N_{\sigma}(X/U,C+sD)=E$ for any $0<s\leq s_0$.
        \item $\Supp N_{\sigma}(X/U,C)\subset E$.
    \end{enumerate}
\end{lem}
\begin{proof}
By Lemma \ref{lem: finiteness of sigmap}, there are only finitely many components of $N_{\sigma}(X/U,C)$ and $N_{\sigma}(X/U,D)$, so we may let
\begin{enumerate}
    \item $P_1,\dots,P_m$ be all components of $N_{\sigma}(X/U,C)$,
    \item  $Q_1,\dots,Q_n$ be all components of $N_{\sigma}(X/U,D)$ that are not components of $N_{\sigma}(X/U,C)$ but are components of $N_{\sigma}(X/U,C+sD)$ for any $s>0$, and
    \item $R_1,\dots,R_l$ be all components of $N_{\sigma}(X/U,D)$ that are not components of $N_{\sigma}(X/U,C)$, and for any $1\leq k\leq l$, $R_k$ is not a component $N_{\sigma}(X/U,C+t_kD)$ for some $t_k>0$.
\end{enumerate}
By \cite[Lemma 3.3(2)]{LX23}, we have
$$N_{\sigma}(X/U,C+sD)\leq N_{\sigma}(X/U,C)+sN_{\sigma}(X/U,D),$$
hence for any $0\leq s<\min\{t_k\}$, all possible divisors that are contained in $\Supp N_{\sigma}(X/U,C+sD)$ are $P_1,\dots,P_m,Q_1,\dots,Q_n$.

By \cite[Lemma 3.3(3)]{LX23}, we have
$$\lim_{\epsilon\rightarrow 0^+}\sigma_{P_i}(X/U,C+\epsilon D)=\sigma_{P_i}(X/U,C).$$
Therefore, there exists $s_0>0$,  such that for any $i$ such that $\sigma_{P_i}(X/U,C)>0$, $\sigma_{P_i}(X/U,C+sD)>0$ for any $s\in [0,s_0]$. Possibly replacing $s_0$ with a smaller number, we may assume that $s_0<t_k$ for any $k$. 

In this case, $s_0$ and $E:=\sum_{i=1}^m P_i+\sum_{j=1}^n Q_j$ satisfy our requirements.
\end{proof}

\begin{lem}\label{lem: minimal model implies nsigma not infinity}
    Let $X\rightarrow U$ be a projective morphism between normal quasi-projective varieties and let $D$ be a pseudo-effective$/U$ $\Rr$-Cartier $\Rr$-divisor on $X$. Assume that $D$ has a minimal model$/U$. Then $N_{\sigma}(X/U,D)$ is an $\Rr$-divisor, i.e. it does not have $+\infty$ as a coefficient.
\end{lem}
\begin{proof}
Let $\phi\colon X\dashrightarrow X'$ be the minimal model$/U$ of $D$ and let $D':=\phi_*D$. Let $p: W\rightarrow X$ and $q: W\rightarrow Y$ be a common resolution. Then
$$p^*D=q^*D'+F$$
for some $F\geq 0$. Since $D'$ is nef$/U$, by \cite[Lemma 3.3(1)(2)]{LX23},
$$N_{\sigma}(p^*D)=N_{\sigma}(q^*D+F)\leq N_{\sigma}(q^*D)+N_{\sigma}(F)\leq F.$$
By \cite[Lemma 3.4(3)]{LX23}, $N_{\sigma}(D)\leq p_*F$. The lemma follows.
\end{proof}

\subsection{Perturbation results}

\begin{lem}[{\cite[Lemma 2.14]{CHLMSSX24}}]\label{lem: foliated log smooth imply lc}
Let $(X,\Ff,B,\Mm,t)/U$ be a foliated log smooth algebraically integrable sub-adjoint foliated structure. Then $(X,\Ff,\Supp B^{\ninv}+(1-t)\Supp B^{\inv},\Mm,t)$ is lc.
\end{lem}

The following lemma is a variation of \cite[Lemma 3.2]{CHLMSSX24}.

\begin{lem}\label{lem: add divisor not containing lc center}
Let $\Aa/U$ be a sub-lc (resp. sub-klt) algebraically integrable sub-adjoint foliated structure such that $t<1$. Let $D\geq 0$ be an $\Rr$-Cartier $\Rr$-divisor on $X$, such that for any nklt center $V$ of $\Aa$, $V\not\subset\Supp D$. Then there exists $\epsilon>0$ such that $(\Aa,\epsilon D)$ is sub-lc (resp. sub-klt)
\end{lem}
\begin{proof}
    Let $h: W\rightarrow X$ be a foliated log resolution of $(\Aa,lD)$ for some real number $l$ general in $\mathbb R$.
    Possibly replacing $\Aa$ with $h^*\Aa$ and $D$ with $h^*D$, we may assume that $\Aa$ is foliated log smooth. Let $B$ be the boundary of $\Aa$. By Lemma \ref{lem: foliated log smooth imply lc}, for any positive real number $s$, $(\Aa,sD)$ is lc if and only if
    $$\mult_P(B+sD)\leq t\epsilon_{\Ff}(P)+(1-t)$$
    for any irreducible component $P$ of $\Supp B\cup\Supp D$. Since $D$ does not contain any lc center of $\Aa$, $\mult_PD=0$ if $\mult_PB=t\epsilon_{\Ff}(P)+(1-t)$. Since there are only finitely many  irreducible components of $\Supp B\cup\Supp D$, the lemma follows.
\end{proof}

\begin{lem}\label{lem: perturbation for q-factorial klt big boundary to +ample}
Let $\Aa/U:=(X,\Ff,B,\Mm,t)/U$ be a $\mathbb Q$-factorial klt algebraically integrable adjoint foliated structure such that $B+\Mm_X$ is big$/U$ and let $A$ be a nef$/U$ $\Rr$-divisor on $X$. Then there exist two positive rational numbers $\epsilon,\delta$ and a klt algebraically integrable adjoint foliated structure $\Aa'/U:=(X,\Ff,B',(1-\epsilon)\Mm,t)/U$, such that
\begin{enumerate}
\item $K_{\Aa}\sim_{\mathbb R,U}K_{\Aa'}+\delta A$, and
\item if $B,A$ and $\Mm_X$ have rational coefficients, then $B'$ has rational coefficients and $K_{\Aa}\sim_{\mathbb Q,U}K_{\Aa'}+\delta A.$
\end{enumerate}
\end{lem}
\begin{proof}
Write $B+\Mm_X\sim_{\mathbb R,U}\tau A+E$ for some $E\geq 0$ and a rational number $0<\tau\ll 1$. Moreover, if $B,A$ and $\Mm_X$ have rational coefficients, then we may assume that $B+\Mm_X\sim_{\mathbb Q,U}\tau A+E$ and $E$ is a $\mathbb Q$-divisor. By Lemma \ref{lem: add divisor not containing lc center}, there exists a rational number $0<\epsilon\ll 1$ such that $$\Aa':=(X,\Ff,B':=(1-\epsilon)B+\epsilon E,(1-\epsilon)\Mm,t)$$ is klt. We may let $\delta:=\tau\epsilon$.
\end{proof}

\begin{lem}\label{lem: combine lc afs}
    Let $\Aa_i:=(X,\Ff,B_i,\Mm_i,t_i)/U$ be a (sub-)lc (sub-)adjoint foliated structures with $i=1,\dots,k$ and let $a_i\in [0,1]$ real numbers such that $\sum a_i=1$. Then:
    \begin{enumerate}
   \item $\sum a_i\Aa_i$ is (sub-)lc.
   \item If $a_j>0$ and $\Aa_j$ is sub-klt for some $j$, then $\Aa$ is sub-klt.
    \end{enumerate}
\end{lem}
\begin{proof}
Let $t:=\sum a_it_i$.For any prime divisor $E$ over $X$, we have
\begin{align*}
a(E,\Aa)+(t\epsilon_{\Ff}(E)+(1-t))&=\sum a_ia(E,\Aa_i)+\sum a_i(t_i\epsilon_{\Ff}(E)+(1-t_i))\\
&=\sum a_i\left(a(E,\Aa_i)+(t_i\epsilon_{\Ff}(E)+(1-t_i))\right)\geq 0
\end{align*}
and the last inequality is a strict inequality if $a_j>0$ and $a(E,\Aa_j)+(t_i\epsilon_{\Ff}(E)+(1-t_j))>0$ for some $j$. The lemma follows.
\end{proof}

\begin{thm}[Bertini-type theorem]\label{thm: Bertini type theorem}
    Let $\Aa/U$ be a sub-lc (resp. sub-klt) algebraically integrable sub-adjoint foliated structure such that the parameter $t$ of $\Aa$ is not $1$. Let $D$ be a semi-ample$/U$ $\Rr$-divisor on $X$. Then:
    \begin{enumerate}
      \item  There exists an $\Rr$-divisor $0\leq L\sim_{\mathbb R,U}D$ such that $(\Aa,L)$ is sub-lc (resp. sub-klt).
      \item If $D$ is base-point-free$/U$, then for any integer $m\geq\frac{1}{1-t}$ and any general element $G\in |mD/U|$, $(\Aa,\frac{1}{m}G)$ is sub-lc  (resp. sub-klt).
    \end{enumerate}
\end{thm}
\begin{proof}
Let $h: X'\rightarrow X$ be a foliated log resolution of $\Aa$. Possibly replacing $\Aa$ with $h^*\Aa$ and $D$ with $h^*D$, we may assume that $\Aa$ is foliated log smooth. 

Let $\Aa:=(X,\Ff,B,\Mm,t)$. Then $(X,\Ff,B^{\ninv},\Mm)$ is sub-lc. Under the condition of (2), 
$$\left(X,B^{\ninv}+\frac{1}{1-t}B^{\inv}+\frac{1}{m(1-t)}G,\Mm\right)$$
is sub-lc (resp. sub-klt). By Lemma \ref{lem: combine lc afs}, $(\Aa,\frac{1}{m}G)$ is sub-lc (resp. sub-klt). 

To prove (1), we may write $D=\sum a_iD_i$ where each $a_i>0$ and each $D_i$ is base-point-free$/U$. (1) follows by repeatedly applying (2).
\end{proof}

\begin{lem}\label{lem: perturbation to q coefficient no m}
Let $\Aa/U:=(X,\Ff,B,\Mm,t)/U$ be an lc algebraically integrable adjoint foliated structure such that $X$ is potentially klt. Let $A$ be an ample$/U$ $\Rr$-divisor on $X$. Then there exist an ample$/U$ $\Rr$-divisor $H$ on $X$ and a klt algebraically integrable adjoint foliated structure $\Aa'/U:=(X,\Ff,B',t')/U$ such that $B'$ has rational coefficients, $t'\in\mathbb Q\cap [0,1)$, and
$$K_{\Aa}+A\sim_{\mathbb R,U}K_{\Aa'}+H.$$
\end{lem}
\begin{proof}
Let $(X,\Delta)$ be a klt pair. By Lemma \ref{lem: combine lc afs}, possibly replacing $\Aa$ with $(1-\epsilon)\Aa+\epsilon(X,\Ff,\Delta,\bm{0},0)$
and $A$ with $A+\epsilon(K_{\Aa}-(K_X+\Delta))$ for some $0<\epsilon\ll 1$ such that $t(1-\epsilon)$ is rational, we may assume that $t$ is rational and $\Aa$ is klt.

Let $h: Y\rightarrow X$ be a foliated log resolution of $\Aa$ and let $\Aa_Y:=h^*\Aa$. Since $\Mm_Y$ is nef$/U$, $\Mm_Y+\frac{1}{2}h^*A$ is big$/U$ and nef$/U$. Since $A$ is ample, $\BB_+(\Mm_Y+\frac{1}{2}h^*A)\subset\Exc(h)$. Thus for any integer $k\geq 1$, 
we may write
$$\Mm_Y+\frac{1}{2}h^*A\sim_{\mathbb R,U}A_{k,Y}+\frac{1}{k}E$$
for some $E\geq 0$ that is exceptional$/X$ and ample$/U$ $\mathbb R$-divisors $A_{k,Y}$. By Lemma \ref{lem: add divisor not containing lc center}, for any integer $k\gg 0$, $(\Aa_Y,\frac{1}{k}E)$ is sub-klt. By Lemma \ref{thm: Bertini type theorem}, there exists $0\leq L\sim_{\mathbb R,U}A_{k,Y}$, such that $(\Aa_Y,\frac{1}{k}E+L)$ is sub-klt. Let $B_Y$ be the boundary of $\Aa_Y$, $\Delta_Y:=B_Y+\frac{1}{k}E+L$, and write $\Delta_Y=\sum_{i=0}^n r_i\Delta_{Y,i}$ where each $\Delta_{Y,i}$ is a $\mathbb Q$-divisor, $r_0=1$, and $r_0,\dots,r_n$ are $\mathbb Q$-linearly independent. Since $tK_{\Ff_Y}+(1-t)K_Y+\Delta_Y\sim_{\mathbb R,X}0$ and $t$ is rational, by \cite[Lemma 5.3]{HLS24}, $\Delta_{Y,i}\sim_{\mathbb R,X}0$ for any $1\leq i\leq n$.

Let $r_i'$ be rational numbers such that $||r_i'-r_i||\ll 1$, then $(Y,\Ff_Y,r_0\Delta_{Y,0}+\sum_{i=1}^n r_i'\Delta_{Y,i},t)$ is foliated log smooth and sub-klt, hence $$\Aa':=\left(X,\Ff,h_*\left(r_0\Delta_{Y,0}+\sum_{i=1}^n r_i'\Delta_{Y,i}\right),t\right)$$
is klt. Moreover, 
$$K_{\Aa'}\sim_{\mathbb R,U}K_{\Aa}+\frac{1}{2}A+\sum_{i=1}^n(r_i'-r_i)h_*\Delta_{Y,i}.$$
Let $H:=\frac{1}{2}A-\sum_{i=1}^n(r_i'-r_i)h_*\Delta_{Y,i}$. Then $\Aa'$ and $H$ satisfy our requirements.
\end{proof}

\begin{lem}\label{lem: perturb gklt to klt}
    Let $(X,B,\Mm)/U$ be a $\mathbb Q$-factorial klt generalized pair such that $B+\Mm_X$ is big$/U$. Then there exists a klt pair $(X,\Delta)$ such that $K_X+\Delta\sim_{\mathbb R,U}K_X+B+\Mm_X$.
\end{lem}
\begin{proof}
    We write $B+\Mm_X=A+E$ where $A$ is ample$/U$ and $E\geq 0$. Take $0<\epsilon\ll 1$ such that $(X,(1-\epsilon)B+\epsilon E,(1-\epsilon)\Mm)$ is klt. By \cite[Lemma 3.4]{HL22}, there exists a klt pair $(X,\Delta)$ such that 
    $$K_X+\Delta\sim_{\mathbb R,U}K_X+(1-\epsilon)B+\epsilon E+(1-\epsilon)\Mm_X+\epsilon A\sim_{\mathbb R,U}K_X+B+\Mm_X.$$
\end{proof}

\subsection{Families of foliations}

\begin{defn}\label{d_bounded}
    We say that a set $\mathcal P$ of projective foliated pairs $(X,\mathcal F)$ is \emph{bounded} if there exist finitely many flat projective morphisms $f^i\colon X^i\to T^i$ of normal varieties with normal fibers and foliations $\mathcal G^i$ of rank $r_i$ on $X^i$ with $i=1,\dots,N$, such that 
    \begin{enumerate}
    \item for any closed point $t\in T^i$, if $X^i_t$ denotes the fiber of $f^i$ over $t$, then $T_{X^i/T^i}|_{X^i_t}\simeq T_{X^i_t}$; 
    \item $T_{\mathcal G^i}\subset T_{X^i/T^i}$ and for any closed point 
    $t\in T^i$, we have that $T_{\mathcal G^i}|_{X^i_t}\subset T_{X^i_t}$ defines a foliation ${\mathcal F^i_t}$ of rank $r_i$ on $X_t^i$; and 
    \item for any  $(X,\mathcal F)\in \mathcal P$, there exist $i=1,\dots,N$, a closed point $t\in T^i$ and an isomorphism $\phi\colon X\to X^i_t$ such that $\phi_*\mathcal F\simeq \mathcal F^i_t$.
    \end{enumerate}
\end{defn}

\begin{lem}
\label{lem_linear_algebra}
Let $k$ be a field and let $V$ be a $k$-vector space of dimension $n$.  Fix an integer $0<r\leq n$ and let $L \subset \bigwedge^r V$ be a one 
dimensional subspace spanned by a non-zero vector $w$.
Consider the map $\alpha\colon V \to \bigwedge^{r+1} V$ defined by $v \mapsto v\wedge w$.
Then:
\begin{enumerate}
\item $\bigwedge^r \ker \alpha \subset L$.  In particular, if $\dim \ker \alpha = r$ then $\bigwedge^r \ker \alpha = L$.

\item If $W\subset V$ is a $r$-dimensional subspace such that $\wedge^r W=L$ then $W=\ker \alpha$.
\end{enumerate}
\end{lem}
\begin{proof}
Note that if $v\in V$ is a non-zero vector such that $v\wedge w=0$, then there exists $w_1\in \bigwedge^{r-1} V$ such that $w= w_1\wedge v$. Thus, the Lemma follows easily. 
\end{proof}

\begin{lem}
Let $f\colon X \to T$ be a flat morphism between normal varieties with normal fibers.  Let $E$ be a coherent sheaf on $X$ and suppose that 
$E\vert_{X_t}$ is reflexive for all closed point $t \in T$, where $X_t$ denotes the fiber of $f$ over $t$.  Then:
\begin{enumerate}
    \item $E$ is a reflexive sheaf; and 
    \item if $S \to T$ is a morphism such that $X_S:= X\times_TS$ is normal, then $E_S := q^*E$ is reflexive,
    where $q\colon X_S \to X$ is the natural projection.
\end{enumerate}
\end{lem}
\begin{proof}
We first prove (1).  Since $E\vert_{X_t}$ is reflexive, it is locally free away from a subset of codimension at least 2.  It follows that there exists a closed subset $Z \subset X$ such that
$Z \cap X_t$ is of codimension at least 2 for all $t \in T$ and $E\vert_{X\setminus Z}$ is locally free. 
The assumption that $E\vert_{X_t}$ is reflexive implies that the natural morphism $f\colon E \to j_*(E\vert_{X\setminus Z})$
is an isomorphism when restricted to any fibre (here $j\colon X\setminus Z \to X$ is the inclusion).  We therefore deduce that $f$ is in fact an isomorphism, i.e. $E \cong j_*(E\vert_{X\setminus Z})$, which implies that $E$ is reflexive. 

    (2) is an easy consequence of (1).
\end{proof}

\begin{defn}
Let $X$ be a variety. A \emph{stratification} of $X$ is a finite set of pairwise disjoint locally closed subsets $X_1,\dots,X_N$ of $X$ such that $\bigcup_{i=1}^N X_i=X$. The subset $X_i$ is called  \emph{stratum} of $\tilde X$.
\end{defn}

Given a sheaf $E$ on a variety $X$ and a point $x\in X$, we denote by $E(x) := E \otimes_{\mathcal O_X} \kappa(x)$ the fiber of $E$ at $x$.

\begin{lem}
\label{lem_reconstruction}
Fix integers $0<r \leq s$. Let $f\colon X \to T$ be a flat projective morphism between normal varieties with normal fibers and let $Z \subset X$ be a closed subset 
such that $Z \cap X_t \subset X_t$ is codimension at least 2 for all closed point $t \in T$, where $X_t$ is the fiber of $f$ over $t$.  Set $X^\circ := X \setminus Z$ and assume that $X_t\cap X^\circ$ is smooth for all $t\in T$.  
Consider the following set-up:
\begin{itemize}
\item let $E$ (resp. $L$) be a reflexive sheaf of rank $s$
(resp. rank 1) on $X$ such that $E\vert_{X_t}$ (resp. $L\vert_{X_t}$) is a reflexive sheaf for all $t \in T$; 

\item suppose that  $E^\circ := E\vert_{X^\circ}$ (resp. 
$L^\circ := L\vert_{X^\circ}$) is a vector bundle of rank $s$ (resp. rank 1); and 
\item suppose that $L^\circ \subset \bigwedge^{r}E^\circ$ is a sub-line bundle.
\end{itemize} 
Consider the subset of $T$ given by 
\[T_r\coloneqq \{ t \in T \mid \text{ there exists a rank $r$ sub-bundle }  F^\circ_t \subset E\vert_{X^\circ_t} \text{ such that } 
\det F^\circ_t = L\vert_{X^\circ_t}\}.\]

Then $T_r$ is a closed subset.  Moreover, 
there exists a stratification of $T_r$ such that if $i\colon \tilde{T} \to T_r$ is a stratum and
$j\colon \tilde{X} := X\times_T\tilde{T} \to X$ is the  projection, then: 
\begin{enumerate}
    \item $\tilde{X}$ is normal; 
    \item there exists a reflexive sub-sheaf
$F \subset j^*E$ which is a sub-bundle away from $j^{-1}(Z)$; and 
\item $\det F = j^*L$.
\end{enumerate}
\end{lem}
\begin{proof}
We have a natural morphism $E^\circ \otimes L^\circ \to \bigwedge^{r+1}E^\circ$
which induces a morphism 
$\psi\colon E^\circ \to \bigwedge^{r+1}E^\circ \otimes (L^\circ)^*$.  Let 
\[W'_r :=\left\{ x \in X^\circ \middle| \dim \ker\left(\psi(x) \colon E^\circ(x) \to \left(\bigwedge^{r+1}E^\circ \otimes (L^\circ)^*\right)(x)\right) = r\right\}.\]
Notice that $W'_r \subset X^\circ$ is a locally closed subset. Let $W_r$ be its Zariski closure in $X$.

We claim that if $t \in T$ is a point such that $X_t \subset W_r$, then $t\in T_r$.
Indeed, let $\eta \in X_t$ be the generic point.  Then by Lemma \ref{lem_linear_algebra} (1), 
$F(\eta)\coloneqq \ker \psi(\eta) \subset E\vert_{X_t}(\eta)$
is a  $r$-dimensional subspace such that $\det F(\eta) = L\vert_{X_t}(\eta)$.  We define $F_t:=i_*F(\eta) \cap E\vert_{X_t}$ where 
$i\colon \{\eta\} \to X_t$ is the inclusion.
Note that $F_t \subset E\vert_{X_t}$ is a coherent saturated sub-sheaf, and is therefore reflexive.  
The fact that 
$\det F_t = L\vert_{X_t}$ 
follows because this equality may be checked at the generic point
of $X_t$.  
Finally, 
\cite[Lemma 1.20]{DPS94} implies that $F_t\vert_{X^\circ_t}$ is in fact a sub-bundle of $E\vert_{X^\circ_t}$.

Conversely, if $t \in T$ is a point such that there exists a sub-bundle $F^\circ_t \subset  E\vert_{X^\circ_t}$ with $\det F^\circ_t = L\vert_{X^\circ_t}$
then Lemma \ref{lem_linear_algebra} (2) implies that $\dim \ker \psi(\eta) = r$, where $\eta \in X_t$ is the generic point. Therefore, $X^\circ_t \subset W_r$.

It follows that $T_r$ is the largest subset of $T$ such that $f^{-1}(T_r) \subset W_r$, in particular it is a closed subset.

Now, consider a stratification of $T_r$ so that if $i\colon \tilde{T} \to T_r$ is a stratum and
$j\colon \tilde{X} := X\times_T\tilde{T}\to X$ is the projection then $\tilde X$ is normal.
It follows that  $F\coloneqq \ker \psi \subset E$ is a reflexive sub-sheaf of rank $r$.
By our previous observations, we see that $\det F\vert_{X^\circ_t} = L\vert_{X^\circ_t}$ for all $t \in T$, and so $\det F = j^*L$.  Likewise, since $F\vert_{X^\circ_t}$ 
is a sub-bundle for all $t \in T$, we deduce that $F \subset E$ is a sub-bundle away from $j^{-1}(Z)$.
\end{proof}

\begin{prop}
\label{prop_pfaff_to_fol}
Let $f\colon X \to T$ be a flat projective morphism between normal varieties with normal 
fibers and let $K$ be a $\mathbb Q$-Cartier divisor on $X$. 

Consider the set 
$\mathcal P$ of projective foliated pairs $(Y,\mathcal F)$ such that there exist a closed point $t\in T$ and an isomorphism $\phi\colon Y\to X_t$, where $X_t$ is the fiber of $f$ over $t$, such that $K_{\mathcal F}\sim \phi^*(K|_{X_t})$.

Then $\mathcal P$ is bounded. 
\end{prop}
\begin{proof}
Fix a positive integer $r$. We may assume that if $(Y,\mathcal F)\in \mathcal P$ then $\mathcal F$ has rank $r$.
Up to replacing $T$ by a stratum of a  stratification of $T$, we may  assume that
\begin{itemize}
\item $\mathcal O_X(-K)\vert_{X_t}$
and 
$(\bigwedge^rT_{X/T})^{**})\vert_{X_t}$ are reflexive sheaves for all $t \in T$;
\item $\HHom (\mathcal O_X(-K), (\bigwedge^rT_{X/T})^{**})\vert_{X_t}$ is a reflexive sheaf for all $t \in T$; and 
\item $E := f_*\HHom (\mathcal O_X(-K), (\bigwedge^rT_{X/T})^{**})$ is locally free. 
\end{itemize}

Let $P := \mathbb P(E)$, let 
$W := X \times_T P$, let $pr_1\colon W \to X$ be the projection  and let 
$\Psi\colon pr_1^*\mathcal O_X(-K) \to (pr_1^*\bigwedge^rT_{X/T})^{**}$ be the tautological morphism.
Note that $pr_1^*(\bigwedge^rT_{X/T})^{**} \simeq (\bigwedge^rT_{W/P})^{**}$ and $pr_1^*\mathcal O_X(-K) = \mathcal O_X(-K_P)$ where $K_P = pr_1^* K$. 

For any $s\in P$, we denote by $W_s$ the fiber of $W\to P$ over $s$. Consider the subset $P_1 \subset P$ given by 
\begin{align*}
P_1 :=\Bigg\{s \in P\Bigg| \Psi_s\colon \mathcal O_{W}(-K_P)\vert_{W_s} \to \left(\bigwedge^rT_{W/P}\right)^{**}\Bigg\vert_{W_s} \text{ is a sub-line bundle} \\  \text{away from a subset of codimension } \ge 2\Bigg\}.
\end{align*}
Note that $P_1$ is an open subset of $P$.
We may find a closed subset $Z \subset W_1\coloneqq W\times_PP_1$ such that $Z \cap W_s$ is codimension at least two for all $s \in P_1$ and 
\begin{itemize}
    \item $W_s \setminus Z$ is smooth; and 
    \item $\mathcal O_{W}(-K_P)\vert_{W_s\setminus Z} \to (\bigwedge^rT_{W_1/P_1})^{**})\vert_{W_s\setminus Z}$
    is a sub-line bundle.
    \end{itemize}
By Lemma \ref{lem_reconstruction}, we may construct a  stratification of $P_1$,
such that
 if $i\colon \tilde{T} \to P_1$ is a stratum, then there exist 
 a reflexive sheaf $F$ on $\tilde{W} := W_1\times_{P_1}\tilde{T}$
and a  morphism $F \to j^*T_{W_1/P_1}$
where $j\colon \tilde{W} \to W_1$ is the natural projection  such that 
\begin{itemize}
\item if $t \in T$ is a closed point and $\mathcal G$ is a rank $r$ foliation on $X_t$ such that
 $K_{\mathcal G} \sim K\vert_{X_t}$ then
taking the dual of the $r$-th wedge power 
of the inclusion $T_{\mathcal G} \subset T_{X_t}$ gives a map $\psi \colon \mathcal O_{X_t}(-K_{\mathcal G}) \to (\bigwedge^rT_{X_t})^{**}$; 
\item there exist $s \in i(\tilde T)$ and an isomorphism $\varphi\colon X_t\to W_s$ such that $K_{\mathcal G}\sim \varphi^*K_P\vert_{W_s}$ and $\psi = \Psi\vert_{X_t}$; and
\item  the restriction of the morphism 
$E \to j^*T_{X_1/P_1}$ to the fiber over $i^{-1}(s)$
 is precisely $T_{\mathcal G} \to T_{X_t}$.
\end{itemize}

Perhaps stratifying $\tilde{T}$ further, we may assume that the morphism 
$\tilde{W} \to \tilde{T}$ admits normal fibers and satisfies condition (1) of Definition \ref{d_bounded}.
To show that (2) holds and that the morphism $E \to j^*T_{W_1/P_1}$ defines a foliation, it suffices to show that
$j^*T_{W_1/P_1} \cong T_{\tilde{W}/\tilde{T}}$. Indeed, we have a natural morphism $T_{\tilde{W}/\tilde{T}} \to j^*T_{W_1/P_1}$ and since this map is an isomorphism fiberwise, it is an isomorphism.
The discussion of the previous paragraph shows that (3) holds.
\end{proof}

\section{Extraction of non-terminal places}\label{sec: extract non-terminal place}

The goal of this section is to prove Theorem \ref{thm: extract non-terminal place intro}, which immediately implies that the existence of a semi-ample model is equivalent to the existence of a $\mathbb Q$-factorial good minimal model for klt algebraically integrable adjoint foliated structures (Proposition \ref{prop: eowlm implies eomm}). We begin with the following weaker version of Theorem \ref{thm: extract non-terminal place intro}.

\begin{thm}\label{thm: extract non-terminal places}
    Let $\Aa/U:=(X,\Ff,B^{\ninv}+(1-t)B^{\inv},\Mm,t)/U$ be a klt algebraically integrable adjoint foliated structure and $\mathcal{S}$ a finite set of prime divisors over $X$ such that $a(D,\Aa)\leq 0$ for any $D\in\mathcal{S}$. Then there exists a projective birational morphism $f: Y\rightarrow X$ such that $Y$ is $\mathbb Q$-factorial, and the divisors contracted by $f$ are exactly the divisors in $\mathcal{S}$.
\end{thm}
\begin{proof}
By Theorem \ref{thm: klt afs implies potentially klt}, there exists a small $\mathbb Q$-factorialization $h: X'\rightarrow X$. Possibly replacing $\Aa$ with $h^*\Aa$, we may assume that $X$ is $\mathbb Q$-factorial. By induction on the number of divisors in $\mathcal{S}$, we may assume that $\mathcal{S}$ contains exactly one prime divisor $E$. 

By Proposition \ref{prop: lc Aat implies lc Aa0}, $(X,B,\Mm)$ is klt. If $a(E,X,B,\Mm)\leq 0$, then the existence of $f: Y\rightarrow X$ follows from \cite[Lemma 4.6]{BZ16}, so we may assume that $a(E,X,B,\Mm)>0$. Since $a(E,\Aa)\leq 0$, by linearity of discrepancies, $a(E,\Ff,B^{\ninv},\Mm)<0$. Since $\Aa$ is klt, possibly replacing $t$ with $t+\delta_1$ for some $0<\delta_1\ll 1$, we may assume that $a(E,\Aa)<0$. Let $b:=-a(E,\Aa)$, then $b>0$.

Let $h: W\rightarrow X$ be a foliated log resolution of $\Aa$ such that $E$ is on $X$ and there exists an anti-ample$/X$ $\Rr$-divisor $H_W\geq 0$ on $W$ such that $H_W\subset\Exc(h)$.

Since $\Aa$ is klt, there exists $\delta_2>0$ such that $a(D,\Aa)>-(t\epsilon_{\Ff}(D)+(1-t))(1-2\delta_2)$
for any prime divisor $D$ on $W$ that is exceptional$/X$.

We let $B_W,F_W$ be the unique $\Rr$-divisors on $W$ such that
\begin{itemize}
\item $B_W-h^{-1}_*B:=F_W$ is exceptional$/X$,
\item $\mult_DF_W=1-\delta_2$ if $D\not=E$ and $D$ is exceptional$/X$, and 
\item $(t\epsilon_{\Ff}(E)+(1-t))\mult_EF_W=b$.
\end{itemize}
Since $t<1$, possibly replacing $H_W$ with $\delta_3H_W$ for some $0<\delta_3\ll 1$, we may assume that $\lfloor B_W+H_W^{\ninv}+\frac{1}{1-t}H_W^{\inv}\rfloor=0$ and $B_W^{\ninv}\geq t(B_W^{\ninv}+H_W^{\ninv})$.

Let $\Ll:=\overline{-H_W}$ an let 
$$\Aa_{W,s}/X:=\left(W,\Ff_W,B_W^{\ninv}+(1-s)B_W^{\inv}+H_W^{\ninv}+\frac{1-s}{1-t}H_W^{\inv},\Mm+\Ll,s\right)\Bigg/X$$
for any $s\in [0,1]$. Then each $\Aa_{W,s}$ is foliated log smooth, and $\Aa_{W,s}$ is klt for any $s<1$.

Since $\Ll_X$ is ample$/X$, by \cite[Theorem 16.1.4]{CHLX23}, we may run a $K_{\Aa_{W,1}}$-MMP$/X$ with scaling of an ample divisor $\phi: W\dashrightarrow V$ which terminates with a $\mathbb Q$-factorial good minimal model $\Aa_{V,1}/X$ of $\Aa_{W,1}/X$ such that $\Aa_{V,0}$ is klt, where $\Aa_{V,s}:=\phi_*\Aa_{W,s}$ for any $s\in [0,1]$. Let $g: V\rightarrow X$ be the induced birational morphism. By our construction, we have
$$K_{\Aa_{V,t}}=g^*K_{\Aa}+\sum_D((t\epsilon_{\Ff}(D)+(1-t))(1-\delta_2)+a(D,\Aa))D:=g^*K_{\Aa}+G$$
where the sum runs through all $g$-exceptional prime divisors that are not $\phi_*E$. 
By our construction of $\delta_2$, $G\geq 0$, $\Supp G=\Exc(g)$ if $E$ is contracted by $\phi$, and $\Supp G$ is the closure of $\Exc(g)\backslash\{\phi_*E\}$ if $E$ is not contracted by $\phi$. 

Let $\Pp:=\overline{K_{\Aa_{V,1}}}$. Then $\Pp$ is nef$/X$. Since $\Aa_{V,0}$ is klt, $\Bb_{V}/X:=(\Aa_{V,0},\frac{t}{1-t}\Pp)/X$ is a $\mathbb Q$-factorial generalized pair, and $K_{\Aa_{V,t}}=(1-t)K_{\Bb_V}.$ By \cite[Lemma 4.4]{BZ16} and \cite[Lemma 3.3]{Bir12}, we may run a $K_{\Bb_V}$-MMP$/U$ with scaling of an ample divisor that is also a $K_{\Aa_{V,t}}$-MMP$/U$ with scaling of an ample divisor, which terminates with a model $T$ such that the divisors contracted by the induced birational map $\psi: V\dashrightarrow T$ are exactly the divisors contained in $\Supp G$.

If $E$ is not contracted by $\phi$, then since $\phi_*E$ is not a component of $\Supp G$, $E$ is also not contracted by $\psi$, hence we may let $T:=Y$ and the induced birational morphism $f: Y\rightarrow X$ satisfies our requirements. Therefore, we may assume that $E$ is contracted by $\phi$. In this case, since $X$ is $\mathbb Q$-factorial and the induced birational map $T\rightarrow X$ does not contract 
any divisor, we have $T=X$ hence $\psi: V\rightarrow X$ is a morphism. Let
$$\Bb:=\left(X,B,\Mm+\Ll+\frac{t}{1-t}\Pp\right),$$
then since $H_W\subset\Exc(h)$, $\Bb=\psi_*\Bb_V$, hence $\Bb$ is a klt generalized pair. Let $\Aa':=\psi_*\Aa_{V,t}$. Note that $K_{\Aa'}=K_{\Aa}$ but $\Aa'\not=\Aa$ because of the $\bb$-divisor $\Ll$.

We set $b_1:=-a(E,\Aa')$, $b_2:=-a(E,\Bb)$, $c_1:=-a(E,\Aa_{V,t})$, and $c_2:=-a(E,\Bb_V)$. Since $\psi_*K_{\Bb_V}=K_{\Bb}$, $\psi_*K_{\Aa_{V,t}}=K_{\Aa'}$, and $K_{\Aa_{V,t}}=(1-t)K_{\Bb_V}$, we have
$$b_1-c_1=(1-t)(b_2-c_2).$$

Let $h: Z\rightarrow V$ be a projective birational morphism such that $E$ is on $Z$. Let $\Aa_{Z,s}:=h^{-1}_*\Aa_{V,s}$ for any $s\in [0,1]$ and $\Bb_Z:=h^{-1}_*\Bb_V$.
Then we have $\mult_E(h^*K_{\Aa_{V,t}}-K_{\Aa_{Z,t}})=c_1$ and $\mult_E(h^*K_{\Bb_V}-K_{\Bb_Z})=c_2$.  Since $E$ is contracted by $\phi$, we have
\begin{align*}
s:=&\mult_E(h^*\Pp_V-h^{-1}_*\Pp_V)=\mult_E(h^*K_{\Aa_{V,1}}-K_{\Aa_{Z,1}})\\
=&-a(E,\Aa_{V,1})<-a(E,\Aa_{W,1})=\mult_E(B_W^{\ninv}+H_W^{\ninv}).
\end{align*}
By linearity of discrepancies, we have
$$c_2=-a(E,\Aa_{V,0})=\frac{1}{1-t}(-a(E,\Aa_{V,t})+ta(E,\Aa_{V,1}))=\frac{1}{1-t}(c_1-ts),$$
hence
$$c_1=(1-t)c_2+ts.$$
This implies that
$$b_2=\frac{1}{1-t}(b_1-c_1)+c_2=\frac{1}{1-t}(b_1-ts).$$
By our construction of $B_W$ and $H_W$, we have
\begin{align*}
   b_1&=-a(E,X,\Ff,B^{\ninv}+(1-t)B^{\inv},\Mm+\Ll,t)\geq-a(E,X,\Ff,B^{\ninv}+(1-t)B^{\inv},\Mm,t)\\
   &=-a(E,\Aa)=\mult_E(B_W^{\ninv}+(1-t)B_W^{\inv})\geq\mult_EB_W^{\ninv}\geq t\mult_E(B_W^{\ninv}+H_W^{\ninv})>ts.
\end{align*}
Therefore, $b_2>0$. Since $\Bb$ is a klt generalized pair, the existence of $f$ follows from \cite[Lemma 4.6]{BZ16}.
\end{proof}

\begin{proof}[Proof of Theorem \ref{thm: extract non-terminal place intro}]
 Since $X$ is potentially klt, possibly replacing $X$ with a small $\mathbb Q$-factorialization, we may assume that $X$ is $\mathbb Q$-factorial.
   
Let $H\geq 0$ be an ample $\Rr$-divisor on $X$ such that $\Center_XD\subset\Supp H$ for any $D\in\mathcal{S}$ such that $a(D,\Aa)=0$ and $\Supp H$ does not contain any nklt center of $\Aa$. By Lemma \ref{lem: add divisor not containing lc center}, there exists $0<\epsilon\ll 1$ such that $(\Aa,\epsilon H)$ is lc. Possibly replacing $\Aa$ with $(\Aa,\epsilon H)$, we may assume that $a(D,\Aa)<0$ for any $D\in\mathcal{S}$. Possibly replacing $\Aa$ with $(X,\Ff,(1-\delta)B,(1-\delta)\Mm,t(1-\delta))$ for some $0<\delta\ll 1$, we may assume that $\Aa$ is klt. We are done by Theorem \ref{thm: extract non-terminal places}.
\end{proof}

\begin{prop}\label{prop: eowlm implies eomm}
  Let $\Aa/U$ be a klt algebraically integrable adjoint foliated structure. If $\Aa/U$ has a weak lc model (resp. a semi-ample model) then $\Aa/U$ has a $\mathbb Q$-factorial minimal model (resp. $\mathbb Q$-factorial good minimal model).
\end{prop}
\begin{proof}
    Let $\Aa'/U$ be a weak lc model (resp. semi-ample model) of $\Aa/U$ and let $X,X'$ be the ambient varieties of $\Aa,\Aa'$ respectively. Let 
    $$\mathcal{S}:=\{E\mid E\text{ is an exceptional}/X'\text{ prime divisor on }X, a(E,\Aa)=a(E,\Aa')\}.$$
    Then $\mathcal{S}$ is a finite set. Moreover, for any $E\in\mathcal{S}$, $a(E,\Aa')\leq 0$. By Theorem \ref{thm: extract non-terminal places}, there exists a birational morphism $f: Y\rightarrow X'$ such that $Y$ is $\mathbb Q$-factorial and the divisors contracted by $f$ are exactly the divisors in $\mathcal{S}$. Let $\phi: X\dashrightarrow Y$ be the induced birational map. By Lemma \ref{lem: weak lc model only check codim 1}, $\phi_*\Aa/U$ is a $\mathbb Q$-factorial minimal model (resp. $\mathbb Q$-factorial good minimal model) of $\Aa/U$.
\end{proof}

\section{Existence of good minimal models}\label{sec: eogmm}

The goal of this section is to prove the ``existence of the $\mathbb Q$-factorial good minimal models" part of the main theorem, Theorem \ref{thm: eogmm ai afs general case} (see Theorem \ref{thm: eogmm boundary big case}). Note that at this point we cannot prove that the good minimal models are obtained by running an MMP.

We shall also prove the first part of the base-point-freeness theorem, Theorem \ref{thm: bpf ai afs} (see Theorem \ref{thm: semi-ampleness theorem}) and the first part of the contraction theorem, Theorem \ref{thm: cont intro} (see Theorem \ref{thm: contraciton theorem potentially klt}).

The following theorem can be considered as an algebraically integrable adjoint foliated structure version of \cite[Theorems 1.2, 1.3]{CP19}, \cite[Theorem 1.2]{ACSS21}. 

\begin{thm}\label{thm: non-pseudo-effective for afs}
    Let $(X,\Ff,B,\Mm)/U$ be a $\mathbb Q$-factorial algebraically integrable generalized foliated quadruple and let $\Aa_s:=(X,\Ff,B^{\ninv}+(1-s)B^{\inv},\Mm,s)$ for any $s\in [0,1]$. Assume that $\Aa_{t}$ is lc and $K_{\Aa_{t}}$ is pseudo-effective$/U$ for some $t\in [0,1]$. Then $K_{\Aa_s}$ is pseudo-effective$/U$ for any $s\in [t,1]$.
\end{thm}
\begin{proof}
We only need to show that $K_{\Aa_1}$ is pseudo-effective$/U$ as $$K_{\Aa_s}=\frac{1-s}{1-t}K_{\Aa_t}+\frac{s-t}{1-t}K_{\Aa_1}.$$
Let $h: W\rightarrow X$ be a foliated log resolution of $(X,\Ff,B,\Mm)$ and let $$\Aa_{W,s}:=(h^{-1}_*\Aa_s,\Exc(h)^{\ninv}+(1-s)\Exc(h)^{\inv})$$
for any $s\in [0,1]$. Since $\Aa_t$ is lc,
$$K_{\Aa_{W,t}}\geq h^*K_{\Aa_t},$$
so $K_{\Aa_{W,t}}$ is pseudo-effective$/U$. Since $$h_*K_{\Aa_{W,1}}=K_{\Aa_1},$$
we may replace $\Aa_s$ with $\Aa_{W,s}$ for any $s$ and assume that each $\Aa_s/U$ is foliated log smooth. By \cite[Proposition 3.6]{ACSS21} and \cite[Proposition 7.3.6]{CHLX23}, there exists a divisor $G\geq\Supp B^{\inv}$ on $X$ and an equidimensional contraction $f: X\rightarrow Z$ such that $K_{\Ff}\sim_ZK_X+G$. 

Suppose for sake of contradiction  that $K_{\Aa_1}$ is not pseudo-effective$/U$. By \cite[Proposition 9.3.3]{CHLX23} we may run a $K_{\Aa_1}$-MMP$/U$ with scaling of an ample$/U$ $\Rr$-divisor such that the MMP is also an MMP$/Z$ and the MMP terminates with a Mori fiber space$/U$ $\pi: X'\rightarrow T$ which is also a Mori fiber space$/Z$. Let $\phi: X\dashrightarrow X'$ be the induced birational map, $B':=\phi_*B,G':=\phi_*G$, and let $\Aa'_t:=\phi_*\Aa_t$ for any $t\in [0,1]$, then 
$$K_{\Aa_1'}\sim_{\mathbb R,Z}K_{X'}+B'^{\ninv}+G+\Mm_{X'}.$$
Thus, $\pi$ is a $(K_{X'}+B'^{\ninv}+G'+\Mm_{X'})$-Mori fiber space$/U$, hence a 
$$(tK_{\Ff'}+(1-t)K_{X'}+B'^{\ninv}+(1-t)G'+\Mm_{X'})\text{-Mori fiber space}/U,$$
and so $tK_{\Ff'}+(1-t)K_{X'}+B'^{\ninv}+(1-t)G'+\Mm_{X'}$ is not pseudo-effective$/U$. Since $G\geq B^{\inv}$, we have that $G'\geq B'^{\inv}$, and hence $K_{\Aa'_t}$ is not pseudo-effective$/U$. Since $\phi$ does not extract any divisor, $K_{\Aa_t}$ is not pseudo-effective$/U$. This contradicts our assumption.
\end{proof}

The following theorem is the core of this paper.

\begin{thm}\label{thm: log smooth existence of good minimal model}
    Let $\Aa/U:=(X,\Ff,B^{\ninv}+(1-t)B^{\inv},\Mm+\overline{A},t)/U$ be a foliated log smooth klt algebraically integrable adjoint foliated structure, where $\Mm$ is nef$/U$ and $A$ is ample$/U$. Assume that $K_{\Aa}$ is pseudo-effective$/U$. 
    Then $\Aa/U$ has a $\mathbb Q$-factorial good minimal model.
\end{thm}
\begin{proof}
    Let $\Aa_s:=(X,\Ff,B^{\ninv}+(1-s)B^{\inv},\Mm+\overline{A},s)$ for any $s\in [0,1]$. By Theorem \ref{thm: non-pseudo-effective for afs}, $K_{\Aa_s}$ is pseudo-effective$/U$ for any $s\in [t,1]$.
Let $\mathcal P \subset [t, 1]$ be the set \[\{s \in [t, 1]: \Aa_s/U \text{ has a } \mathbb Q \text{-factorial good minimal model}\}.\]
    We will show that $t \in \mathcal P$ in two steps.

\medskip

\noindent\textbf{Step  1.} 
We first show that for $\mu \in (t, 1]$, if $\Aa_\mu/U$ has a $\mathbb Q$-factorial minimal model then 
$\Aa_{\mu-\epsilon}/U$ has a $\mathbb Q$-factorial good minimal model for all $0<\epsilon \ll 1$.

Consider the $\mathbb Q$-factorial minimal model of $\Aa_\mu/U$, which we will denote by $\phi\colon X \dashrightarrow X'$.  For any $s \in [0, 1]$ let $\Aa'_s:=\phi_*\Aa_s.$
Since $\phi$ is $K_{\Aa_{\mu}}$-negative, for any $0<\epsilon\ll 1$ and any prime divisor $E$ on $X$, we have
$$a(E,\Aa_{\mu-\epsilon})\leq a(E,\Aa'_{\mu-\epsilon})$$
and strict inequality holds if $E$ is exceptional$/X'$.
Since $\Aa'_{\mu}$ is klt, by Proposition \ref{prop: lc Aat implies lc Aa0}, $\Aa'_0/U$ is a klt generalized pair whose generalized boundary is big$/U$. Since $K_{\Aa'_{\mu}}$ is nef$/U$, by \cite[Lemma 4.4]{BZ16}, we may run a
$$\left(K_{\Aa'_0}+\frac{\mu-\epsilon}{\epsilon}K_{\Aa'_{\mu}}\right)\text{-MMP}/U$$
which terminates with a good minimal model$/U$ of $(\Aa_0',\frac{\mu-\epsilon}{\epsilon}\overline{K_{\Aa'_{\mu}}})/U$ with induced birational map $\psi: X'\dashrightarrow X''$. Since
$$K_{\Aa'_0}+\frac{\mu-\epsilon}{\epsilon}K_{\Aa'_{\mu}}\sim_{\mathbb R,U}\frac{\mu}{\epsilon}K_{\Aa'_{\mu-\epsilon}},$$
this MMP is also a $K_{\Aa'_{\mu-\epsilon}}$-MMP$/U$. In particular, $\psi_*\Aa'_{\mu-\epsilon}/U$ is a $\mathbb Q$-factorial good minimal model of $\Aa'_{\mu-\epsilon}/U$. Since $\phi$ is $K_{\Aa_{\mu-\epsilon}}$-negative, 
$$a(E,\Aa'_{\mu-\epsilon})\leq a(E,\psi_*\Aa'_{\mu-\epsilon})$$
for any prime divisor $E$ on $X'$, and strict inequality holds if $E$ is exceptional$/X''$, we deduce that $a(E,\Aa_{\mu-\epsilon})<a(E,\psi_*\Aa'_{\mu-\epsilon})$ for any prime divisor $E$ on $X$ that is exceptional$/X''$. By Lemma \ref{lem: weak lc model only check codim 1}(2), $\psi\circ\phi$ is $K_{\Aa_{\mu-\epsilon}}$-negative, hence $\psi_*\Aa'_{\mu-\epsilon}/U$ is a $\mathbb Q$-factorial good minimal model of $\Aa_{\mu-\epsilon}/U$.

\medskip

\noindent\textbf{Step 2.} 
   Suppose for $\mu \in [t, 1]$ that $\Aa_s/U$ has a $\mathbb Q$-factorial good minimal model for all $s \in (\mu, 1]$.  We will show that $\Aa_\mu/U$ has a $\mathbb Q$-factorial good minimal model.

Let $N_s:=N_{\sigma}(X/U,\Aa_s)$ for any $s\in [\mu,1]$. Since $\Aa_{s}/U$ has a good minimal model for any $s\in (\mu,1]$, by Lemma \ref{lem: minimal model implies nsigma not infinity}, $N_s$ is an $\Rr$-divisor for any $s\in (\mu,1]$. By Lemma \ref{lem: limit of nakayama-zariski decomposition}, there exists $s_0>\mu$ and a reduced divisor $E$ on $X$, such that $\Supp N_\mu\subset E$ and $\Supp N_s=E$ for any $s\in (\mu,s_0]$.

We let $\phi: X\dashrightarrow X'$ be a $\mathbb Q$-factorial good minimal model of $\Aa_{s_0}/U$ and let $\Aa'_s:=\phi_*\Aa_s$ for any $s\in [0,1]$. By Lemma \ref{lem: nz for lc divisor}, $E=\Exc(\phi)$. By Lemma \ref{lem: if contract n then movable}, $K_{\Aa_s'}$ is movable$/U$ for any $s\in [\mu,s_0]$.

Since $K_{\Aa'_{s_0}}$ is semi-ample$/U$, we may write $K_{\Aa'_{s_0}}=\sum a_iN_i$ where each $a_i>0$ and each $N_i$ is a base-point-free$/U$ divisor. Let $0<\delta_0\ll s_0-\mu$ be a real number, such that $a_i\frac{s_0-\delta_0}{\delta_0}>2\dim X$ for any $i$.

We run a $K_{\Aa'_{s_0-\delta_0}}$-MMP$/U$ with scaling of an ample$/U$ divisor. Since $\Aa'_{s_0}$ is klt, by Proposition \ref{prop: lc Aat implies lc Aa0}, $\Aa'_0/U$ is a klt generalized pair whose generalized boundary is big$/U$. Since
$$K_{\Aa'_0}+\frac{s_0-\delta_0}{\delta_0}K_{\Aa'_{s_0}}\sim_{\mathbb R,U}\frac{s_0}{\delta_0}K_{\Aa'_{s_0-\delta_0}},$$
this MMP is also a $(K_{\Aa'_0}+\frac{s_0-\delta_0}{\delta_0}K_{\Aa'_{s_0}})$-MMP$/U$ with scaling of an ample divisor, and is $K_{\Aa'_{s_0}}$-trivial by \cite[Lemma 4.4(3)]{BZ16}.  Thus the MMP terminates with a good minimal model $\Aa''_{s_0-\delta_0}/U$ of $\Aa'_{s_0-\delta_0}/U$. Let $\psi: X'\dashrightarrow X''$ be the induced birational map and let $\Aa''_s:=\psi_*\Aa'_s$ for any $s\in [0,1]$. Since $K_{\Aa'_{s_0-\delta_0}}$ is movable$/U$, by Lemma \ref{lem: nz for lc divisor}, $\psi$ does not contract any divisor, and $K_{\Aa_s''}$
is movable$/U$ for any $s\in [\mu,s_0]$. Since $\psi$ is $K_{\Aa'_{s_0}}$-trivial, $K_{\Aa''_{s_0}}$ is semi-ample$/U$ and $\Aa''_{s_0}$ is klt. By Proposition \ref{prop: lc Aat implies lc Aa0}, $\Aa''_{0}/U$ is a klt generalized pair whose generalized boundary is big$/U$.

We have that $$K_{\Aa''_0}+\frac{s_0-\delta_0}{\delta_0}K_{\Aa''_{s_0}}\sim_{\mathbb R,U}\frac{s_0}{\delta_0}K_{\Aa''_{s_0-\delta_0}}$$
is nef$/U$ and therefore we may
run a
$K_{\Aa''_0}$-MMP$/U$ with scaling of $K_{\Aa''_{s_0}}$. 
  Let $f_i: X_i\dashrightarrow X_{i+1}$ be the steps of this MMP where $X_1:=X''$ and with scaling constant $\lambda_i$. We claim that this MMP terminates. By Lemma \ref{lem: perturb gklt to klt}, there exists a klt pair $(X'',\Delta'')$ such that $\Delta''$ is big and $K_{X''}+\Delta''\sim_{\mathbb R,U}K_{\Aa''_0}$, so the MMP $f_i: X_i\dashrightarrow X_{i+1}$ is also a $(K_{X''}+\Delta'')$-MMP$/U$ with scaling of $K_{\Aa''_{s_0}}$. Since $K_{\Aa''_{s_0}}$ is semi-ample$/U$, the MMP terminates by \cite[Corollary 1.4.2]{BCHM10}.

Let $n$ be the minimal positive integer such that $\lambda_n\leq \frac{\mu}{\delta_0}$. Such $n$ exists because the MMP terminates. By definition of the MMP with scaling, if $\mu' = \delta_0\lambda_{n-1}$, the induced birational map $\tau: X''\dashrightarrow \bar X:=X_n$ is a $K_{\Aa''_s}$-MMP$/U$ with scaling of $K_{\Aa''_{s_0}}$ for any $s\in [\mu,\mu')$. Let $\bar\Aa_s:=\tau_*\Aa''_s$ for any $s\in [0,1]$. Since
$$K_{\Aa''_0}+\frac{s}{s_0-s}K_{\Aa''_{s_0}}\sim_{\mathbb R,U}\frac{s_0}{s_0-s}K_{\Aa''_{s}},$$
$\Aa_s''/U$ has a klt generalized pair structure with big$/U$ generalized boundary, hence by the base-point-freeness theorem for klt generalized pairs, $K_{\bar{\Aa}_{s}}$ is semi-ample$/U$ for any $s\in [\mu,\mu')$. Since $K_{\Aa''_s}$ is movable$/U$ for any $s\in [\mu,s_0]$,  by Lemma \ref{lem: nz for lc divisor}, $\tau$ does not contract any divisor. 

For any $s\in (\mu,\mu')$, let $\Aa^s_s/U$ be a good minimal model of $\Aa_s/U$ whose existence is guaranteed by our assumption. Let $\phi_s: X\dashrightarrow X^s$ be the induced birational map. By Lemma \ref{lem: nz for lc divisor}, $\Exc(\phi_s)=\Supp N_s=E$. Thus $X_s$ and $\bar X$ are isomorphic in codimension $1$. Let $p_s: W_s\rightarrow X^s$ and $q_s: W_s\rightarrow\bar X$ be a common resolution and let
$$p_s^*K_{\Aa^s_s}=q_s^*K_{\bar\Aa_s}+E_s,$$
then $E_s$ is exceptional$/X^s$ and exceptional$/\bar X$. Since $K_{\Aa^s_s}$ and $K_{\bar\Aa_s}$ are both nef$/U$, by applying the negativity lemma $p_s^*K_{\Aa^s_s}-q_s^*K_{\bar\Aa_s}$ with respect to $p_s$, we deduce $E_s \ge 0$.  Applying the negativity lemma to $p_s^*K_{\Aa^s_s}-q_s^*K_{\bar\Aa_s}$ with respect to $q_s$ we deduce that $E_s \le 0$.  Thus $E_s = 0$ and therefore $\bar\Aa_s/U$ is a weak lc model of $\Aa_s/U$ for any $s\in (\mu,\mu')$. Let $p: W\rightarrow X$ and $q: W\rightarrow\bar X$ be a common resolution, then for any $s\in (\mu,\mu')$,
$$p^*K_{\Aa_s}=q^*K_{\bar\Aa_s}+F_s$$
for some $F_s\geq 0$. Therefore,
$$p^*K_{\Aa_\mu}=\lim_{s\rightarrow 0^+}p^*K_{\Aa_s}=\lim_{s\rightarrow 0^+}(q^*K_{\bar\Aa_s}+F_s)=q^*K_{\bar\Aa_{\mu}}+\lim_{s\rightarrow 0^+}F_s$$
and $\lim_{s\rightarrow 0^+}F_s\geq 0$. 

Therefore, $\bar\Aa_{\mu}/U$ is a semi-ample model of $\Aa_{\mu}/U$. By Proposition \ref{prop: eowlm implies eomm}, $\Aa_{\mu}/U$ has a $\mathbb Q$-factorial good minimal model.

\medskip

By \cite[Theorem 1.5]{LMX24} we see that $1 \in \mathcal P$. Then \textbf{Steps 1} and \textbf{2} easily imply that $\mathcal P = [t, 1]$. Therefore $\Aa_t/U$ has a $\mathbb Q$-factorial good minimal model.
\end{proof}

\begin{thm}\label{thm: eogmm lc plus ample case}
    Let $\Aa/U:=(X,\Ff,B,\Mm,t)/U$ be an lc algebraically integrable adjoint foliated structure such that $X$ is potentially klt and let $A$ be an ample$/U$ $\Rr$-divisor on $X$ such that $K_{\Aa}+A$ is pseudo-effective$/U$. Then $(\Aa,\overline{A})/U$ has a $\mathbb Q$-factorial good minimal model.
\end{thm}
\begin{proof}
 Let $(X,\Delta)$ be a klt pair. Possibly replacing $\Aa$ with $(1-s)\Aa+s(X,\Ff,\Delta,\bm{0},0)$ and $A$ with $A+s(K_{\Aa}-(K_X+\Delta))$ for some $0<s\ll 1$, we may assume that $\Aa$ is klt and the parameter $t$ of $\Aa$ is $<1$.
 
 Let $h: X'\rightarrow X$ be a foliated log resolution of $\Aa$ such that there exists an anti-ample$/X$ $\Rr$-divisor $H\geq 0$ such that $\Supp H\subset\Exc(h)$. Let $0<\delta\ll 1$ be a real number such that $$a(D,\Aa)>-(1-\delta)(t\epsilon_{\Ff}(D)+(1-t))$$
for any prime divisor $D$ on $X'$ that is exceptional$/X$, and let $0<\epsilon\ll\delta$ be a real number such that $A':=h^*A-\epsilon H$ is ample$/U$ and 
$$\left\lfloor\frac{\epsilon}{\delta}\left(H^{\ninv}+\frac{1}{1-t}H^{\inv}\right)\right\rfloor=0.$$
Then
 $$\Aa':=(h^{-1}_*\Aa,(1-\delta)(\Exc(h)^{\ninv}+(1-t)\Exc(h)^{\inv})+\epsilon H)$$
 is a foliated log smooth klt algebraically integrable adjoint foliated structure and
 $$K_{\Aa'}+A'\sim_{\mathbb R,U}h^*(K_{\Aa}+A)+G$$
 for some $G\geq 0$ such that $\Supp G=\Exc(h)$. In particular, $K_{\Aa'}+A'$ is pseudo-effective$/U$. By Theorem \ref{thm: log smooth existence of good minimal model}, $(\Aa',\overline{A'})/U$ has a $\mathbb Q$-factorial good minimal model $(\Aa'',\overline{A'})/U$ with induced birational map $\phi: X'\dashrightarrow X''$. By Lemma \ref{lem: nz for lc divisor},
 $$\Exc(\phi)=\Supp N_{\sigma}(X'/U,K_{\Aa'}+A')=\Supp N_{\sigma}(X'/U,h^*(K_{\Aa}+A)+G)\supset\Supp G=\Exc(h),$$ so the induced birational map $\psi: X\dashrightarrow X''$ does not extract any divisor. Moreover, for any prime divisor $D$ on $X$ that is exceptional$/X''$, 
 $$a(D,(\Aa,\overline{A'}))=a(D,(\Aa',\overline{A'}))<a(D,(\Aa'',\overline{A'})).$$ 
 Therefore, $(\Aa'',\overline{A'})/U$ is a $\mathbb Q$-factorial good minimal model of $(\Aa,\overline{A'})/U$. In particular, $\psi$ is $(K_{\Aa}+(\overline{A'})_{X})$-negative. Since $h_*A'=A$, we have
 $$K_{\Aa}+A=K_{\Aa}+(\overline{A'})_{X}\text{ and }K_{\Aa''}+(\overline{A'})_{X''}=\psi_*(K_{\Aa}+A),$$
 hence $\psi$ is $(K_{\Aa}+A)$-negative and $\psi_*(K_{\Aa}+A)$ is nef. Since $(\Aa,\overline{A})$ is klt, $(\Aa'',\overline{A})$ is klt. Thus, $(\Aa'',\overline{A})/U$ is a good minimal model of $(\Aa,\overline{A})/U$. 
\end{proof}

 \begin{thm}\label{thm: eogmm boundary big case}
    Let $\Aa/U$ be a klt algebraically integrable adjoint foliated structure with big$/U$ generalized boundary and $K_{\Aa}$ is pseudo-effective$/U$. Then $\Aa/U$ has a $\mathbb Q$-factorial good minimal model.
\end{thm}
\begin{proof}
By Theorem \ref{thm: klt afs implies potentially klt}, possibly replacing $X$ with a small $\mathbb Q$-factorialization, we may assume that $X$ is $\mathbb Q$-factorial. By Lemma \ref{lem: perturbation for q-factorial klt big boundary to +ample}, $K_{\Aa}\sim_{\mathbb R}K_{\Aa'}+A$ for some klt algebraically integrable adjoint foliated structure $\Aa'/U$ and ample$/U$ $\Rr$-divisor $A$. By Theorem \ref{thm: eogmm lc plus ample case}, $(\Aa',\overline{A})/U$ has a good minimal model, hence $\Aa/U$ has a good minimal model.
\end{proof}

 \begin{thm}[{$=$Theorem \ref{thm: bpf ai afs}(1)}]\label{thm: semi-ampleness theorem}
    Let $\Aa/U=(X,\Ff,B,\Mm,t)/U$ be an lc algebraically integrable adjoint foliated structure and let $H$ be a nef$/U$ $\Rr$-Cartier $\Rr$-divisor such that $H-K_{\Aa}$ is ample$/U$ and $X$ is potentially klt. Then $H$ is semi-ample$/U$.
\end{thm}
\begin{proof}
Let $(X,\Delta)$ be a klt pair. Possibly replacing $\Aa$ with $(1-s)\Aa+s(X,\Ff,\Delta,\bm{0},0)$ for some $0<s\ll 1$, we may assume that $\Aa$ is klt. Let $A:=H-K_{\Aa}$ and let $\Bb:=(\Aa,\overline{A})$. By Theorem \ref{thm: eogmm lc plus ample case}, $\Bb/U$ has a $\mathbb Q$-factorial good minimal model $\Bb'/U$. Since $K_{\Bb}$ is nef$/U$, by Lemma \ref{lem: g-pair version bir12 2.7}, $K_{\Bb}\sim_{\mathbb R,U}H$ is  semi-ample$/U$.
\end{proof}

\begin{thm}[{$=$Theorem \ref{thm: cont intro}(1)}]\label{thm: contraciton theorem potentially klt}
     Let $\Aa/U$ be an lc algebraically integrable adjoint foliated structure such that $X$ is potentially klt. Let $\pi: X\rightarrow U$ be the associated morphism. Let $R$ be a $K_{\Aa}$-negative extremal ray$/U$. 
     
     Then there exists a
contraction$/U$ $\cont_R: X\rightarrow T$, such that for any integral curve $C$ such that $\pi(C)$ is a point, $\cont_R(C)$ is a point if and only $[C]\in R$.
\end{thm}
\begin{proof}
By the cone theorem \cite[Theorem 1.2]{CHLMSSX24}, $R$ is a $K_{\Aa}$-negative exposed ray$/U$. By \cite[Lemma 8.4.1]{CHLX23}, the supporting function of $R$ is $K_{\Aa}+A$ for some ample$/U$ $\Rr$-divisor $A$. By Theorem \ref{thm: semi-ampleness theorem}, $K_{\Aa}+A$ is semi-ample$/U$, so there exists a unique contraction $\cont_R: X\rightarrow T$ such that $K_{\Aa}+A\sim_{\mathbb R,U}\cont_R^*H$ for some ample$/U$ $\Rr$-divisor $H$. Thus, $\cont_R$ satisfies our requirements.
\end{proof}

\section{Finiteness of models}\label{sec: fom}

In this section we prove the finiteness of minimal models (Theorem \ref{thm: finiteness of minimal models intro}). The first part of Theorem \ref{thm: finiteness of minimal models intro} will be proven in Theorem \ref{thm: finiteness of log minimal model part i} and the second part of Theorem \ref{thm: finiteness of minimal models intro} will be proven in Theorem \ref{thm: finiteness of log minimal model part ii}. The ideas of the proofs can mainly be referred to \cite[Section 10]{HK10} as a simplified version of \cite[Section 7]{BCHM10}. Since the conditions of Theorem \ref{thm: finiteness of log minimal model part ii} are very complicated, we summarize them in the following set-up which we will use throughout this section. 

\begin{setup}\label{setup: finiteness of models}
$\pi, X,U,\MM,V,\Ff,t,\mathcal{C}$ are as follows:
\begin{enumerate}
    \item $\pi: X\rightarrow U$ is a projective morphism between normal quasi-projective varieties.
    \item $\MM$ is a tuple of nef$/U$ $\bb$-Cartier $\bb$-divisors.
    \item $V$ is a finite dimensional $\mathbb Q$-affine subspace of $\Weil_{\mathbb R}(X)\times\Span_{\mathbb R}(\MM)$.
    \item $\Ff$ is an algebraically integrable foliation on $X$.
    \item $t\in [0,1]$ is a rational number. 
    \item $\mathcal{C}\subset\mathcal{L}(V_{(X,\Ff,t)})$ is a rational polytope, such that for any $\Aa:=(X,\Ff,B,\Mm,t)\in\mathcal{C}$, $\Aa$ is klt and $B+\Mm_X$ is big$/U$.
\end{enumerate}
\end{setup}

\begin{lem}\label{lem: relative good minimal model neighborhood is good minimal model}
Notation and conditions as in Set-up \ref{setup: finiteness of models}.

Assume that $X$ is $\mathbb Q$-factorial and there exist $\Aa_0\in\mathcal{C}$ and a contraction$/U$ $f: X\rightarrow Z$, such that $K_{\Aa_0}\sim_{\mathbb R,U}f^*H$ for some ample$/U$ $\Rr$-divisor $H$. Let $\phi: X\dashrightarrow Y$ be a birational map$/Z$. Then there exists a neighborhood $\mathcal{C}_{\phi}$ of $\Aa_0$ in $\mathcal{C}$, such that for any $\Aa\in\mathcal{C}_{\phi}$, $\phi_*\Aa/U$ is a $\mathbb Q$-factorial good minimal model of $\Aa/U$ if and only if $\phi_*\Aa/Z$ is a $\mathbb Q$-factorial good minimal model of $\Aa/Z$. 
\end{lem}
\begin{proof}
Let $\Aa_{0,Y}:=\phi_*\Aa_0$. Since $\phi$ is $K_{\Aa_0}$-trivial, $\Aa_{0,Y}$ is klt. Since $V$ is finite dimensional, there exist a positive integer $n$ and $\Rr$-divisors $D_1,\dots,D_n$ on $X$, such that
\begin{itemize}
    \item for any $\Aa\in\mathcal{C}$, $K_{\Aa}=K_{\Aa_0}+\sum_{i=1}^n a_iD_i$ for some real numbers $a_i\geq 0$, and
    \item for any $1\leq i\leq n$, there exists $\Aa_i\in\mathcal{C}$ such that $K_{\Aa_i}=K_{\Aa_0}+D_i$.
\end{itemize}
By Lemma \ref{lem: add divisor not containing lc center}, possibly replacing $D_i$ with $\epsilon D_i$ and replacing $\mathcal{C}$ with $\{(1-\epsilon)\Aa_0+\epsilon\Aa\mid\Aa\in\mathcal{C}\}$ for some $0<\epsilon\ll 1$, we may assume that $\Aa_{i,Y}:=\phi_*\Aa_i$ is klt for any $i$. Let $\pi_Y: Y\rightarrow Z$ be the induced morphism. Since $H$ is ample$/U$, $\pi_Y^*H$ is semi-ample$/U$, so there exists $\delta\in (0,1)$, such that for any curve $C$ on $Y$, either $K_{\Aa_{0,Y}}\cdot C=0$ or $K_{\Aa_{0,Y}}\cdot C>\delta$.

Pick $$\Aa=\left(1-\frac{\delta}{2\dim X+1}\right)\Aa_0+\frac{\delta}{2\dim X+1}\sum_{i=1}^nb_i\Aa_i$$ where $b_i\in (0,1]$ and $\sum_{i=1}^nb_i=1$. Let $\Aa_Y:=\phi_*\Aa$, then $\Aa_Y$ is klt as it is contained in the convex hull spanned by $\Aa_{i,Y}$, $0\leq i\leq n$. For any $K_{\Aa_{Y}}$-negative extremal ray$/U$ $R$, by the cone theorem \cite[Theorem 1.3]{CHLMSSX24}, $R$ is spanned by a rational curve $C$. Possibly replacing $C$, we may assume that $C$ is of minimal length, i.e. for any rational curve $C'$ such that $[C']=R\in\overline{NE}(Y/U)$, $C'\equiv_U\lambda C$ for some $\lambda\geq 1$. By the length of extremal rays \cite[Theorem 1.3]{CHLMSSX24}, for any $i$, either $-2\dim X\leq K_{\Aa_{i,Y}}\cdot C<0$, or $K_{\Aa_{i,Y}}\cdot C\geq 0$. Therefore, if $K_{\Aa_{0,Y}}\cdot C>0$, then
\begin{align*}
    K_{\Aa_Y}\cdot C&=\left(1-\frac{\delta}{2\dim X+1}\right)(K_{\Aa_{0,Y}}\cdot C)+\frac{\delta}{2\dim X+1}\sum_{i=1}^nb_i(K_{\Aa_{i,Y}}\cdot C)\\
    &>\left(1-\frac{\delta}{2\dim X+1}\right)\cdot\delta-\frac{2\dim X}{2\dim X+1}\delta>0,
\end{align*}
which is not possible. Hence $K_{\Aa_{0,Y}}\cdot C=0$. In particular, if $K_{\Aa_Y}$ is nef$/Z$, then $K_{\Aa_Y}$ is nef$/U$. Therefore, we may take 
$$\mathcal{C}_{\phi}:=\left\{\left(1-\frac{\delta}{2\dim X+1}\right)\Aa_0+\frac{\delta}{2\dim X+1}\sum_{i=1}^nb_i\Aa_i\middle|b_i\in (0,1],\sum_{i=1}^nb_i=1\right\}.$$
\end{proof}

\begin{thm}\label{thm: finiteness of log minimal model part i}
Notation and conditions as in Set-up \ref{setup: finiteness of models}.

There are finitely birational maps$/U$ $\phi_j: X\dashrightarrow Y_j$, $1\leq j\leq k$, such that for any $\Aa\in\mathcal{C}\cap\mathcal{E}_{\pi}(V_{(X,\Ff,t)})$, there exists $1\leq j\leq k$ such that $(\phi_j)_*\Aa/U$ is a $\mathbb Q$-factorial good minimal model of $\Aa/U$.
\end{thm}
\begin{proof}
We apply induction on $\dim\mathcal{C}$. If $\dim\mathcal{C}=0$, then $\mathcal{C}$ consists of finitely many points, and the theorem follows from Theorem \ref{thm: eogmm boundary big case}. Since the cone of pseudo-effective$/U$ $\Rr$-divisors on $X$ is closed, $\mathcal{C}\cap\mathcal{E}_{\pi}(V_{(X,\Ff,t})$ is compact, so it suffices to work locally around $\Aa_0\in\mathcal{C}\cap\mathcal{E}_{\pi}(V_{(X,\Ff,t)})$. We are therefore free to replace $\mathcal{C}$ by the closure of an appropriate
neighborhood of $\Aa_0$ inside $\mathcal{C}$.

\medskip

\noindent\textbf{Step 1.} In this step we reduce to the case when $X$ is $\mathbb Q$-factorial and $K_{\Aa_0}$ is semi-ample$/U$.

By Theorem \ref{thm: eogmm boundary big case}, there exists a $\mathbb Q$-factorial good minimal model $\Aa_{0,Y}/U$ of $\Aa_0/U$ associated with a birational map $\phi\colon X\dashrightarrow Y$. Since $\phi$ is $K_{\Aa_0}$-negative, possibly shrinking $\mathcal{C}$, we may assume that $a(E,\Aa)<a(E,\phi_*\Aa)$ for any prime divisor $E$ on $X$ that is exceptional$/Y$ and any $\Aa\in\mathcal{C}$. Since $\Aa_{0,Y}$ is klt, possibly shrinking $\mathcal{C}$ further, we may assume that $\phi_*\Aa$ is klt for any $\Aa\in\mathcal{C}$.
Therefore, $\mathbb Q$-factorial good minimal model of $\phi_*\Aa/U$ is a $\mathbb Q$-factorial good minimal model of $\Aa/U$. Possibly replacing $(X,\Ff)$ with $(Y,\phi_*\Ff)$, $V$ with $\{\phi_*v\mid v\in V\}$, $\pi$ with the induced birational map $Y\rightarrow U$, $\Aa_0$ with $\Aa_{0,Y}$, and $\mathcal{C}$ with $\{\phi_*\Aa\mid\Aa\in\mathcal{C}\}$, we may assume that $X$ is $\mathbb Q$-factorial and $K_{\Aa_0}$ is semi-ample$/U$.

\medskip

\noindent\textbf{Step 2.}  In this step we reduce to the case when $K_{\Aa_0}\sim_{\mathbb R,U}0$.

Since $K_{\Aa_0}$ is semi-ample$/U$, there exists a contraction$/U$ $f: X\rightarrow Z$ such that $K_{\Aa_0}\sim_{\mathbb R,U}f^*H$ for some ample$/U$ $\Rr$-divisor $H$. In particular, $K_{\Aa_0}\sim_{\mathbb R,Z}0$.

Assume that there are finitely birational maps$/Z$ $\phi_j: X\dashrightarrow Y_j$ such that for any $\Aa\in\mathcal{C}\cap\mathcal{E}_{\pi}(V_{(X,\Ff,t)})$, there exists $j$ such that $(\phi_j)_*\Aa/U$ is a good minimal model of $\Aa/U$. By Lemma \ref{lem: relative good minimal model neighborhood is good minimal model}, possibly shrinking $\mathcal{C}$ to a neighborhood of $\Aa_0$, we may assume that for any $\phi_j$ and any $\Aa\in\mathcal{C}$, $(\phi_j)_*\Aa/U$ is a good minimal model of $\Aa/U$ if and only if $(\phi_j)_*\Aa/Z$ is a good minimal model of $\Aa/Z$. Thus, possibly shrinking $\mathcal{C}$ to a neighborhood of $\Aa_0$, these $\phi_j$ satisfy our requirements. Therefore, we may replace $\pi: X\rightarrow U$ with $f\colon X\rightarrow Z$ and assume that $K_{\Aa_0}\sim_{\mathbb R,U}0$.

\medskip

\noindent\textbf{Step 3.} In this step we conclude the proof. 

For any $\Aa\in\mathcal{C}$, there exists $\Aa'\in\partial\mathcal{C}$ (the boundary of $\mathcal{C}$) such that
$$\Aa-\Aa_0=\lambda (\Aa'-\Aa_0)$$
for some real number $\lambda>0$. Note that $\Aa-\Aa_0$ and $(\Aa'-\Aa_0)$ are considered as subsets of the affine space spanned by $\mathcal C$.
Thus, we have
$$K_{\Aa}\sim_{\mathbb R,U}\lambda K_{\Aa'}.$$
In particular, $\Aa\in\mathcal{E}_{\pi}(V_{(X,\Ff,t)})$ if and only if $\Aa'\in\mathcal{E}_{\pi}(V_{(X,\Ff,t)})$, and if $\bar\Aa/U$ is a $\mathbb Q$-factorial good minimal model of $\Aa/U$ with induced birational map $\phi: X\dashrightarrow\bar X$, then $\phi_*\Aa'/U$ is a good minimal model of $\Aa'/U$. Since the boundary of $\mathcal{C}$ is the intersection of finitely many $\mathbb Q$-affine hyperplanes,
we are done by induction on $\dim\mathcal{C}$.
\end{proof}

\begin{lem}\label{lem: first translation to greater than ample}
Notation and conditions as in Set-up \ref{setup: finiteness of models}.

Assume that $X$ is $\mathbb Q$-factorial and $\mathcal{C}$ is a simplex. Let $A$ be an ample$/U$ Cartier divisor on $X$ such that $\overline{A}\not\in\Span_{\mathbb R}(\MM)$, and let $\MM':=(\MM,\overline{A})$. 
Then there exist a positive rational number $\delta$, a finite dimensional subspace $V'\subset\Weil_{\mathbb R}(X)\times\Span_{\mathbb R}(\MM')$, and $\mathbb Q$-affine function $L: \mathcal{C}\rightarrow V'$ satisfying the following: 

For any $\Aa\in\mathcal{C}$, let $\Aa'=L(\Aa):=(X,\Ff,B',\Mm',t)\in V'$. Then $K_{\Aa'}\sim_{\mathbb R}K_{\Aa}$, $\Aa'$ is klt, and $\Mm'-\delta\overline{A}\in\Span_{\mathbb R_{\geq 0}}(\MM')$.
\end{lem}
\begin{proof}
    Let $\Aa_i:=(X,\Ff,B_i,\Mm_i,t)$ be the vertices of $\mathcal{C}$. By Lemma \ref{lem: perturbation for q-factorial klt big boundary to +ample}, there exist rational numbers $\epsilon_i,\delta_i$ and $\mathbb Q$-divisors $B_i'$ on $X$ such that $\Aa_i':=(X,\Ff,B_i',(1-\epsilon_i)\Mm_i+\delta_i\overline{A},t)$ is klt and $K_{\Aa_i}\sim_{\mathbb Q,U}K_{\Aa_i'}$. For any $\Aa\in\mathcal{C}$, there exist unique real numbers $a_i\in (0,1]$ such that $\sum a_i=1$ and $\sum a_i\Aa_i=\Aa$, and we define $L(\Aa):=\sum a_i\Aa_i'$. Let $\delta:=\min\{\delta_i\}$, $V_0\subset\Weil_{\mathbb R}(X)$ the subspace spanned by all components of $B_i'$, and $V':=V_0\times\Span_{\mathbb R}(\MM')$. Then $\delta$, $V'$ and $L$ satisfy our requirements.
\end{proof}

\begin{thm}\label{thm: finiteness of log minimal model part ii}
Notation and conditions as in Set-up \ref{setup: finiteness of models}.

Then there are finitely birational maps$/U$ $\phi_j: X\dashrightarrow Y_j$, $1\leq j\leq k$, such that for any $\Aa\in\mathcal{C}\cap\mathcal{E}_{\pi}(V_{(X,\Ff,t)})$ and any $\mathbb Q$-factorial good minimal model $\Aa_Y/U$ of $\Aa/U$ with induced birational map$/U$ $\phi\colon X\dashrightarrow Y$, there exists $1\leq j\leq k$ such that $(\phi_j)_*\Aa/U$ is a $\mathbb Q$-factorial good minimal model of $\Aa/U$ and $\phi_j\circ\phi^{-1}$ is an isomorphism.
\end{thm}
\begin{proof}
    Possibly replacing $X$ with a small $\mathbb Q$-factorialization, we may assume that $X$ is $\mathbb Q$-factorial. Possibly shrinking $\mathcal{C}$, we may assume that $\mathcal{C}$ is a simplex. 

    Let $A$ be a general very ample divisor on $X$. Then $\overline{A}\not\in\Span_{\mathbb R}(\MM)$. By Lemma \ref{lem: first translation to greater than ample}, possibly replacing $\mathcal{C}$, we may assume that there exists a positive real number $\delta$, such that $\MM=(\MM_0,\overline{A})$ for some $\MM_0$ with $\overline{A}\not\in\Span_{\mathbb R}(\MM)$, and $\Mm-\delta\overline{A}\in\Span_{\mathbb R_{\geq 0}}(\MM)$ for any $(X,\Ff,B,\Mm,t)\in\mathcal{C}$. 
    
    Let $S\geq 0$ be a reduced divisor such that the components of $S$ span $N^1(X/U)$. Let $0<\epsilon\ll\delta$ be a rational number such that $H:=\delta A-\epsilon S$ is ample$/U$, and $(X,\Ff,B+2\epsilon S,\Mm,t)$ is klt for any $(X,\Ff,B,\Mm,t)\in\mathcal{C}$ (which is possible as we only need $(X,\Ff,B_i+2\epsilon S,\Mm_i,t)$ is klt for the vertices $(X,\Ff,B_i,\Mm_i,t)$
    of $\mathcal{C}$). Let
    $$\mathcal{C}':=\{(X,\Ff,B+\epsilon S+C,\Mm-\delta A+\overline{H},t)\mid (X,\Ff,B,\Mm,t)\in\mathcal{C},\Supp C\subset\Supp S, ||C||\leq\epsilon\}.$$
    Then $\mathcal{C}'$ is a rational polytope contained in $\mathcal{L}(V'_{(X,\Ff,t)})$, where $$V'\subset\Weil_{\mathbb R}(X)\times\Span_{\mathbb R}(\MM,l\overline{H})$$ for some $l\gg 0$ sufficiently divisible, and for any $\Aa'=(X,\Ff,B',\Mm',t)\in\mathcal{C}'$, $\Aa'$ is klt and $B'+\Mm_{X}'$
    is big. 
    
    By Theorem \ref{thm: finiteness of log minimal model part i}, there are finitely many birational maps $\psi_i: X\dashrightarrow Z_i$, $1\leq i\leq k$, such that for any $\Aa'\in\mathcal{C}'$ such that $K_{\Aa'}$ is pseudo-effective$/U$, $(\psi_i)_*\Aa'/U$ is a $\mathbb Q$-factorial good minimal model of $\Aa'/U$.

    Now let $\{\phi_j: X\dashrightarrow Y_j\}_{j\in\Lambda}$ be all birational maps such that $(\phi_j)_*\Aa_j/U$ is a good minimal model of $\Aa_j/U$ for some $\Aa_j=(X,\Ff,B_j,\Mm_j,t)\in\mathcal{C}$ and let $H_j\geq 0$ be an ample$/U$ $\Qq$-divisor on $Y_j$. Then for any fixed $j$ and any $0<\tau\ll 1$, 
    $$\widehat{\Aa}_j:=\left(X,\Ff,B_j+\tau(\phi_j^{-1})_*H_j,\Mm_j,t\right)$$ is klt and $\phi_j$ is $K_{\widehat{\Aa}_j}$-negative. Since $(\phi_j)_*K_{\widehat{\Aa}_j}$ is ample$/U$, $(\phi_j)_*K_{\widehat{\Aa}_j}/U$ is the unique $\mathbb Q$-factorial good minimal model of $\widehat{\Aa}_j/U$. 

    We may write $(\phi_j^{-1})_*H_j\equiv_U\sum c_iC_i$ where $C_i$ are the irreducible components of $S$. Then for $0<\tau\ll 1$, 
    $$\bar\Aa_j:=\left(X,\Ff,B_j+\epsilon S+\tau\sum c_iC_i,\Mm_j-\delta A+\overline{H},t\right)\in\mathcal{C'}$$
and
$$K_{\bar\Aa_j}\equiv K_{\widehat{\Aa}_j},$$
hence $(\phi_j)_*K_{\bar{\Aa}_j}/U$ is the unique $\mathbb Q$-factorial good minimal model of $\bar{\Aa}_j/U$. Therefore, $\phi_j=\psi_i$ and $Y_j=Z_i$ for some $1\leq i\leq k$.
In particular, $\#\Lambda\leq k$. The theorem follows.
\end{proof}

\begin{proof}[Proof of Theorem \ref{thm: finiteness of minimal models intro}]
    (1) follows from Theorem \ref{thm: finiteness of log minimal model part i}. 
    
    (2) By (1), for any $\Aa\in\mathcal{C}\cap\mathcal{E}_{\pi}(V_{(X,\Ff,t)})$, $\Aa/U$ has a good minimal model, hence any minimal model of $\Aa/U$ is good. Thus (2) follows from Theorem \ref{thm: finiteness of log minimal model part ii}.
\end{proof}

\section{MMP with scaling}\label{sec: mmp with scaling}

In this section we study the MMP with scaling for algebraically integrable adjoint foliated structures. We first show that we can run an MMP with scaling (Lemma \ref{lem: ai afs can run mmp with scaling}), and then show that many of them terminate (Theorem \ref{thm: nqc mmp with scaling posssible not ample}).  

\begin{lem}\label{lem: non-isomorphism lemma}
    Let $\Aa/U$ be an lc adjoint foliated structure with ambient variety $X$. Let $f_i: X_i\dashrightarrow X_{i+1}$ be a sequence of steps of a $K_{\Aa}$-MMP$/U$ with $X_1:=X$. Then the induced birational map $X_n\dashrightarrow X_m$ is not an isomorphism for any $m>n$.
    \end{lem}
    \begin{proof}
        Let $\Aa_i$ be the image of $\Aa$ on $X_i$ for each $i$. Since $f_i$ is $K_{\Aa_i}$-negative for any $i$, there exists a prime divisor $E$ over $X_n$ such that
        $$a(E,\Aa_n)<a(E,\Aa_{n+1})\leq a(E,\Aa_m)$$
        for any $m\geq n+1$, and the lemma follows.
    \end{proof}

\begin{lem}\label{lem: mmp preseves +a property}
    Let $\Aa/U:=(X,\Ff,B,\Mm,t)/U$ be an lc adjoint foliated structure. Let $A$ be an ample$/U$ $\Rr$-divisor and let $\phi: X\dashrightarrow X'$ be a sequence of steps of a $(K_{\Aa}+A)$-MMP$/U$. Then there exists an ample$/U$ $\Rr$-divisor $A'$ on $X'$ and a nef$/U$ $\Rr$-divisor $\Nn$ on $X'$ satisfying the following. 
    \begin{enumerate}
        \item $\Aa':=(X',\Ff',B',\Nn,t)$ is lc, where $\Ff':=\phi_*\Ff$ and $B':=\phi_*B$.
        \item $\Nn_{X'}+A'\sim_{\mathbb R,U}\Mm_{X'}+\phi_*A$.
        \item $\Nn-\Mm$ is nef$/U$.
        \item If $\Ff$ is algebraically integrable and $\Aa$ is $\mathbb Q$-factorial qdlt (\cite[Definition 3.1]{CHLMSSX24}), then $\Aa'$ is $\Qq$-factorial qdlt.
    \end{enumerate}
\end{lem}
\begin{proof}
    We may assume that $\phi$ is a single step of a MMP$/U$, and we have the following diagram$/U$
    \begin{center}$\xymatrix{
   X\ar@{->}[rd]^{g}\ar@{-->}[rr]^{\phi} & & X'\ar@{->}[dl]^{h}\\
    & T &
}$
\end{center}
such that either $g$ is a divisorial contraction and $h$ is a small modification, or $g$ is the flipping contraction and $h$ is the flipped contraction. In either case, there exists an ample$/T$ divisor $H$ on $X$ such that $K_{\Aa}+A+H\sim_{\mathbb R,T}0$. Let $H':=\phi_*H$, then $-H'$ is ample$/T$. Since $A$ is ample$/U$, there exists an ample$/U$ $\Rr$-divisor $C$ on $T$ such that $A-g^*C$ is ample$/U$. 

Let $0<\epsilon\ll 1$ be a real number. Then $A':=h^*C-\epsilon H'$ and $L:=A-g^*C+\epsilon H$ are ample$/U$, and $\phi$ is a step of a $(K_{\Aa}+A+\epsilon H)$-MMP$/T$. 

Let $\Nn:=\Mm+\overline{L}$ and let $\widehat{\Aa}:=(X,\Ff,B,\Nn,t)$. Since
$$K_{\Aa}+A+\epsilon H\sim_{\mathbb R,T}K_{\widehat{\Aa}},$$
we have that $\phi$ is a step of a $K_{\widehat{\Aa}}$-MMP$/T$, hence a step of a $K_{\widehat{\Aa}}$-MMP$/U$. 

Let $L':=\phi_*L$, then
$$\phi_*L+A'=\phi_*A.$$ 
By our construction, $\Nn$ and $\tilde A'$ satisfy (2) and (3). 

Since $\Nn-\Mm$ descends to $X$ and $\Aa$ is lc,  $\widehat{\Aa}$ is lc. Since $\phi$ is a step of a $K_{\widehat{\Aa}}$-MMP$/U$ and $\phi_*\widehat{\Aa}=\Aa'$, $\Aa'$ is lc, which implies (1). 

Finally, if $\Aa$ is $\mathbb Q$-factorial qdlt, then $\widehat{\Aa}$ is $\mathbb Q$-factorial qdlt. By \cite[Proposition 3.5]{CHLMSSX24}, $\Aa'$ is $\mathbb Q$-factorial qdlt, which implies (4).
\end{proof}

\begin{lem}\label{lem: ai afs scaling number obtained}
    Let $\Aa/U$ be an lc algebraically integrable adjoint foliated structure, $A$ an ample$/U$ $\Rr$-divisor, and let $D$ be a nef$/U$ $\Rr$-divisor on $X$. Let
    $$\lambda:=\sup\{t\geq 0\mid D+t(K_{\Aa}+A)\text{ is nef}/U\}.$$
    Then there exists a rational curve $C$ which spans a $(K_{\Aa}+A)$-negative extremal ray$/U$, such that $(D+t(K_{\Aa}+A))\cdot C=0$.
\end{lem}
\begin{proof}
    By the length of extremal rays \cite[Theorem 1.3(2)]{CHLMSSX24}, there are finitely many $(K_{\Aa}+A)$-negative extremal rays $R_i$, $1\leq i\leq k$, and each $R_i$ is spanned by a rational curve $C_i$. Thus
    $$\lambda=\min_{1\leq i\leq k}\left\{-\frac{D\cdot C_i}{(K_{\Aa}+A)\cdot C_i}\geq 0\right\}.$$
\end{proof}

\begin{lem}\label{lem: ai afs can run mmp with scaling}
    Let $\Aa/U$ be an lc algebraically integrable adjoint foliated structure with potentially klt ambient variety $X$. Let $A$ be an ample$/U$ $\Rr$-divisor, and let $C$ be an $\Rr$-divisor on $X$ such that $K_{\Aa}+A+C$ is nef$/U$. Then we may run a $(K_{\Aa}+A)$-MMP$/U$ with scaling of $C$.
\end{lem}
\begin{proof}
    The first step of such MMP exists by Lemma \ref{lem: ai afs scaling number obtained}, Theorem \ref{thm: contraciton theorem potentially klt}, and \cite[Theorem 1.5]{CHLMSSX24}. Suppose we have already constructed a sequence of $(K_{\Aa}+A)$-MMP$/U$ $\phi: X\dashrightarrow X'$. Let $\Aa':=\phi_*\Aa$, $A':=\phi_*A$, and $C':=\phi_*C$. By Lemma \ref{lem: mmp preseves +a property}, there exists an lc algebraically integrable adjoint foliated structure $\Aa'$ on $X'$ and an ample$/U$ $\Rr$-divisor $A'$, such that $$K_{\Aa'}+A'\sim_{\mathbb R,U}\phi_*(K_{\Aa}+A).$$
    Thus we may keep running this MMP.
\end{proof}

\begin{thm}\label{thm: nqc mmp with scaling posssible not ample}
  Let $\Aa/U:=(X,\Ff,B,\Mm,t)/U$ be a $\mathbb Q$-factorial NQC lc algebraically integrable adjoint foliated structure such that $X$ is klt, $A$ an ample$/U$ $\Rr$-divisor, and let $C\equiv_U D+\Nn_X$ be an $\Rr$-divisor on $X$ such that $K_{\Aa}+A+C$ is nef$/U$, $(X,\Ff,B+D,\Mm+\Nn,t)$ is lc, and $\Nn$ is NQC$/U$. Assume that $K_{\Aa}+A$ is pseudo-effective$/U$.
  
  Then we may run a $(K_{\Aa}+A)$-MMP$/U$ with scaling of $C$ and any such MMP terminates with a $\mathbb Q$-factorial good minimal model of $(X,\Ff,B,\Mm+\overline{A},t)/U$.
\end{thm}
\begin{proof}
Possibly replacing $\Aa$ with $(1-\epsilon)\Aa+\epsilon(X,\Ff,0,0,0)$
and $A$ with $A+\epsilon(K_{\Aa}-K_{X})$ for some $0<\epsilon\ll 1$ such that $t(1-\epsilon)\in\mathbb Q$, we may assume that $t\in\mathbb Q$ and $\Aa$ is klt. 

    By Lemma \ref{lem: ai afs can run mmp with scaling}, we can run the $(K_{\Aa}+A)$-MMP$/U$ with scaling of $C$. Let $f_i: X_i\dashrightarrow X_{i+1}$ be this MMP with $X_1:=X$ and let $\Ff_i,B_i,D_i,C_i,\Aa_i$ be the images of $\Ff,B,D,C,\Aa$ on $X_i$ respectively. Let
    $$\lambda_i:=\sup\{s\geq 0\mid K_{\Aa_i}+A_i+sC_i\text{ is nef}/U\}$$
    be the scaling numbers. Let $\phi_i: X\dashrightarrow X_i$ be the induced birational maps. Then $$\Aa_{i}(\lambda_i)/U:=(X_i,\Ff_i,B_i+\lambda_iD_i,\Mm+\overline{A}+\lambda_i\Nn,t)/U$$
    is a $\mathbb Q$-factorial good minimal model$/U$ of $$\Aa(\lambda_i)/U:=(X,\Ff,B+\lambda_i D,\Mm+\overline{A}+\lambda_i\Nn,t)/U.$$
    Let $V_0\subset\Weil_{\mathbb R}(X)$ be the finite dimensional vector subspace spanned by all the components of $B,D$. Since $\Mm,\Nn,\overline{A}$ are NQC$/U$, there exists a tuple of nef$/U$ $\bb$-Cartier $\bb$-divisors $\MM$ such that $\Mm+\overline{A}+\Nn$ and $\Mm+\overline{A}$ are contained in $\Span_{\mathbb R_{\geq 0}}(\MM)$. Let $V:=V_0\times\Span_{\mathbb R}(\MM)$.

    Assume that $\lambda_n<1$ for some $n$. Since $\Aa$ is klt, then $\Aa(\lambda_i)$ is klt for any $i\geq n$. By our construction and Theorem \ref{thm: finiteness of log minimal model part ii}, there are finitely many birational maps$/U$ $\psi_j: X\dashrightarrow Y_j$ such that for any $i\geq n$, $\psi_j\circ\phi_i^{-1}$ is the identity morphism. Thus either this MMP terminates, or there exist $i_1>i_2\geq n$ such that the induced birational map $X_{i_1}\dashrightarrow X_{i_2}$ is an isomorphism. The latter is not possible as it contradicts Lemma \ref{lem: non-isomorphism lemma}.

    Thus, $\lambda_n=1$ for any $n$, hence $\Aa_i(1)/U$ is a $\mathbb Q$-factorial good minimal model of $\Aa(1)/U:=(X,\Ff,B+D,\Mm+\overline{A}+\Nn,t)/U$. Let $0<\tau\ll 1$ be a rational number, then
    $$\Aa':=(X,\Ff,(1-\tau)(B+D),(1-\tau)\Mm+(1-\tau)\Nn+\overline{H},t(1-\tau))$$
    is klt, where $H:=A+\tau(K_{\Aa(1)}-K_X)$ is ample$/U$. Now $\Aa_i'/U$ is a $\mathbb Q$-factorial good minimal model of $\Aa'/U$ for any $i$, where $\Aa_i'$ is the image of $\Aa$ on $X_i$. This contradicts Lemma \ref{lem: non-isomorphism lemma} again.
 \end{proof}

\section{Proof of the main theorems}\label{sec: proof of the main theorems}

In this section we prove all the main theorems except Theorems \ref{thm: log foliated BAB} and \ref{thm: ft ai afs mds} (the more general cases of Theorems \ref{thm: BAB analogue} and \ref{thm: ft ai is ft}). Theorem \ref{thm: extract non-terminal place intro} was proven in Section \ref{sec: extract non-terminal place} while Theorem \ref{thm: finiteness of minimal models intro} was proven in Section \ref{sec: fom} and so we  are free to use them. For all the other main theorems, we shall prove them in order except Theorem \ref{thm: flop ai afs} as its proof needs Theorem \ref{thm: shokurov type polytope}.

\begin{proof}[Proof of Theorem \ref{thm: eogmm ai afs general case}]
By Theorem \ref{thm: klt afs implies potentially klt}, there exists a small $\mathbb Q$-factorialization $h: X'\rightarrow X$. Let $\Aa':=h^*\Aa$, then either $K_{\Aa'}$ is big$/U$ or the generalized boundary of $\Aa'$ is big$/U$. By Lemma \ref{lem: perturbation for q-factorial klt big boundary to +ample} and \cite[Lemma 6.6]{CHLMSSX24}, we may assume that $\Mm=\Nn+\overline{A}$ where $\Nn$ is NQC$/U$ and $A$ is an ample$/U$ $\Rr$-divisor on $X'$.  Since $\overline{H}$ is NQC$/U$ and 
$$K_{\Aa'}+lh^*H=h^*(K_{\Aa}+lH)$$ 
is nef$/U$ for $l\gg 0$, by Theorem \ref{thm: nqc mmp with scaling posssible not ample}, any $K_{\Aa'}$-MMP$/U$ with scaling of $h^*H$ terminates.

By Lemma \ref{lem: ai afs can run mmp with scaling}, we may run a $K_{\Aa}$-MMP$/U$ with scaling of $H$. Let $f_i: X_i\dashrightarrow X_{i+1}$ be this MMP with $X_1:=X$ and $\Aa_i$ and $H_i$ the images of $\Aa$ and $H$ on $X_i$ respectively.

Fix an integer $n\geq 0$. Assume that there exists a small  $\mathbb Q$-factorialization $h_i: X_i'\rightarrow X_i$ for any $1\leq i\leq n+1$, and there exist birational maps $f_i': X_i'\dashrightarrow X_{i+1}'$  for any $1\leq i\leq n$, such that
\begin{itemize}
    \item $h_1=h$,
    \item $f_i\circ h_{i-1}=h_i\circ f'_{i-1}$, and
    \item $f'_i$ is a non-trivial sequence of steps of a $K_{\Aa_i'}$-MMP$/U$ with scaling of $H_i'$, where $H_i'$ and $\Aa_i'$ are the images of $h^*H$ and $\Aa_i$ on $X_i'$ for each $i$.
\end{itemize}
We construct $f'_{n+1}$ and $h_{n+1}$ in the following way. Let $X_n\rightarrow T_n\leftarrow X_{n+1}$ be the $n$-th step of this MMP, where either $X_n\rightarrow T_n$ is a flipping contraction and $X_{n+1}\rightarrow T_n$ is a flipped contraction, or $X_n\rightarrow T_n$ is a divisorial contraction and  $X_{n+1}\rightarrow T_n$ is a small modifications. Let 
$$\lambda_i:=\inf\{s\geq 0\mid K_{\Aa_i}+sH_i\text{ is nef}/U\}$$
be the scaling numbers of this MMP, and $H_i$ the image of $H$ on $X_i$ for each $i$. Then $K_{\Aa_n}+\lambda_nH_n\sim_{\mathbb R,T_n}0$. Since $h_n$ is small, $K_{\Aa_n'}+\lambda_nH_n'\sim_{\mathbb R,T_n}0$. Now we run a $K_{\Aa_n'}$-MMP$/T_n$ with scaling of an ample divisor which terminates with a $\mathbb Q$-factorial good minimal model$/T_n$ by Theorem \ref{thm: nqc mmp with scaling posssible not ample}. This MMP is non-trivial because $K_{\Aa_n'}$ is not nef$/T_n$.

Let $f_n': X_n'\dashrightarrow X_{n+1}'$ be this MMP. This MMP is also a $K_{\Aa_n'}$-MMP$/T_n$ with scaling of $H_n'$, hence a $K_{\Aa_n'}$-MMP$/U$ with scaling of $H_n'$. Since $X_{n+1}$ is the ample model$/T_n$ of $K_{\Aa_n}$,  $X_{n+1}$ is the ample model$/T_n$ of $K_{\Aa_n'}$, hence the induced birational map $h_{n+1}: X_{n+1}'\dashrightarrow X_{n+1}$ is a morphism and $K_{\Aa_{n+1}'}=h_{n+1}^*K_{\Aa_{n+1}}$, where $\Aa_{n+1}'$ is the image of $\Aa_{n}'$ on $X_{n+1}'$. Moreover, if $f_n$ contracts a divisor $E$, then
$$a(E,\Aa_n')=a(E,\Aa_n)<a(E,\Aa_{n+1})=a(E,\Aa'_{n+1}),$$
hence $E$ is also contracted by $f_n'$. Therefore, $h_{n+1}$ is small.

We may continue this process, which shows that any infinite sequence of $K_{\Aa}$-MMP$/U$ with scaling of $H$ induces an infinite sequence of $K_{\Aa'}$-MMP$/U$ with scaling of $H'$. Since any $K_{\Aa'}$-MMP$/U$ with scaling of $h^*H$ terminates, any  sequence of $K_{\Aa}$-MMP$/U$ with scaling of $H$ terminates.

The $K_{\Aa}$-MMP$/U$ with scaling of $H$ must terminates with a minimal model of $\Aa/U$, which is a good minimal model of $\Aa/U$ by Lemma \ref{lem: g-pair version bir12 2.7} and Theorem \ref{thm: eogmm boundary big case}. The ``in particular part" also follows from Theorem \ref{thm: eogmm boundary big case}.
\end{proof}

The following theorem is an important variation of Theorem \ref{thm: eogmm ai afs general case}.

\begin{thm}\label{thm: eogmm ai afs +a case}
    Let $\Aa/U$ be an lc algebraically integrable adjoint foliated structure with potentially klt ambient variety $X$. Let $A$ be an ample$/U$ $\Rr$-divisor on $X$ such that $K_{\Aa}+A$ is pseudo-effective$/U$.

    Then we may run a $(K_{\Aa}+A)$-MMP$/U$ with scaling of an ample$/U$ $\Rr$-divisor and any such MMP terminates with a good minimal model of $(\Aa,\overline{A})/U$.
\end{thm}
\begin{proof}
    By Lemma \ref{lem: perturbation to q coefficient no m} we may assume that $\Aa$ is klt and $B+\Mm_X$ is big$/U$. Theorem \ref{thm: eogmm ai afs +a case} follows from Theorems \ref{thm: eogmm ai afs general case}.
\end{proof}

\begin{proof}[Proof of Theorem \ref{thm: bchm analogue}]
 It is a special case of Theorem \ref{thm: eogmm ai afs general case}.
\end{proof}

\begin{proof}[Proof of Theorem \ref{thm: gmm klt algint foliation gt on cy variety}]
    Let $\Ff$ be an algebraically integrable foliation
    of general type on a normal projective variety $X$ such that $(X,\Delta)$ is klt and $K_X+\Delta\sim_{\mathbb Q}0$ for some $\mathbb Q$-divisor $\Delta$. Then $(X,\Ff,(1-t)\Delta,t)$ is klt and of general type for some $0<t\ll 1$. By Theorem \ref{thm: eogmm ai afs general case}, $(X,\Ff,(1-t)\Delta,t)$ has a good minimal model, hence $(X,\Ff,(1-t)\Delta,t)$ has a canonical model. Since
    $$tK_{\Ff}+(1-t)K_X+(1-t)\Delta\sim_{\mathbb R} tK_{\Ff},$$
    the canonical model of $(X,\Ff,(1-t)\Delta,t)$ is also the canonical model of $\Ff$.
\end{proof}

\begin{proof}[Proof of Theorem \ref{thm: bpf ai afs}]
(1) It follows from Theorem \ref{thm: semi-ampleness theorem}. 

(2) By Theorem \ref{thm: klt afs implies potentially klt}, possibly replacing $X$ with a small $\mathbb Q$-factorialization, we may assume that $X$ is $\mathbb Q$-factorial. We write $\Aa:=(X,\Ff,B,\Mm,t)$ and let $A:=aH-K_{\Aa}$. Then $A$ is nef$/U$ and big$/U$, so there exist ample$/U$ $\Rr$-divisors and an $\Rr$-divisor $E\geq 0$ on $X$ such that $A=A_k+\frac{1}{k}E$ for any positive integer $k$. By Lemma \ref{lem: add divisor not containing lc center}, possibly replacing $A$ with $A_k$ and $B$ with $B+\frac{1}{k}E$ for some $k\gg 0$, we may assume that $A$ is ample$/U$.

Let $l:=2\dim X+1$, 
$$\Aa_s:=\left(X,\Ff,B^{\ninv}+\frac{1-s}{1-t}B^{\inv},\Mm,s\right),$$
and
$$\Bb_s:=\left(X,\Ff,B^{\ninv}+\frac{1-s}{1-t}B^{\inv},\Mm+\overline{A}+l\overline{H},s\right)$$ 
for any $s\in [0,1]$. By Theorem \ref{thm: non-pseudo-effective for afs}, $K_{\Aa_s}+A$ is pseudo-effective$/U$ for any $s\in [t,1]$, hence $K_{\Bb_s}=K_{\Aa_s}+A+lH$
is pseudo-effective$/U$ for any $s\in [t,1]$. Since $\Aa_t$ is klt, $\Bb_t$ is klt. By \cite[Lemma 3.2]{CHLMSSX24}, there exists a real number $0<\tau\ll 1$ such that $\Bb_{t+\tau}$ is klt. By Theorem \ref{thm: eogmm ai afs general case}, we may run a
$$K_{\Bb_{t+\tau}}\text{-MMP}/U$$
with scaling of an ample divisor which terminates with a good minimal model. This MMP is also a $(K_{\Aa_{t+\tau}}+A+lH)$-MMP$/U$. By the length of extremal rays and the contraction theorem on $\mathbb Q$-factorial klt varieties \cite[Theorem 1.3(2), Theorem 1.4(2)]{CHLMSSX24}, this MMP is $H$-trivial. Let $\phi: X\dashrightarrow X'$ be this MMP, $\Aa_s'=\phi_*\Aa_s,\Bb_s:=\phi_*\Bb_s$ for any $s\in [0,1]$, $A':=\phi_*A$, and $H':=\phi_*H$. By Proposition \ref{prop: lc Aat implies lc Aa0}, $\Bb_0'/U$ is a klt generalized pair with big$/U$ generalized boundary. Since $(a+l)H'=K_{\Bb_t'}$, we have
$$\frac{(t+\tau)(a+l)}{\tau}H'=K_{\Bb_0'}+\frac{t}{\tau}K_{\Bb_{t+\tau}'}.$$
By the base-point-freeness theorem for klt generalized pairs -- 
cf. \cite[Lemma 3.4]{HL22} plus the base-point-freeness theorem for klt pairs or \cite[Theorem 2.2.6]{CHLX23} -- $\mathcal{O}_{X'}(mH')$ is globally generated$/U$ for any integer $m\gg 0$. Since $\phi$ is $H$-trivial, $\mathcal{O}_{X}(mH)$ is globally generated$/U$ for any integer $m\gg 0$ and we are done.
\end{proof}
The following theorem is an important variation of Theorem \ref{thm: bpf ai afs}.

\begin{thm}\label{thm: bpf ai afs lc+a}
   Let $\Aa/U=(X,\Ff,B,\Mm,t)/U$ be an lc algebraically integrable adjoint foliated structure such that $X$ is potentially klt and $H$ a nef$/U$ Cartier divisor on $X$ such that $aH-K_{\Aa}$ is ample$/U$ for some positive real number $a$. Then $\mathcal{O}_X(mH)$ is globally generated$/U$ for any integer $m\gg 0$.
\end{thm}
\begin{proof}
    Let $(X,\Delta)$ be a klt pair. Possibly replacing $\Aa$ with $(1-s)\Aa+s(X,\Ff,\Delta,\bm{0},0)$ for some $0<s\ll 1$, we may assume that $\Aa$ is klt. Theorem \ref{thm: bpf ai afs lc+a} follows from Theorem \ref{thm: bpf ai afs}.
\end{proof}

\begin{proof}[Proof of Theorem \ref{thm: cont intro}]
(1) It follows from Theorem \ref{thm: contraciton theorem potentially klt}.

(2) Since $L-K_{\Aa}$ is ample$/T$, by Theorem \ref{thm: bpf ai afs lc+a}, $\mathcal{O}_X(mL)$ is globally generated$/T$ for any integer $m\gg 0$. Since $L\cdot C=0$ for any curve C contracted by $\cont_R$, $\cont_R$ is defined
by $|mL/U|$ for any integer $m\gg 0$. Thus, $mL\cong f^\ast L_{m,T}$ for some line bundle $L_{m,T}$ on $T$ and any integer $m\gg 0$. We may let $L_T:=L_{m+1,T}-L_{m,T}$ for some integer $m\gg 0$.
\end{proof}

\begin{proof}[Proof of Theorem \ref{thm: mmp potentially klt ambient algint afs}]
    It follows from \cite[Theorems 1.3,1.5]{CHLMSSX24} and Theorem \ref{thm: cont intro}.
\end{proof}

\begin{proof}[Proof of Theorem \ref{thm: eomfs ai afs}]
    Let $H$ be an ample$/U$ $\Rr$-divisor on $X$ and let 
    $$\mu:=\sup\{s\geq 0\mid K_{\Aa}+sH\text{ is not pseudo-effective}/U\}.$$
    Then $\mu>0$ and $K_{\Aa}+\mu H$ is pseudo-effective$/U$. 
    
    Let $f_i: X_i\dashrightarrow X_{i+1}$ with $X_1:=X$ be a sequence of $K_{\Aa}$-MMP$/U$ with scaling of $H$, whose existence is guaranteed by Theorem \ref{thm: eogmm ai afs +a case}, and let $\Aa_i$ and $H_i$ be the images of $\Aa$ and $H$ on $X_i$ respectively for each $i$. Let
    $$\lambda_i:=\inf\{s\geq 0\mid K_{\Aa_i}+sH_i\text{ is nef}/U\}$$
    be the scaling numbers. Then $\lambda_i\geq\mu$ for each $i$ so $f_i\colon X_i\dashrightarrow X_{i+1}$ is a sequence of a$(K_{\Aa}+\mu H)$-MMP$/U$ with scaling of $H$. By Theorem \ref{thm: eogmm ai afs +a case}, this MMP terminates.
\end{proof}

\begin{proof}[Proof of Theorem \ref{thm: abundance ai afs}]
We may assume that $K_{\Aa}$ is pseudo-effective$/U$, otherwise the theorem is trivial. Since the MMP preserves the invariant Iitaka dimension, the numerical Iitaka dimension, effectivity, and finite generation, Theorem \ref{thm: abundance ai afs} is a special case of Theorem \ref{thm: eogmm ai afs general case}.
\end{proof}

\begin{proof}[Proof of Theorem \ref{thm: mmp ai potentially klt ambient variety}]
    Theorem \ref{thm: mmp ai potentially klt ambient variety}(1) is a special case of Theorem \ref{thm: mmp potentially klt ambient algint afs}.  Theorem \ref{thm: mmp ai potentially klt ambient variety}(2) is a special case of Theorem \ref{thm: eogmm ai afs +a case}. Theorem \ref{thm: mmp ai potentially klt ambient variety}(3) is a special case of Theorem \ref{thm: eomfs ai afs}.  Theorem \ref{thm: mmp ai potentially klt ambient variety}(4) follows from  Lemma \ref{lem: ai afs can run mmp with scaling} and \cite[Theorem 1.11(1)]{LMX24}.
\end{proof}

\begin{proof}[Proof of Theorem \ref{thm: shokurov type polytope}]
Let $h: W\rightarrow X$ be a foliated log resolution of $\Aa$. Let $\mathcal{S}$ be the set of $h$-exceptional prime divisors union all components of $B$. Then for any prime divisor $D\in\mathcal{S}$,
$$(\bm{v},t)\rightarrow a(D,\Aa(\bm{v},t))\text{ and }(\bm{v},t)\rightarrow  a(D,\Aa(\bm{v},t))+t$$
are both $\mathbb Q$-affine functions $V\rightarrow\mathbb R$. Therefore, if $a(D,\Aa)=1$ (resp. $a(D,\Aa)+t_0=1$), then $a(D,\Aa(\bm{v},t))=1$ (resp. $a(D,\Aa(\bm{v},t))+t=1$) for any $(\bm{v},t)\in\mathbb R^{m+n+1}$.

Since $\mathcal{S}$ is a finite set, there exists an open neighborhood $U_1$ of $\bm{v}_0$, such that for any $D\in\mathcal{S}$, if $a(D,\Aa)<1$ (resp. $a(D,\Aa)+t<1$), then $a(D,\Aa(\bm{v},t))<1$ (resp. $a(D,\Aa(\bm{v},t))+t<1$) for any $(\bm{v},t)\in U_1$, and $t\leq 1$ for any $(\bm{v},t)\in U_1$. Therefore, for any $(\bm{v},t)\in U_1$, $\Aa(\bm{v},t)$ is lc.

Let $c:=\dim U_1$. Pick $\bm{v}_1,\dots,\bm{v}_{c+1}\in U_1\cap\mathbb Q^{m+n+1}$ such that $\bm{v}_0$ is contained in the convex hull $U_2$ spanned by $\bm{v}_1,\dots,\bm{v}_{c+1}$.  Then there exist positive real numbers $a_1,\dots,a_{c+1}$ such that $\sum_{i=1}^{c+1}a_i\bm{v}_i=\bm{v}_0$ and $\sum_{i=1}^{c+1}a_i=1$. We let $I$ be a positive integer such that $IK_{\Aa(\bm{v}_i)}$ is Cartier for each $i$. Let $d:=\dim X$ and $a_0:=\min_{1\leq i\leq c+1}\{a_i\}$. 

Consider the set
    $$\Ii:=\left\{\sum a_i\gamma_i\middle|\gamma_i\in [-2dI,+\infty)\cap\mathbb Z\right\}\cap (0,+\infty).$$
    We have $\gamma_0:=\inf\{\gamma\in\Ii\}>0$. We let $U$ be the interior of the set
    $$\left\{\frac{1}{2d+\gamma_0}(2d\bm{v}_0+\gamma_0\bm{v})\Bigg| \bm{v}\in U_2\right\}.$$

    We show that $U$ satisfies our requirement. By our construction, $\Aa(\bm{v})$ is lc for any $\bm{v}\in U$ so we only need to show that $K_{\Aa(\bm{v})}$ is nef$/Z$ for any $\bm{v}\in U$. We let $R$ be an extremal ray in $\overline{NE}(X/U)$. There are three cases.

    \medskip

    \noindent\textbf{Case 1}. $K_{\Aa}\cdot R=0$. In this case, $K_{\Aa}\cdot R=0$ for any $\bm{v}\in U_1$, so $K_{\Aa}\cdot R=0$ for any $\bm{v}\in U$.

\medskip

    \noindent\textbf{Case 2}. $K_{\Aa(\bm{v}_i)}\cdot R\geq 0$ for any $i$. In this case, $K_{\Aa(\bm{v})}\cdot R\geq 0$ for any $\bm{v}\in U_1$, so $K_{\Aa(\bm{v})}\cdot R\geq 0$ for any $\bm{v}\in U$.

    \medskip

     \noindent\textbf{Case 3}. $K_{\Aa}\cdot R>0$ and $K_{\Aa(\bm{v}_j)}\cdot R<0$ for some $j$. In this case, by the cone theorem \cite[Theorem 1.3]{CHLMSSX24}, $R$ is spanned by a curve $C$ such that $K_{\Aa(\bm{v}_i)}\cdot C\geq -2d$ for any $i$. Thus $$IK_{\Aa(\bm{v}_i)}\cdot C\in [-2dI,+\infty)\cap\mathbb Z,$$
    so
    $$IK_{\Aa(\bm{v}_0)}\cdot C\in\Ii_0.$$
    Then for any $\bm{v}\in U$, there exists $\bm{v}'\in U_2$ such that $(2d+\gamma_0)\bm{v}=2d\bm{v}_0+\gamma_0\bm{v}'$. We have
    \begin{align*}
        IK_{\Aa(\bm{v})}\cdot C&=\frac{\gamma_0}{2d+\gamma_0}IK_{\Aa(\bm{v}')}\cdot C+\frac{2d}{2d+\gamma_0}IK_{\Aa(\bm{v}_0)}\cdot C\\
        &\geq \frac{\gamma_0}{2d+\gamma_0}\cdot (-2d)+\frac{\gamma_0}{2d+\gamma_0}\cdot\gamma_0=0,
    \end{align*}
    so $IK_{\Aa(\bm{v})}\cdot R\geq 0$. The theorem follows.
\end{proof}

\begin{proof}[Proof of Theorem \ref{thm: nqc of ka}]
    It immediately follows from Theorem \ref{thm: shokurov type polytope}.
\end{proof}

\begin{proof}[Proof of Theorem \ref{thm: flop ai afs}]
The idea of the proof is similar to the ones in \cite{Kaw08,CLW24}. Let $X,X_1,X_2$ be the ambient varieties of $\Aa,\Aa_1,\Aa_2$ respectively. Let $\phi_i: X\dashrightarrow X_i$ be the induced birational maps. By Lemma \ref{lem: nz for lc divisor}, the divisors contracted by $\phi_i$ are $\Supp N_{\sigma}(X/U,K_{\Aa})$. Thus $\phi$ is an isomorphism in codimension $1$. Let $L_2\geq 0$ be an ample$/U$ divisor on $X$ and let $L_1:=\phi^{-1}_*L_2$. Then $L_1\geq 0$ is big$/U$. 

If $\Aa$ is klt then also $\Aa_1,\Aa_2$ are klt. By Lemma \ref{lem: add divisor not containing lc center}, there exists $\epsilon_0>0$ such that $(\Aa_1,\epsilon L_1)$ and $(\Aa_2,\epsilon L_2)$ are klt for any $\epsilon\in (0,\epsilon_0]$.

 Since $\phi$ is an isomorphism in codimension $1$, $\phi$ is then unique $\mathbb Q$-factorial good minimal model of $(\Aa_1,\epsilon L_1)$ for any $\epsilon>0$. We run a $(K_{\Aa_1}+\epsilon L_1)$-MMP$/U$ with scaling of an ample$/U$ divisor. Since $L_1$ is big$/U$, by Theorem \ref{thm: eogmm ai afs general case}, this MMP terminates with the good minimal model $(\Aa_2,\epsilon L_2)/U$ of $(\Aa_1,\epsilon L_1)/U$.

By Theorem \ref{thm: shokurov type polytope}, $K_{\Aa_1}$ is NQC$/U$. Since
$$\frac{\epsilon_0}{\epsilon}(K_{\Aa_1}+\epsilon L_1)=K_{\Aa_1}+\epsilon_0L_1+\frac{\epsilon_0-\epsilon}{\epsilon}K_{\Aa_1},$$
by the length of extremal rays and the contraction theorem \cite[Theorem 1.3(2), Theorem 1.4(2)]{CHLMSSX24}, for any $0<\epsilon\ll\epsilon_0$, any sequence of $(K_{\Aa_1}+\epsilon L_1)$-MMP$/U$ is $K_{\Aa_1}$-trivial, and hence the induced birational map $\phi: X_1\dashrightarrow X_2$ is a sequence of $K_{\Aa_1}$-flops$/U$.
\end{proof}

\section{Boundedness of Fano adjoint foliated structures}\label{sec: bdd fano}

In this section, we prove Theorems \ref{thm: log foliated BAB} and \ref{thm: ft ai afs mds}, which imply Theorems \ref{thm: BAB analogue} and \ref{thm: ft ai is ft}. We begin with the following proposition as a stronger version of Proposition \ref{prop: lc Aat implies lc Aa0}.

\begin{prop}\label{prop: algint afs elc implies x elc}
    Let $\epsilon$ be a non-negative real number, $t\in [0,1)$ a real number, and $\Aa:=(X,\Ff,B^{\ninv}+(1-t)B^{\inv},\Mm,t)$ a $\mathbb Q$-factorial $\epsilon$-lc (resp. $\epsilon$-klt) algebraically integrable adjoint foliated structure. Then $(X,B,\Mm)$ is $\epsilon$-lc  (resp. $\epsilon$-klt).
\end{prop}
\begin{proof}
    We may assume $\epsilon>0$, otherwise the proposition follows from Proposition \ref{prop: lc Aat implies lc Aa0}. We may also assume that $t>0$, otherwise the proposition is trivial.

    We only need to show that for any prime divisor $E$ over $X$, $a(E,X,B,\Mm)\geq -1+\epsilon$ (resp. $a(E,X,B,\Mm)>-1+\epsilon$). In the following, we fix such a divisor $E$.

 Let $h: W\rightarrow X$ be a foliated log resolution of $\Aa$ such that $E$ is a divisor on $W$. Let $\Ff_W:=h^{-1}\Ff$ and $B_W:=h^{-1}_*B+\Exc(h)$. Then there exists an $\Ff$-invariant divisor $G_W\geq 0$ on $W$ and a contraction $f: W\rightarrow Z$ such that $(W,\Ff_W,B_W^{\ninv},\Mm;G_W)/Z$ is $\mathbb Q$-factorial ACSS and $\Supp G_W\supset \Supp B_W^{\inv}$.

    We write $K_{\Ff_W}+B_W^{\ninv}+\Mm_W\sim_{\mathbb R,X}F_1-F_2$ where $F_1,F_2\ge 0$ are exceptional$/X$ and $F_1\wedge F_2=0$. By \cite[Theorem 9.4.1]{CHLX23} we may run a $(K_{\Ff_W}+B_W^{\ninv}+\Mm_W)$-MMP$/X$ which contracts $F_1$ after finitely many steps $\phi\colon W\dashrightarrow Y$. By \cite[Lemma 9.1.4]{CHLX23}, $\phi$ is also a sequence of steps of a $(K_{W}+B_W^{\ninv}+G_W+\Mm_W)$-MMP$/X$. 
    
    Let $g\colon Y\rightarrow X$ be the induced birational morphism, $\Ff_Y:=\phi_*\Ff_W$, $B_Y:=\phi_*B_W$, and $G_Y:=\phi_*G_W$. Let $\Aa_W:=(W,\Ff_W,B_W,\Mm,t)$ and let $\Aa_Y:=\phi_*\Aa_W$.

    There are three possibilities:

    \medskip

\noindent\textbf{Case 1.} $E$ is not exceptional$/X$. In this case,
\begin{align*}
  & -(t\epsilon_{\Ff}(E)+(1-t))+\epsilon\leq\text{(resp. }<\text{) } a(E,\Aa)=-\mult_E(B^{\ninv}+(1-t)B^{\inv})\\
   =&-(t\epsilon_{\Ff}(E)+(1-t))\mult_EB=(t\epsilon_{\Ff}(E)+(1-t))a(E,X,B,\Mm).
\end{align*}
Hence $$a(E,X,B,\Mm)\geq\text{(resp. }>\text{) }  -1+\frac{\epsilon}{t\epsilon_{\Ff}(E)+(1-t)}\geq -1+\epsilon$$
and we are done. 

\medskip

\noindent\textbf{Case 2.} $E$ is exceptional$/X$ but not exceptional$/Y$. In this case, since
$$K_{\Ff_Y}+B_Y^{\ninv}+\Mm_Y\sim_{\mathbb R,X}-\phi_*F_2\leq 0$$
and $\mult_EB_Y^{\ninv}=\epsilon_{\Ff}(E)$, we have $$a(E,X,\Ff,B^{\ninv},\Mm)\leq -\epsilon_{\Ff}(E).$$
Thus,
\begin{align*}
  & -(t\epsilon_{\Ff}(E)+(1-t))+\epsilon\leq a(E,\Aa)=ta(E,X,\Ff,B^{\ninv},\Mm)+(1-t)a(E,X,B,\Mm)\\
   \leq &-t\epsilon_{\Ff}(E)+(1-t)a(E,X,B,\Mm).
\end{align*}
Therefore, 
$$a(E,X,B,\Mm)\geq -1+\frac{\epsilon}{1-t}>-1+\epsilon$$
and we are done.
\medskip

\noindent\textbf{Case 3.} $E$ is exceptional$/Y$. In this case, by \cite[Proposition 7.3.6, Lemma 9.1.4]{CHLX23},  $\phi$ is a sequence of steps of a $(K_{\Ff_W}+B_W^{\ninv}+\Mm_W)$-MMP$/Z$ as well as a sequence of steps of $(K_W+B_W^{\ninv}+G_W+\Mm_W)$-MMP$/Z$, and 
$$K_{\Ff_W}+B_W^{\ninv}+\Mm_W\sim_{Z}K_W+B_W^{\ninv}+G_W+\Mm_W.$$
Therefore, by the negativity lemma, we have
\begin{align*}
    &a(E,W,B_W^{\ninv}+G_W,\Mm)-a(E,Y,B_Y^{\ninv}+G_Y,\Mm)\\
    =&a(E,W,\Ff_W,B_W^{\ninv},\Mm)-a(E,Y,\Ff_Y,B_Y^{\ninv},\Mm).
\end{align*}
Since $E$ is a component of $\Exc(h)$, we have $\mult_E(B_W^{\ninv}+G_W)=1$ and $\mult_EB_W^{\ninv}=\epsilon_{\Ff}(E)$, hence
$$a(E,Y,\Ff_Y,B_Y^{\ninv},\Mm)=a(E,Y,B_Y^{\ninv}+G_Y,\Mm)+(1-\epsilon_{\Ff}(E)).$$
Let
$$K_Y+B_Y'+\Mm_Y:=g^*(K_X+B+\Mm_X).$$
By Proposition \ref{prop: lc Aat implies lc Aa0}, $(X,B,\Mm)$ is lc, so 
$$B_Y^{\ninv}+G_Y=\phi_*(B_W^{\ninv}+G_W)\geq\phi_*B_W=\phi_*(h^{-1}_*B+\Exc(h))=g^{-1}_*B+\Exc(g)\geq B_Y'.$$
Therefore,
$$a(E,Y,B_Y^{\ninv}+G_Y,\Mm)\leq a(E,Y,B_Y',\Mm).$$
By our construction of $h$, we have
$$h^*(K_{\Ff}+B^{\ninv}+\Mm_X)\geq K_{\Ff_Y}+B_Y^{\ninv}+\Mm_X,$$
hence
$$a(E,X,\Ff,B^{\ninv},\Mm)\leq a(E,Y,\Ff_Y,B_Y^{\ninv},\Mm).$$
By linearity of discrepancies and combining the inequalities above, we have
\begin{align*}
-(t\epsilon_{\Ff}(E)+(1-t))+\epsilon&\leq \text{(resp. }<\text{) } a(E,\Aa)\\
&=ta(E,X,\Ff,B^{\ninv},\Mm)+(1-t)a(E,X,B,\Mm)\\
&\leq ta(E,Y,\Ff_Y,B_Y^{\ninv},\Mm)+(1-t)a(E,X,B,\Mm)\\
&=ta(E,Y,B_Y^{\ninv}+G_Y,\Mm)+t(1-\epsilon_{\Ff}(E))+(1-t)a(E,X,B,\Mm)\\
&\leq ta(E,Y,B_Y',\Mm)+t(1-\epsilon_{\Ff}(E))+(1-t)a(E,X,B,\Mm)\\
&= ta(E,X,B,\Mm)+t(1-\epsilon_{\Ff}(E))+(1-t)a(E,X,B,\Mm)\\
&= a(E,X,B,\Mm)+t(1-\epsilon_{\Ff}(E)).
\end{align*}
Therefore, $a(E,X,B,\Mm)\geq \text{(resp. }>\text{) } -1+\epsilon$ and we are done.
\end{proof}

\begin{thm}\label{thm: precise fano type afs to x}
    Let $\epsilon\in (0,1)$ be a real number. Let $\Aa/U:=(X,\Ff,B,\Mm,t)/U$ be an algebraically integrable $\epsilon$-lc (resp. $\epsilon$-klt) adjoint foliated structure such that $B+\Mm_X$ is big$/U$ and $K_{\Aa}\sim_{\mathbb R,U}0$. Then there exists an $\epsilon$-lc (resp. $\epsilon$-klt) pair $(X,\Delta)$ such that $\Delta$ is big$/U$ and $K_X+\Delta\sim_{\mathbb R,U}0$. In particular, $X$ is of Fano type$/U$.
\end{thm}
\begin{proof}
    By our assumptions, $\Aa$ is klt. In particular, $t<1$. Possibly replacing $\Aa$ with a small $\mathbb Q$-factorialization, we may assume that $X$ is $\mathbb Q$-factorial. 

We first show that, for any $\Rr$-divisor $D$ on $X$, we may run a $D$-MMP$/U$ which terminates with either a good minimal model$/U$ or a Mori fiber space$/U$. We write $B+\Mm_X=A+E$ where $A$ is ample$/U$ and $E\geq 0$. By Lemma \ref{lem: add divisor not containing lc center}, we may pick $\tau>0$ such that $\widehat{\Aa}:=(X,\Ff,(1-\tau)B+\tau E,(1-\tau)\Mm)$ is klt. For any $\Rr$-divisor $D$ on $X$, let $r$ be a positive real number such that $\tau A+rD$ is ample$/U$. By Theorems \ref{thm: eogmm ai afs general case} and \ref{thm: eomfs ai afs}, we may run a $(K_{\Aa}+(\tau A+rD))$-MMP$/U$,  which is also a $D$-MMP$/U$, and we may choose such MMP so that it terminates with either a good minimal model$/U$ or a Mori fiber space$/U$.

We run a $(-K_X)$-MMP$/U$ $\phi: X\dashrightarrow X'$ which terminates with either a Mori fiber space$/U$ or a good minimal model$/U$. If the MMP terminates with a Mori fiber space$/U$ then we let $f: X'\rightarrow Z$ be the $(-K_{X'})$-Mori fiber space$/U$. If the MMP terminates with a good minimal model$/U$, then we let $f: X'\rightarrow Z$ be the ample model$/U$ of $-K_{X'}$. Let $\Aa':=\phi_*\Aa$ and $\Ff':=\phi_*\Ff$. Since $K_{\Aa}\sim_{\mathbb R,U}0$, $\Aa$ and $\Aa'$ are crepant and $K_{\Aa'}\sim_{\mathbb R,U}0$. In particular, $\Aa'$ is $\epsilon$-lc (resp. $\epsilon$-klt). By Proposition \ref{prop: algint afs elc implies x elc}, $X'$ is $\epsilon$-lc (resp. $\epsilon$-klt). 

Since $B+\Mm_X$ is big$/U$, $-tK_{\Ff}-(1-t)K_X$ is big$/U$, hence $-tK_{\Ff'}-(1-t)K_{X'}$ is big$/U$, and so $-tK_{\Ff'}-(1-t)K_{X'}$ is big$/Z$. By our construction, $K_{X'}$ is pseudo-effective$/Z$. By Theorem \ref{thm: non-pseudo-effective for afs}, $tK_{\Ff'}+(1-t)K_{X'}$ is pseudo-effective$/Z$. Therefore, $f$ is birational, so $-K_{X'}$ is big$/U$ and semi-ample$/U$. 

Since $X'$ is $\epsilon$-lc (resp. $\epsilon$-klt) and $\epsilon<1$, we may choose $\Delta'\in |-K_{X'}|_{\mathbb Q/U}$ such that $(X',\Delta')$ is $\epsilon$-lc (resp. $\epsilon$-klt) and $K_{X'}+\Delta'\sim_{\mathbb R,Z}0$. Let $p: W\rightarrow X$ and $q: W\rightarrow X'$ be a resolution of indeterminacy of $\phi$ and let
$$K_X+\Delta:=p_*q^*(K_{X'}+\Delta').$$
By the negativity lemma, $(X,\Delta)$ and $(X',\Delta')$ are crepant, so $(X,\Delta)$ is sub-$\epsilon$-lc (resp. sub-$\epsilon$-klt) and $K_X+\Delta\sim_{\mathbb R,Z}0$. Since $\phi$ is $(-K_X)$-negative, $\phi$ is also $\Delta$-negative, hence
$$p^*\Delta=q^*\Delta'+F$$
for some $F\geq 0$, and so
$$\Delta=p_*q^*\Delta'+p_*F\geq 0$$
is big$/U$. In particular, $(X,\Delta)$ is $\epsilon$-lc (resp. $\epsilon$-klt) and we are done.
\end{proof}

\begin{deflem}\label{deflem: fano type afs}
    Let $\pi: X\rightarrow U$ be a projective morphism between normal quasi-projective varieties, $\Ff$ a foliation on $X$, and $t\in [0,1]$, $\epsilon\in [0,1)$ real numbers. Then the following conditions are equivalent:
    \begin{enumerate}
        \item There exists an $\epsilon$-klt $\Aa_1:=(X,\Ff,B_1,\Mm_1,t)$ such that $-K_{\Aa_1}$ is ample$/U$. 
        \item There exists an $\epsilon$-klt $\Aa_2:=(X,\Ff,B_2,\Mm_2,t)$ such that $-K_{\Aa_2}$ is big$/U$ and nef$/U$.
        \item There exists an $\epsilon$-klt $\Aa_3:=(X,\Ff,B_3,\Mm_3,t)$ such that $K_{\Aa_3}\sim_{\mathbb R,U}0$ and $-tK_{\Ff}-(1-t)K_X$ is big$/U$.
    \end{enumerate}
    We say that $(X,\Ff,t)/U$ is of \emph{$\epsilon$-Fano type} if one of the conditions above (equivalently, all conditions above) holds. We say that $(X,\Ff,t)/U$ is of \emph{Fano type} if $(X,\Ff,t)/U$ is of $0$-Fano type. If $U=\{pt\}$ then we omit $U$.
\end{deflem}
\begin{proof}
    (1)$\Rightarrow$(2): We let $\Aa_2:=\Aa_1$.
    
    (2)$\Rightarrow$(3): Since $-K_{\Aa_2}$ is big$/U$,  $-tK_{\Ff}-(1-t)K_X$ is big$/U$. We let $\Aa_3:=(\Aa_2,\overline{-K_{\Aa_2}})$.

    (3)$\Rightarrow$(1): Let $h: (X',\Ff',B_3',\Mm_3,t)\rightarrow (X,\Ff,B_3,\Mm_3,t)$ be a small $\mathbb Q$-factorialization of $X$. Write $B_3+\Mm_{3,X}=A+E$ for some $E\geq 0$ and ample$/U$ $\Rr$-divisor $A$. Since $(X,\Ff,B_3,\Mm_3,t)$ is $\epsilon$-klt, $(X',\Ff',B_3',\Mm_3,t)$ is $\epsilon$-klt, hence $(X',\Ff',B_3',\Mm_3,t)$ is $((1+\alpha)\epsilon)$-lc for some $\alpha>0$. By Lemma \ref{lem: add divisor not containing lc center},
    $$\Aa':=(X',\Ff',(1-\tau)B_3'+\tau h^*E,(1-\tau)\Mm_3,t)$$
    is klt for some $0<\tau\ll 1$. Possibly replacing $\tau$ with $\frac{\alpha\tau}{1+\alpha}$, we may assume that $\Aa'$ is $\epsilon$-klt. Moreover, $-K_{\Aa'}\sim_{\mathbb R,U}\tau h^*A$, so $h_*K_{\Aa'}$ is anti-ample$/U$. Thus, (3) implies (1) by letting $\Aa_1:=h_*\Aa'$.
\end{proof}

\begin{proof}[Proof of Theorem \ref{thm: ft ai afs mds}] 
Let $(X,\Delta_0)$ be a klt pair. Possibly replacing $\Aa$ with $(1-s)\Aa+s(X,\Delta_0)$ for some $0<s\ll 1$, we may assume that $\Aa$ is klt. By Definition-Lemma \ref{deflem: fano type afs}, there exists a klt adjoint foliated structure $\Aa'/U$ with big$/U$ generalized boundary. The main part follows from Theorem \ref{thm: precise fano type afs to x} and the ``moreover'' part follows from \cite[Corollary 1.3.1]{BCHM10}.
\end{proof}

\begin{proof}[Proof of Theorem \ref{thm: ft ai is ft}]
  This is a special case of Theorem \ref{thm: ft ai afs mds}.
\end{proof}

\begin{proof}[Proof of Theorem \ref{thm: log foliated BAB}]
Let $(X,\mathcal F)\in \mathcal P$. 
By Definition-Lemma \ref{deflem: fano type afs} and Theorem \ref{thm: precise fano type afs to x} there exists an $\epsilon$-klt pair $(X,\Delta)$ such that $\Delta$ is big and $K_X+\Delta\sim_{\mathbb R}0$. By \cite[Corollary 1.4]{Bir21}, $X$ belongs to a bounded family. 

Therefore, there exists a positive integer $r$ which does not depend on $X$, and a very ample divisor $A$ on $X$, such that $A^d\leq r$ and $-K_X\cdot A^{d-1}\leq r$. Since $-tK_{\Ff}-(1-t)K_X$ is big,
$$0\leq (-tK_{\Ff}-(1-t)K_X)\cdot A^{d-1}\leq r(1-t)-t(K_{\Ff}\cdot A^{d-1}),$$
so
$$K_{\Ff}\cdot A^{d-1}\leq \frac{r(1-t)}{t}\leq\frac{r(1-\epsilon)}{\epsilon}.$$

Since $X$ is bounded, there exists a positive real number $s$ which does not depend on $X$, such that $K_X+sA$ is pseudo-effective. Let $h: X'\rightarrow X$ be a small $\mathbb Q$-factorialization of $X$. Then $(X',0,\overline{A})$ is klt and $K_{X'}+sh^*A$ is pseudo-effective. By Theorem \ref{thm: non-pseudo-effective for afs}, $K_{h^{-1}\Ff}+sh^*A$ is pseudo-effective, hence $K_{\Ff}+sA$ is pseudo-effective. Therefore,
$$(K_{\Ff}+sA)\cdot A^{d-1}\geq 0,$$
so
$$K_{\Ff}\cdot A^{d-1}\geq -sA^d=-sr.$$

In summary, there exists a positive real number $r_0$ such that
$$-r_0\leq K_{\Ff}\cdot A^{d-1}\leq r_0.$$

Thus, we may assume that there exist a flat morphism  
$f\colon \mathscr{X} \to T$ 
between normal varieties with normal fibers
and a $\mathbb Q$-Cartier divisor $K$ on $\mathscr{X}$ such that if $(X,\mathcal F)\in \mathcal P$ then 
 there exist a closed point $t\in T$ and an isomorphism $\phi\colon X\to X_t$, where $X_t$ is the fiber of $f$ over $t$, such that $K_{\mathcal F}\sim \phi^*(K|_{X_t})$. Therefore, 
Proposition \ref{prop_pfaff_to_fol} implies that $\mathcal P$ is bounded.
\end{proof}

\begin{proof}[Proof of Theorem \ref{thm: BAB analogue}]
 It is a special case of Theorem \ref{thm: log foliated BAB}.   
\end{proof}

\begin{rem}
In the statements of Theorem \ref{thm: BAB analogue} and Theorem \ref{thm: log foliated BAB}, we require that $t\geq\epsilon$. This is a necessary condition, as shown by considering $X=\mathbb P^2$ and $\Ff$ a foliation on $X$ such that $\mathcal{O}_X(K_{\Ff})=\mathcal{O}(n)$ with arbitrary large $n$, as it does not form a bounded family.

On the other hand, in the statements of Theorem \ref{thm: log foliated BAB}, we do not have any requirement on the upper bound of the parameter $t$ and we actually allow $t=1$. However, the only values of $t$ which usually make sense are those $t$ such that $t\leq 1-\epsilon$. Indeed, let $L$ be an $\Ff$-invariant divisor, then
$$-(1-t)+\epsilon\leq a(L,X,\Ff,B,\Mm,t)=-\mult_LB^{\inv}\leq 0,$$
which implies that $t\leq 1-\epsilon$. Since $\mathcal F$ is algebraically integrable, such $L$ always exists 
unless $\Ff=T_X$. Thus, we automatically have the condition $t\leq 1-\epsilon$ when we assume that $(X,\Ff,B,\Mm,t)$ is $\epsilon$-lc.
\end{rem}


\begin{thebibliography}{99}
\bibitem[Amb05]{Amb05} F. Ambro, \textit{The moduli b-divisor of an lc-trivial fibration}, Compos. Math. \textbf{141} (2005), no. 2, 385--403.
 
\bibitem[ACSS21]{ACSS21} F. Ambro, P. Cascini, V. V. Shokurov, and C. Spicer, \textit{Positivity of the moduli part}, arXiv:2111.00423.

\bibitem[Ara20]{Ara20} C. Araujo, Talk on \textit{Fano Foliations 3 - Classification of Fano foliations of large index}, \url{https://www.youtube.com/watch?v=WNVa_DdOdKs}. In ``Geometry and Dynamics of Foliations", CIRM ((Marseille, France), May 11th, 2020.



\bibitem[AD13]{AD13} C. Araujo and S. Druel, \textit{On Fano foliations}, Adv. Math. \textbf{238} (2013), 70--118.

\bibitem[AD14]{AD14}  C. Araujo and S. Druel, \textit{On codimension 1 del Pezzo foliations on varieties with
mild singularities}, Math. Ann. \textbf{360} (2014), no. 3--4, 769--798.

\bibitem[AD19]{AD19}  C. Araujo and S. Druel, \textit{Characterization of generic projective space bundles and algebraicity of foliations},  Comment. Math. Helv. \textbf{94} (2019), no. 4, 833--853.

\bibitem[ADK08]{ADK08}  C. Araujo, S. Druel, and S. Kovács, \textit{Cohomological characterizations of
projective spaces and hyperquadrics}, Invent. Math. \textbf{174} (2008), no. 2, 233--253.

\bibitem[Bir12]{Bir12} C. Birkar, \textit{Existence of log canonical flips and a special LMMP}, Publ. Math. IHÉS \textbf{115} (2012), 325--368.

\bibitem[Bir19]{Bir19} C. Birkar, \textit{Anti-pluricanonical systems on Fano varieties}, Ann. of Math. (2), \textbf{190} (2019), 345--463.
	

\bibitem[Bir21]{Bir21} C. Birkar, \textit{Singularities of linear systems and boundedness of Fano varieties}, Ann. of Math \textbf{193} (2021), no. 2, 374--405.


\bibitem[BCHM10]{BCHM10} C. Birkar, P. Cascini, C. D. Hacon, and J. M\textsuperscript{c}Kernan, \textit{Existence of minimal models for varieties of log general type}, J. Amer. Math. Soc. \textbf{23} (2010), no. 2, 405--468.


\bibitem[BZ16]{BZ16} C. Birkar and D.-Q. Zhang, \textit{Effectivity of Iitaka fibrations and pluricanonical systems of polarized pairs}, Publ. Math. IHÉS \textbf{123} (2016), 283--331.


\bibitem[Bru15]{Bru15} M. Brunella, \textit{Birational geometry of foliations}, IMPA Monographs \textbf{1} (2015), Springer, Cham.

\bibitem[CP19]{CP19} F. Campana and M. P\u{a}un, \textit{Foliations with positive slopes and birational stability of orbifold cotangent bundles}, Publ. Math. IHÉS \textbf{129} (2019), 1--49.

\bibitem[CHL$^+$24]{CHLMSSX24} P. Cascini, J. Han, J. Liu, F. Meng, C. Spicer, R. Svaldi, and L. Xie, \textit{Minimal model program for algebraically integrable adjoint foliated structures}, arXiv:2408.14258.

\bibitem[CM24]{CM24}
P. Chaudhuri and R. Mascharak, 
\textit{Log canonical Minimal Model Program for corank one foliations on threefolds}, arXiv:2410.05178 

\bibitem[CS20]{CS20} P. Cascini and C. Spicer, \textit{On the MMP for rank one foliations on threefolds}, arXiv:2012.11433.

\bibitem[CS21]{CS21} P.~Cascini and C. Spicer, \textit{MMP for co-rank one foliations on threefolds}, Invent. Math. \textbf{225} (2021), no. 2, 603--690.


\bibitem[CS24]{CS24} P. Cascini and C. Spicer, \textit{MMP for algebraically integrable foliations}, in \textit{Higher Dimensional Algebraic Geometry: A Volume in Honor of V. V. Shokurov}, Cambridge University Press (2025), 69--84.

\bibitem[CC25]{CC25} C.-W. Chang and Y.-A. Chen, \textit{Boundedness of toric foliations}, arXiv:2502.11080.


\bibitem[CHLX23]{CHLX23} G. Chen, J. Han, J. Liu, and L. Xie, \textit{Minimal model program for algebraically integrable foliations and generalized pairs}, arXiv:2309.15823. 

\bibitem[Che21]{Che21} Y.-A. Chen, \textit{Boundedness of minimal partial du Val resolutions of canonical surface foliations}, Math. Ann. \textbf{381} (2021), 557--573.

\bibitem[CLW24]{CLW24} Y. Chen, J. Liu, and Y. Wang, \textit{Flop between algebraically integrable foliations on potentially klt varieties}, arXiv:2410.05764.



\bibitem[dFKX17]{dFKX17} T. de Fernex, J. Koll\'ar, and C. Xu, \textit{The dual complex of singularities}, in \textit{Higher dimensional algebraic geometry: in honor of Professor Yujiro Kawamata’s sixtieth birthday}, Adv. Stud. Pure Math. \textbf{74} (2017), Math. Soc. Japan, Tokyo, 103--129. 

\bibitem[DPS94]{DPS94} J. P. Demailly, T. Peternell, and M. Schneider, \textit{Compact complex manifolds
with numerically eﬀective tangent bundles}, J. Algebraic Geom. \textbf{3} (1994), no. 2, 295--345.


\bibitem[Dru21]{Dru21} S. Druel, \textit{Codimension 1 foliations with numerically trivial canonical class on singular spaces}, Duke Math. J. \textbf{170} (2021), no. 1, 95--203.


\bibitem[HK10]{HK10} C. D. Hacon and S. J. Kov\'acs, \textit{Classification of higher dimensional algebraic varieties}, Oberwolfach Seminars \textbf{41} (2010), Birkh\"auser Verlag, Basel.


\bibitem[HL21]{HL21} C. D. Hacon and A. Langer, \textit{On birational boundedness of foliated surfaces}, J. Reine. Angew. Math. \textbf{770} (2021), 205--229.

\bibitem[HL23]{HL23} C. D. Hacon and J. Liu, \textit{Existence of flips for generalized lc pairs}, Camb. J. Math. \textbf{11} (2023), no. 4, 795--828.  

\bibitem[HMX18]{HMX18} C. D. Hacon, J. M\textsuperscript{c}Kernan, and C. Xu, \textit{Boundedness of moduli of varieties of general type}, J. Eur. Math. Soc. \textbf{20} (2018), no. 4, 865--901.

\bibitem[HX13]{HX13} C. D. Hacon and C. Xu, \textit{Existence of log canonical closures}, Invent. Math. \textbf{192} (2013), no. 1, 161--195.

\bibitem[HL22]{HL22} J. Han and Z. Li, \textit{Weak Zariski decompositions and log terminal models for generalized polarized pairs}, Math. Z. \textbf{302} (2022), 707--741.

\bibitem[HLS24]{HLS24} J. Han, J. Liu, and V. V. Shokurov, \textit{ACC for minimal log discrepancies of exceptional singularities}, Peking Math. J. (2024).

\bibitem[Hör14]{Hör14} A. Höring, \textit{Twisted cotangent sheaves and a Kobayashi-Ochiai theorem for foliations}, Ann. Inst. Fourier (Grenoble), \textbf{64} (2014), no. 6, 2464--2480.

\bibitem[JV23]{JV23} D. Jiao and P. Voegtli, \textit{Flop connections between minimal models for corank 1 foliations over threefolds}, arXiv:2305.19728.

\bibitem[Kaw08]{Kaw08} Y. Kawamata, \textit{Flops connect minimal models}, Publ. Res. Inst. Math. Sci. \textbf{44} (2008), no. 2, 419--423.

\bibitem[KM98]{KM98} J. Koll\'{a}r and S. Mori, \textit{Birational geometry of algebraic varieties}, Cambridge Tracts in Math. \textbf{134} (1998), Cambridge Univ. Press.

\bibitem[Liu23]{Liu23} J. Liu, \textit{Fano foliations with small algebraic ranks}, Adv. Math. \textbf{423} (2023), 109038.

\bibitem[LLM23]{LLM23} J. Liu, Y. Luo, and F. Meng, \textit{On global ACC for foliated threefolds},  Trans. Amer. Math. Soc. \textbf{376} (2023), no. 12, 8939--8972.

\bibitem[LMX24]{LMX24} J. Liu, F. Meng, and L. Xie, \textit{Minimal model program for algebraically integrable foliations on klt varieties}, arXiv:2404.01559.

\bibitem[LX23]{LX23} J. Liu and L. Xie, \textit{Relative Nakayama-Zariski decomposition and minimal models of generalized pairs}, Peking Math. J. (2023).

\bibitem[LW24]{LW24} J. Lu and X. Wang, \textit{On the 1-adjoint canonical divisor of a foliation}, Manuscripta. Math. \textbf{175} (2024), 739--752.

\bibitem[LWX25]{LWX25} J. Lu, X. Wu, and S. Xu, \textit{On adjoint divisors for foliated surfaces}, arXiv:2501.00470.

\bibitem[Lü24]{Lü24} X. Lü, \textit{Unboundedness of foliated varieties}, Int. Math. J. (2024), Article No. 2550003.

\bibitem[Mas24]{Mas24} R. Mascharak, \textit{On the log Sarkisov program for foliations on projective 3-folds}, arXiv:2406.09434.

\bibitem[McQ08]{McQ08} M. McQuillan, \textit{Canonical models of foliations}, Pure Appl. Math. Q. \textbf{4} (2008), no. 3, Special Issue: In honor of Fedor Bogomolov, Part 2, 877--1012.

\bibitem[McQ24]{McQ24} M. McQuillan, \textit{Semi-Stable Reduction of Foliations} (2024), In: Arithmetic and Algebraic Geometry (Y. Tschinkel eds), Simons Symposia. Springer, Cham.

\bibitem[MZ23]{MZ23} F. Meng and Z. Zhuang, \textit{MMP for locally stable families and wall crossing for moduli of stable pairs}, arXiv:2311.01319.

\bibitem[Pas24]{Pas24} A. Passantino, \textit{Numerical conditions for the boundedness of foliated surfaces}, arXiv:2412.05986.

\bibitem[PS19]{PS19} J. V. Pereira and R. Svaldi, \textit{Effective algebraic integration in bounded genus}, Algebr. Geom. \textbf{6} (2019), no. 4, 454--485.

\bibitem[Spi20]{Spi20} C. Spicer, \textit{Higher dimensional foliated Mori theory}, Compos. Math. \textbf{156} (2020), no. 1, 1--38.

\bibitem[SS22]{SS22} C. Spicer and R. Svaldi, \textit{Local and global applications of the Minimal Model Program for co-rank 1 foliations on threefolds}, J. Eur. Math. Soc. \textbf{24} (2022), no. 11, 3969--4025.


\bibitem[SS23]{SS23} C. Spicer and R. Svaldi, \textit{Effective generation for foliated surfaces: Results and applications}, J. Reine Angew. Math. \textbf{795} (2023), 45--84.

\bibitem[Wah83]{Wah83} J. M. Wahl, \textit{A cohomological characterization of $\mathbb P^n$}, Invent. Math. \textbf{72} (1983), no. 2, 315--322.

\end{thebibliography}
\end{document}